\documentclass[11pt,reqno]{amsart}
\usepackage{mathrsfs}
\usepackage[breaklinks]{hyperref}
\usepackage{url}
\usepackage{amssymb}

\newcommand{\tf}{\tfrac}

  \renewcommand{\a}{\alpha}
\renewcommand{\b}{\beta}

\renewcommand{\d}{{\delta}}
\newcommand{\g}{\gamma}
\newcommand{\G}{\Gamma}
\renewcommand{\l}{\lambda}

\renewcommand{\(}{\left\(}
\renewcommand{\)}{\right\)}
\renewcommand{\[}{\left\[}
\renewcommand{\]}{\right\]}
\let\dotlessi=\i
\renewcommand{\i}{\infty}
\numberwithin{equation}{section}
 \theoremstyle{plain}
\newtheorem{theorem}{Theorem}[section]
\newtheorem{lemma}[theorem]{Lemma}
\newtheorem{corollary}[theorem]{Corollary}

\newtheorem{definition}[theorem]{Definition}

\newtheorem{remark}[]{Remark}

   \makeatletter
\def\proof{\@ifnextchar[{\@oproof}{\@nproof}}
\def\@oproof[#1][#2]{\trivlist\item[\hskip\labelsep\textit{#2 Proof of\
#1.}~]\ignorespaces}
\def\@nproof{\trivlist\item[\hskip\labelsep\textit{Proof.}~]\ignorespaces}

\makeatother

\makeatletter
\def\@tocline#1#2#3#4#5#6#7{\relax
  \ifnum #1>\c@tocdepth 
  \else
    \par \addpenalty\@secpenalty\addvspace{#2}%
    \begingroup \hyphenpenalty\@M
    \@ifempty{#4}{%
      \@tempdima\csname r@tocindent\number#1\endcsname\relax
    }{%
      \@tempdima#4\relax
    }%
    \parindent\z@ \leftskip#3\relax \advance\leftskip\@tempdima\relax
    \rightskip\@pnumwidth plus4em \parfillskip-\@pnumwidth
    #5\leavevmode\hskip-\@tempdima
      \ifcase #1
       \or\or \hskip 1em \or \hskip 2em \else \hskip 3em \fi%
      #6\nobreak\relax
    \dotfill\hbox to\@pnumwidth{\@tocpagenum{#7}}\par
    \nobreak
    \endgroup
  \fi}
\makeatother

\numberwithin{theorem}{section}
\numberwithin{theorem}{subsection}
\numberwithin{equation}{section}
\numberwithin{equation}{subsection}

\begin{document}
\title[Superimposing theta structure on a generalized modular relation]{Superimposing theta structure on a generalized modular relation}
\author{Atul Dixit and Rahul Kumar}
\address{Discipline of Mathematics, Indian Institute of Technology Gandhinagar, Palaj, Gandhinagar 382355, Gujarat, India}\email{adixit@iitgn.ac.in; rahul.kumr@iitgn.ac.in}

\begin{abstract}
A generalized modular relation of the form $F(z, w, \a)=F(z, iw,\b)$, where $\a\b=1$ and $i=\sqrt{-1}$, is obtained in the course of evaluating an integral involving the Riemann $\Xi$-function. It is a two-variable generalization of a transformation found on page $220$ of Ramanujan's Lost Notebook. This modular relation involves a surprising generalization of the Hurwitz zeta function $\zeta(s, a)$, which we denote by $\zeta_w(s, a)$. While $\zeta_w(s, 1)$ is essentially a product of confluent hypergeometric function and the Riemann zeta function, $\zeta_w(s, a)$ for $0<a<1$ is an interesting new special function. We show that $\zeta_w(s, a)$ satisfies a beautiful theory generalizing that of $\zeta(s, a)$ albeit the properties of $\zeta_w(s, a)$ are much harder to derive than those of $\zeta(s, a)$. In particular, it is shown that for $0<a<1$ and $w\in\mathbb{C}$, $\zeta_w(s, a)$ can be analytically continued to Re$(s)>-1$ except for a simple pole at $s=1$. This is done by obtaining a generalization of Hermite's formula in the context of $\zeta_w(s, a)$. The theory of functions reciprocal in the kernel $\sin(\pi z) J_{2 z}(2 \sqrt{xt}) -\cos(\pi z) L_{2 z}(2 \sqrt{xt})$, where $L_{z}(x)=-\frac{2}{\pi}K_{z}(x)-Y_{z}(x)$ and $J_{z}(x), Y_{z}(x)$ and $K_{z}(x)$ are the Bessel functions, is worked out. So is the theory of a new generalization of $K_{z}(x)$, namely, ${}_1K_{z,w}(x)$. Both these theories as well as that of $\zeta_w(s, a)$ are essential to obtain the generalized modular relation.
\end{abstract}
\maketitle
\vspace{-0.8cm}
\tableofcontents

\section{Introduction}\label{intro}
\setcounter{subsection}{1}
Regarding the results in Ramanujan's only published paper \cite{riemann} fully devoted to the study of the Riemann zeta function $\zeta(s)$, Hardy says \cite{hardyw}:\\

\textit{``It is difficult at present to estimate the importance of these results. The unsolved problems concerning the zeros of $\zeta(s)$ or of $\Xi(t)$ are among the most obscure and difficult in the whole range of Pure Mathematics. Any new formulae involving $\zeta(s)$ or $\Xi(t)$ are of very great interest, because of the possibility that they may throw light on some of these outstanding questions. It is, as I have shown in a short note attached to Mr. Ramanujan's paper, certainly possible to apply his formulae in this direction; but the results which can be deduced from them do not at present go beyond those obtained already by Mr. Littlewood and myself in other ways. But I should not be at all surprised if still more important applications were to be made of Mr. Ramanujan's formulae in the future".}\\

One of the goals of this paper is to not only give non-trivial applications of Ramanujan's formulas but also to generalize an integral involving the Riemann $\Xi$-function (defined in \eqref{xi} below) occurring in one of them. In the course of doing so, we are led to an interesting generalization of the Hurwitz zeta function $\zeta(s, a)$ whose theory is developed here. As will be seen, it is difficult to a priori conceive this generalized Hurwitz zeta function, however, we were naturally led to it through our quest to obtain a generalized modular relation which is a two-parameter generalization of yet another formula of Ramanujan (\eqref{ramfor} below). Developing the theory of this generalized Hurwitz zeta function turned out to be an extremely interesting and a fruitful task, for, in this endeavor, we stumbled upon a new generalization of the modified Bessel function of the second kind. The rich theory of this generalized modified Bessel function is also developed here.

The first of the two results in Hardy's aforementioned quote can be stated in an equivalent form as \cite[Equation (13)]{riemann}
\begin{align}\label{mrram}
&\alpha^{\frac{1}{2}}-4\pi\alpha^{\frac{3}{2}}\int_{0}^{\infty}\frac{xe^{-\pi\alpha^2 x^2}}{e^{2\pi x}-1}\, dx=\beta^{\frac{1}{2}}-4\pi\beta^{\frac{3}{2}}\int_{0}^{\infty}\frac{xe^{-\pi\beta^2 x^2}}{e^{2\pi x}-1}\, dx\nonumber\\
&=\frac{1}{4\pi^{3/2}}\int_{0}^{\infty}\Gamma\left(\frac{-1+it}{4}\right)\Gamma\left(\frac{-1-it}{4}\right)\Xi\left(\frac{t}{2}\right)\cos \left(\frac{1}{2}t\log\alpha\right)\, dt,
\end{align}
where $\Xi(t)$ is a function of Riemann defined by \cite[p.~16]{titch}
\begin{align}\label{xi}
\Xi(t)&:=\xi\left(\tfrac{1}{2}+it\right),
\end{align}
with 
\begin{equation}\label{xii}
\xi(s):=\frac{1}{2}s(s-1)\pi^{-\frac{s}{2}}\Gamma\left(\frac{s}{2}\right)\zeta(s),
\end{equation}
and $\Gamma(s)$ being the Euler Gamma function. Here, and throughout the sequel, unless specified otherwise, $\a$ and $\b$ will always denote two positive numbers such that $\a\b=1$. 

Even though Ramanujan does not explicitly mention, the first equality in \eqref{mrram} is a transformation satisfied by an \emph{integral analogue of partial theta function}, namely, $\displaystyle\int_{0}^{\infty}\frac{xe^{-\pi\alpha^2 x^2}}{e^{2\pi x}-1}\, dx$. The reason for this nomenclature is now explained. The transformation for the Jacobi theta function can be rephrased for $\a\b=1$, Re$(\alpha^2)>0$ and Re$(\beta^2)>0$ in the form \cite[p.~43, Entry 27(i)]{berndt0}
\begin{align}\label{thetaab}
\sqrt{\alpha}\bigg(\frac{1}{2\alpha}-\sum_{n=1}^{\infty}e^{-\pi\alpha^2n^2}\bigg)=\sqrt{\beta}\bigg(\frac{1}{2\beta}-\sum_{n=1}^{\infty}e^{-\pi\beta^2n^2}\bigg),
\end{align}
Since the variable of integration in $\displaystyle\int_{0}^{\infty}\frac{xe^{-\pi\alpha^2 x^2}}{e^{2\pi x}-1}\, dx$ runs only from $0$ to $\infty$ similar to the summation indices in the partial theta functions in \eqref{thetaab} running only from $1$ to $\infty$, the integral is called an `integral analogue of partial theta function' in \cite{drz5}. 

It is well-known that \eqref{thetaab} is equivalent to the functional equation of $\zeta(s)$ given by  \cite[p.~22, eqn. (2.6.4)]{titch}
\begin{equation}\label{zetafe}
\pi^{-\frac{s}{2}}\Gamma\left(\frac{s}{2}\right)\zeta(s)=\pi^{-\frac{(1-s)}{2}}\Gamma\left(\frac{1-s}{2}\right)\zeta(1-s),
\end{equation}
Hardy \cite{Har} showed that both sides of \eqref{thetaab} also equal
\begin{equation}\label{thetaabint}
\frac{2}{\pi}\int_{0}^{\infty}\frac{\Xi(t/2)}{1+t^2}\cos\bigg(\frac{1}{2}t\log \a\bigg)\, dt,
\end{equation}
and used this fact to prove that there are infinitely many zeros of the Riemann zeta function $\zeta(s)$ on the critical line. Regarding the integral involving $\Xi$-function in \eqref{mrram}, Hardy \cite{ghh} says:

`\emph{The integral has properties similar to those of the integral by means of which I proved recently that $\zeta(s)$ has an infinity of zeros on the line $\sigma=1/2$, and may be used for the same purpose.}''

Note that the first equality in \eqref{mrram} is of the form $F(\a)=F(\b)$ for $\a\b=1$.

The second integral that Ramanujan studied in \cite[Section 5]{riemann} is 
\begin{equation}\label{rint}
\int_{0}^{\infty}\Gamma\left(\frac{z-1+it}{4}\right)\Gamma\left(\frac{z-1-it}{4}\right)
\Xi\left(\frac{t+iz}{2}\right)\Xi\left(\frac{t-iz}{2}\right)\frac{\cos(nt)}{(z+1)^2+t^2}\, dt,
\end{equation}
where $n\in\mathbb{R}$, and represented it in terms of equivalent integrals \cite[Equations (19), (20), (21)]{riemann}\footnote{See \cite[Theorem 1.2]{dixit} for the corrected versions of equations (19) and (21) of \cite{riemann}.}. He does not specify any connection whatsoever between the above integral and the corresponding one in \eqref{mrram}. However, it was shown in \cite[Section 6]{drz5} that they are related through \emph{``squaring of the functional equation''} of $\zeta(s)$.

Regarding the special case $z=0$ of the above integral, Hardy \cite{ghh} says, \textit{``the properties of this integral resemble those of one which Mr. Littlewood and I have used, in a paper to be published shortly in Acta Mathematica to prove that\footnote{Note that there is a typo in this formula in that $\pi$ should not be present.}}
\begin{equation*}
\int_{-T}^{T}\left|\zeta\left(\frac{1}{2}+ti\right)\right|^2\, dt \sim
\frac{2}{\pi} T\log T\hspace{3mm}(T\to\infty)\textup{''}.
\end{equation*}
In his Lost Notebook \cite[p.~220]{lnb}, Ramanujan found an exquisitely beautiful result linked to this special case $z=0$ of the integral in \eqref{rint}, namely, if $\psi(s)=\Gamma'(s)/\Gamma(s)$, $\g$ denotes Euler's constant and
\begin{align*}
\mathfrak{F}(\a):=\sqrt{\alpha}\bigg(\sum_{m=1}^{\infty}\left(-\psi(m\a)+\log(m\a)-\frac{1}{2m\a}\right)-\frac{(\gamma-\log(2\pi\alpha))}{2\alpha}\bigg),
\end{align*}
then
\begin{align}\label{ramfor}
\mathfrak{F}(\a)=\mathfrak{F}(\b)=\frac{1}{\pi^{3/2}}\int_0^{\infty}\Gamma\left(\frac{-1+it}{4}\right)\Gamma\left(\frac{-1-it}{4}\right)
\Xi^{2}\left(\frac{t}{2}\right)\frac{\cos\left(\tfrac{1}{2}t\log\alpha\right)}{1+t^2}\, dt.
\end{align}
The reader is referred to \cite{bcbad} and \cite{series} for proofs of this result. The extra variable $z$ in \eqref{rint}, however, suggests that the above result can be generalized. Indeed, the following generalization of \eqref{ramfor} was obtained in \cite[Theorem 1.4]{dixit} (see also \cite[Theorem 1.5]{transf}). Let 
\begin{equation*}
\mathcal{F}(z, \a):=\a^{\frac{z+1}{2}}\left(\sum_{m=1}^{\infty}\left(\zeta(z+1,m\a)-\frac{(m\a)^{-z}}{z}-\frac{1}{2}(m\a)^{-z-1}\right)-\frac{\zeta(z+1)}{2\alpha^{z+1}}-\frac{\zeta(z)}{\alpha z}\right),
\end{equation*}
where $\zeta(s, x)$ denotes the Hurwitz zeta function, defined for Re$(s)>1$ and $x\in\mathbb{C}\backslash(\mathbb{R}^{-}\cup\{0\})$, by
\begin{equation*}
\zeta(s, x):=\sum_{n=0}^{\infty}\frac{1}{(n+x)^{s}}.
\end{equation*}
Then for $-1<\textup{Re}(z)<1$,
\begin{align}\label{hurwitzeqn}
\mathcal{F}(z, \a)=\mathcal{F}(z, \b)&=\frac{2^z\pi^{\frac{z-3}{2}}}{\Gamma(z+1)}\int_{0}^{\infty}\Gamma\left(\frac{z-1+it}{4}\right)\Gamma\left(\frac{z-1-it}{4}\right)
\Xi\left(\frac{t+iz}{2}\right)\nonumber\\
&\qquad\qquad\qquad\times\Xi\left(\frac{t-iz}{2}\right)\frac{\cos\left( \tf{1}{2}t\log\a\right)}{(z+1)^2+t^2}\, dt.
\end{align}

Before commencing the main topic which led to this work, we turn our attention to the general theta transformation, also known as Jacobi's imaginary transformation \cite[p.~475]{ww}. This transformation, along with the corresponding integral involving $\Xi(t)$ which was obtained in \cite[Theorem 1.2]{dixthet}, together generalize \eqref{thetaab} and \eqref{thetaabint}. For $\a\b=1$, Re$(\a^2)>0$, Re$(\b^2)>0$, and $w\in\mathbb{C}$, this general theta transformation is given by
\begin{align}\label{genthetatr}
\sqrt{\alpha}\bigg(\frac{e^{-\frac{w^2}{8}}}{2\alpha}-e^{\frac{w^2}{8}}\sum_{n=1}^{\infty}e^{-\pi\alpha^2n^2}\cos(\sqrt{\pi}\alpha nw)\bigg)
&=\sqrt{\beta}\bigg(\frac{e^{\frac{w^2}{8}}}{2\beta}-e^{-\frac{w^2}{8}}\sum_{n=1}^{\infty}e^{-\pi\beta^2n^2}\cosh(\sqrt{\pi}\beta nw)\bigg)\nonumber\\
&=\frac{1}{\pi}\int_{0}^{\infty}\frac{\Xi(t/2)}{1+t^2}\nabla\left(\alpha,w,\frac{1+it}{2}\right)\, dt,
\end{align}
where
\begin{align*}
\nabla(x,w,s)&:=\rho(x,w,s)+\rho(x,w,1-s),\nonumber\\
\rho(x,w,s)&:=x^{\frac{1}{2}-s}e^{-\frac{w^2}{8}}{}_1F_{1}\left(\frac{1-s}{2};\frac{1}{2};\frac{w^2}{4}\right),
\end{align*}
with ${}_1F_{1}(a;c;z):=\sum_{n=0}^{\infty}\frac{(a)_{n}z^{n}}{(c)_{n}n!}$ being the confluent hypergeometric function and $(a)_n:=a(a+1)\cdots(a+n-1)$. Note that the first equality in \eqref{genthetatr} is of the form $F(w,\a)=F(iw,\b)$, where $\a\b=1$ and $i=\sqrt{-1}$. 

It would be good at this juncture to explain not only what we mean by a modular relation but also by the one with a theta structure. The theta transformation \eqref{thetaab} is a modular transformation since the Jacobi theta function is a weight $1/2$ modular form on $\G_{0}(4)$ twisted by $\chi_{-1}$, the Dirichlet character modulo $4$ defined by $\chi_{-1}(n):=\left(\frac{-1}{n}\right)=(-1)^{(n-1)/2}$. By a\emph{ modular relation} or a modular-type transformation, however, we mean a transformation of the form $F(-1/z)=F(z)$, where $z\in\mathbb{H}$ (the upper-half plane). The $F$ in such a relation may not be governed by $z\to z+1$. A \emph{generalized modular relation} or a \emph{generalized modular-type transformation} is the one which involves additional variable(s) besides $z$.


As explained in \cite{mtt}, any generalized modular relation may be recast in an equivalent form governed by the relation $\a\to\b$, where $\a\b=1$ and Re$(\a)>0$, Re$(\b)>0$. Thus, while the transformations in \eqref{thetaab} and \eqref{genthetatr} are modular in that they come from modular forms, the ones in \eqref{ramfor} and \eqref{hurwitzeqn} are only modular relations.
Generalized modular relations have many applications in analytic number theory, special functions and asymptotics. See \cite{mtt} for a survey on them as well as for the necessary references.

A \emph{modular relation with a theta structure} is the one in which the associated function, say $F(w, \a)$, becomes invariant only when $\a$ and $w$ are \emph{simultaneously} replaced by $\b$ and $iw$ respectively. The prototypical example of this is \eqref{genthetatr}. Several other examples can be found in \cite{dixthet}.

The present work arose from answering the question - \emph{can one superimpose the theta structure on \textup{\eqref{hurwitzeqn}}}, that is, can one obtain a generalized modular relation of the form $F(z, w, \a)=F(z,iw,\b)$, where $\a\b=1$, which reduces to \eqref{hurwitzeqn} when $w=0$? One reason why one may be interested in doing so is because, that would pave a way towards obtaining a new Jacobi form. Loosely speaking, a Jacobi form is a function $f:\mathbb{H}\times\mathbb{C}\to\mathbb{C}$, satisfying certain growth conditions, and which is modular in the first variable and elliptic in the second. Note that the Jacobi theta function in \eqref{genthetatr} is a prototypical example of a Jacobi form.

One of the goals of this paper, indeed, is to obtain a relation of the form $F(z, w, \a)=F(z,iw,\b)$ generalizing \eqref{hurwitzeqn}. This transformation is given in Theorem \ref{genramhureq}. It is then used to evaluate a generalization of the integral in \eqref{rint} (see Theorem \ref{xiintgenramhurthm}). 

The motivation for this work arose from the fact that answering a similar question in \cite{dkmt} in the context of the Ramanujan-Guinand formula \eqref{mainagain3} led us to its interesting generalization of the form $F(z, w, \a)=F(z,iw,\b)$ which is given in \eqref{genrgeqeqn}. The importance of having such a generalization is clear from the fact that the Ramanujan-Guinand formula itself is equivalent to the functional equation of the non-holomorphic Eisenstein series on $\textup{SL}_{2}(\mathbb{Z})$ as can be seen, for example, from \cite{cohen}. The framework in which this generalization lies is explained in detail in Section \ref{mr}.

The modular relation of the form $F(z, w, \a)=F(z,iw,\b)$ generalizing \eqref{hurwitzeqn} which we obtain in Theorem \ref{genramhureq} contains a special feature - it involves a new generalization of the Hurwitz zeta function $\zeta(s, a)$. The second goal of this paper is to develop the theory of this \emph{generalized Hurwitz zeta function}.

Let $\mathfrak{B}:=\{\xi:\textup{Re}(\xi)=1, \textup{Im}(\xi)\neq 0\}$. We define the generalized Hurwitz zeta function for $w\in\mathbb{C}\backslash\{0\}$, Re$(s)>1$ and $a\in\mathbb{C}\backslash\mathfrak{B}$ by
\begin{align}\label{new zeta}
\zeta_w(s, a)&:=\frac{4}{w^2\sqrt{\pi}\G\left(\frac{s+1}{2}\right)}\sum_{n=1}^{\infty}\int_{0}^{\infty}\int_{0}^{\infty}\frac{(uv)^{s-1}e^{-(u^2+v^2)}\sin(wv)\sinh(wu)}{\left(n^2u^2+(a-1)^2v^2\right)^{s/2}}\, dudv.
\end{align}
Even though the definition may look intimidating at first glance, we hope to convince the reader that it is indeed an interesting new object satisfying a beautiful theory! This is owing to the fact that it naturally arises while generalizing \eqref{hurwitzeqn}, and is not concocted artificially. The first conspicuous appearance of this generalized Hurwitz zeta function is seen while evaluating a generalization of \eqref{hurwitzeqn} as can be observed in \eqref{doub3}. Another interesting fact is that, even though the double integral in the summand cannot be evaluated in closed-form in general, we are still able to develop the theory of $\zeta_w(s, a)$.

The series in \eqref{new zeta} converges absolutely for Re$(s)>1$. This is easy to see since the summand is $O_{w,a}\left(n^{-\textup{Re}(s)}\right)$ as $n\to\infty$. Note that we require $a\notin\mathfrak{B}$, for, otherwise, the branch of logarithm in $(n^2u^2+(a-1)^2v^2)^{s/2}$ becomes multi-valued. In the case of the Hurwitz zeta function $\zeta(s, a)$, one has to omit all non-positive real numbers to avoid it being multi-valued.

Note that $w=0$ is actually a removable singularity of $\zeta_w(s, a)$ since the power series expansion of $\displaystyle\frac{\sin(wv)\sinh(wu)}{w^2}$ around $w=0$ starts with $1$. Hence for Re$(s)>1$, by $\zeta_0(s, a)$, we mean
\begin{equation}\label{zeta0sa}
\zeta_0(s, a)=\lim_{w\to 0}\zeta_w(s, a).
\end{equation}
Indeed, from the fact that the summand in the series definition of $\zeta_w(s, a)$ in \eqref{new zeta} is $O_{w,a}\left(n^{-\textup{Re}(s)}\right)$ as $n\to\infty$, we can interchange the order of $\lim_{w\to 0}$ and summation. Also, the exponential decay of the integrand of the double integral as $u, v\to\infty$ allows us to interchange the order of $\lim_{w\to 0}$ and integration.

As will be seen in Theorem \ref{hurwitzsc} below, $\zeta_0(s, a)$ is the familiar Hurwitz zeta function. Our first result is
\begin{theorem}\label{zwsares1}
Let $w\in\mathbb{C}$ and let $a\in\mathbb{R}$. Then the generalized Hurwitz zeta function $\zeta_w(s, a)$ is analytic in \textup{Re}$(s)>1$.
\end{theorem}



The generalized Hurwitz zeta function $\zeta_w(s,a)$ is the first-of-its-kind generalization of $\zeta(s, a)$ which, for Re$(s)>1$, is symmetric about the line $a=1$, that is, 
\begin{align}\label{zwfea}
\zeta_w(s,2-a)=\zeta_w(s,a).
\end{align}
Moreover, it is an even function of $w$.
 
For $a=1$ and Re$(s)>1$, $\zeta_w(s,a)$ reduces (essentially) to the Riemann zeta function case since
\begin{align}\label{a=1caserz}
\zeta_w(s,1)
&=\frac{4}{w^2\sqrt{\pi}\Gamma\left(\frac{s+1}{2}\right)}\sum_{n=1}^\infty\frac{1}{n^s} \int_0^\infty\int_0^\infty\frac{v^{s-1}}{u}e^{-(u^2+v^2)}\sin(wv)\sinh(wu)dudv \nonumber\\
&=\frac{\sqrt{\pi}}{w}\ \mathrm{erfi}\left(\frac{w}{2}\right){}_1F_1\left(\frac{1+s}{2};\frac{3}{2};-\frac{w^2}{4}\right)\zeta(s),
\end{align} 
as shown in Section \ref{seczws1}. Here $\textup{erfi}(w)$ is the imaginary error function defined by
\begin{align}
\textup{erfi}(w)&:=\frac{2}{\sqrt{\pi}}\int_{0}^{w}e^{t^2}\, dt=\frac{2w}{\sqrt{\pi}}e^{w^2}{}_1F_{1}\left(1;\frac{3}{2};-w^2\right).\label{erferfi}
\end{align}
We also note the definition of the error function $\textup{erf}(w)$:
\begin{align}
\textup{erf}(w)&:=\frac{2}{\sqrt{\pi}}\int_{0}^{w}e^{-t^2}\, dt=\frac{2w}{\sqrt{\pi}}e^{-w^2}{}_1F_{1}\left(1;\frac{3}{2};w^2\right).\label{erf}
\end{align}
From \eqref{a=1caserz}, 
\begin{equation*}
\lim_{w\to 0}\zeta_w(s, 1)=\zeta(s).
\end{equation*}
The analytic properties of $\zeta_w(s, 1)$ are studied in Section \ref{seczws1}. However, the more important case $0<a<1$ is dealt in the remainder of that section. In Section \ref{nzf}, the following result is established too.
\begin{theorem}\label{hurwitzsc}
For \textup{Re}$(s)>1$,
\begin{equation}\label{hurwitzsceqn}
\zeta_0(s, a)=\begin{cases}
\zeta(s,a) & \text{if\ $\mathrm{Re}(a)>1$}, \\
   \zeta(s, 2-a) & \text{if\ $\mathrm{Re}(a)<1$}. 
\end{cases}
\end{equation}
\end{theorem}
\noindent
The theory of any zeta function is fruitless unless one analytically continues the function beyond Re$(s)>1$. We show in this paper that $\zeta_w(s, a)$ admits analytic continuation to $\textup{Re}(s)>-1, s\neq1$, as can be seen from the following theorem. Before stating it, however, we recall the definition of Bessel function of the first kind be \cite[p.~40]{watson-1966a}:
\begin{align}\label{sumbesselj}
	J_{\nu}(s):=\sum_{m=0}^{\infty}\frac{(-1)^m(s/2)^{2m+\nu}}{m!\Gamma(m+1+\nu)}, \quad |s|<\infty.
	\end{align}
Also, we will be working only with real $a$ from now onwards, and that too such that $0<a<1$.
\begin{theorem}\label{lse}
Let $0<a<1$ and $w\in\mathbb{C}$. The generalized Hurwitz zeta function $\zeta_w(s, a)$ can be analytically continued to $\textup{Re}(s)>-1$ except for a simple pole at $s=1$. The residue at this pole is
\begin{equation}\label{lse1}
e^{\frac{w^2}{4}}{}_1F_{1}^{2}\left(1;\frac{3}{2};-\frac{w^2}{4}\right)=\frac{\pi}{w^2}e^{-\frac{w^2}{4}}\textup{erfi}^{2}\left(\frac{w}{2}\right),
\end{equation}
where $\textup{erfi}(w)$ is defined in \eqref{erferfi}. Moreover near $s=1$, we have
\begin{align}\label{luslus}
\zeta_w(s, a+1)=\frac{e^{\frac{w^2}{4}}{}_1F_{1}^{2}\left(1;\frac{3}{2};-\frac{w^2}{4}\right)}{s-1}-\psi_w(a+1) +O_{w, a}(|s-1|),
\end{align}
where $\psi_w(a)$ is a new generalization of the digamma function $\psi(a)$ defined by
\begin{align}\label{new psi function}
\psi_w(a):=\frac{4}{w^2\sqrt{\pi}}\int_0^\infty\int_0^\infty\int_0^\infty \frac{e^{-(u^2+v^2+x)}}{u}\sin(wv)\sinh(wu)\left(\frac{1}{x}-\frac{J_{0}\left((a-1)\frac{vx}{u}\right)}{1-e^{-x}}\right)dxdudv.
\end{align}	
\end{theorem}
The analytic continuation of $\zeta_w(s, a)$ to Re$(s)>-1, s\neq1$, is not straightforward as in the case of $\zeta(s, a)$. For example, Hermite's formula for $\zeta(s, a)$ given for $s\in\mathbb{C}\backslash\{1\}$ and Re$(a)>0$ by \cite[p.~609, Formula \textbf{25.11.27}]{nist}
\begin{align}\label{hermiteor}
\zeta(s,a)=\frac{a^{-s}}{2}+\frac{a^{1-s}}{s-1}+2\int_0^\infty \frac{\sin\left(s\tan^{-1}\left(\frac{y}{a}\right)\right)\, dy}{(e^{2\pi y}-1)(y^2+a^2)^{\frac{s}{2}}}
\end{align}
gives analytic continuation of $\zeta(s, a)$ in $\mathbb{C}\backslash\{1\}$ at once since the integral on the right-hand side is an entire function of $s$. However, as will be shown in Section \ref{lseproof}, the generalization of the above formula in the setting of $\zeta_w(s, a)$ which we state in the theorem below gives analytic continuation only in Re$(s)>-1, s\neq1$ owing to the fact that the triple integral in the theorem is analytic only in $-1<\textup{Re}(s)<2$ (see Lemma \ref{tia}).

The aforementioned generalization of Hermite's formula is contained in the following theorem which is initially proved for $1<\mathrm{Re}(s)<2$.
\begin{theorem}\label{hermite}
Let $A_w(z)$ be defined by
\begin{align}\label{ab}
A_w(z)&:=\frac{\sqrt{\pi}}{w}\textup{erf}\left(\frac{w}{2}\right){}_1F_{1}\left(1+\frac{z}{2};\frac{3}{2};\frac{w^2}{4}\right)=e^{-\frac{w^2}{4}}{}_1F_{1}\left(1;\frac{3}{2};\frac{w^2}{4}\right){}_1F_{1}\left(1+\frac{z}{2};\frac{3}{2};\frac{w^2}{4}\right).
\end{align}
Then for $a>0$ and $1<\mathrm{Re}(s)<2$,
\begin{align}\label{hermiteeqn}
\zeta_w(s,a+1)&=-\frac{a^{-s}}{2}A_w(s-1)+\frac{a^{1-s}}{s-1}A_{iw}(1-s)\nonumber\\
&\quad+\frac{2^{s+2}a^{1-s}}{w^2\Gamma\left(1-\frac{s}{2}\right)\Gamma(s)}\int_0^\infty\int_1^\infty\int_{0}^\infty \frac{u^{s-2}e^{-(u^2+v^2)}\sin(wv)\sinh(wu)}{(e^{\frac{2\pi avt}{u}}-1)(t^2-1)^{\frac{s}{2}}}\ dvdtdu.
\end{align}
\end{theorem}
In Lemma \ref{tia}, we show that the triple integral on the right-hand side of \eqref{hermiteeqn} is analytic in $-1<\textup{Re}(s)<2$. Since the first expression on the right-hand side of \eqref{hermiteeqn} is entire in $s$, and the second is analytic in $\mathbb{C}\backslash\{1\}$, we see that \eqref{hermiteeqn} and Lemma \ref{tia} give analytic continuation of $\zeta_w(s, a)$ in $-1<\textup{Re}(s)\leq1, s\neq 1$ upon \underline{\emph{defining}} $\zeta_w(s, a+1)$ to be the right-hand side of \eqref{hermiteeqn} in this region.

\begin{remark}\label{ext}
It is clear from Theorem \textup{\ref{lse}} and the discussion following it that the first equality in Theorem \textup{\ref{hurwitzsc}} now holds for $s$ such that $\textup{Re}(s)>-1$ and $1<a<2$, that is,  
\begin{equation}\label{hurwitzscall}
\zeta_{0}(s,a)=\zeta(s, a)
\end{equation}
since the right-hand side of \eqref{hermiteeqn} for $w=0$ reduces to that of \eqref{hermiteeqn} as can be seen from Lemma \textup{\ref{3in1}} below.
\end{remark}
\begin{remark}
The reader may have noticed that in passing from Theorems \textup{\ref{zwsares1}} and \textup{\ref{hurwitzsc}} to Theorems \textup{\ref{lse}} and \textup{\ref{hermite}}, we have switched from stating results for $\zeta_w(s, a)$ to doing the same for $\zeta_w(s, a+1)$. The reason for doing so is now explained. Firstly, in view of \eqref{hurwitzscall} and to recover the known results for $\zeta(s, a)$ from ours, if we let $0<a<1$, then we are forced to work with $\zeta_w(s, a+1)$. Secondly, we are unable to find a relation between $\zeta_w(s, a+1)$ and $\zeta_w(s, a)$ that would contain the relation \cite[p.~607, Formula \textbf{25.11.3}]{nist}
\begin{equation}\label{hurrel}
\zeta(s, a)=\zeta(s,a+1)+a^{-s}\hspace{8mm}(s\in\mathbb{C},\hspace{1mm}0<a\leq 1)
\end{equation}
for $\textup{Re}(s)>-1$ as its special case. However, stating the results for $\zeta_w(s, a+1)$ causes no loss of generality, and hence in the sequel too we state most of the results for $\zeta_w(s, a+1)$ with $0<a<1$ rather than for $\zeta_w(s,a)$ with $0<a<1$. 
\end{remark}

Since $J_{0}(-x)=J_{0}(x)$, as can be easily seen from \eqref{sumbesselj}, $\psi_w(a)$ defined in \eqref{new psi function} clearly satisfies
\begin{align*}
\psi_w(a)=\psi_w(2-a).
\end{align*}
Similar to the case of $\zeta_w(s, a)$ for Re$(s)>1$, we see that $\psi_w(a)$ also has a removable singularity at $w=0$. Hence by $\psi_0(a)$, we mean\footnote{The interchange of the order of $\lim_{w\to 0}$ and integration can be easily justified.} 
\begin{equation*}
\psi_0(a)=\lim_{w\to 0}\psi_w(a).
\end{equation*}
That
\begin{equation}\label{limitpsi}
\psi_0(a)=\begin{cases}
\psi(a) & \text{if\ $\mathrm{Re}(a)>1$}, \\
\psi(2-a) & \text{if\ $\mathrm{Re}(a)<1$}
\end{cases}
\end{equation}
is derived in Section \ref{psiwa}. It is natural that $\psi_w(a)$ turns out to be a triple integral since $\psi(a)$ is given by a single integral \cite[p.~911, Formula \textbf{8.361.1}]{grn}, namely for Re$(a)>0$,
\begin{align}\label{psiintegral}
\psi(a)=\int_0^\infty \left(\frac{e^{-x}}{x}-\frac{e^{-ax}}{1-e^{-x}}\right)dx,
\end{align}
and, as can be seen from \eqref{int2} below, the residue of $\zeta_w(s, a)$ at $s=1$ itself is a double integral!


Now that we have stated the basic properties of $\zeta_w(s, a)$, we give below the generalized modular relation of the form $F(z, w, \a)=F(z, iw, \b)$, where $\zeta_w(s, a)$ appears naturally. This result along with the one following it generalizes Theorem 1.5 from \cite{transf}, that is, \eqref{hurwitzeqn}. To prove it, however, it is necessary to first obtain a generalized modular-type transformation between two triple integrals (see Theorem \ref{modular trans in intergal form} below). The following theorem is then derived from it.
\begin{theorem}\label{genramhureq}
Let $\zeta_w(s,a)$ be the generalized Hurwitz zeta function\footnote{Note that $\zeta_w(s, a)$ is defined for \textup{Re}$(s)>1$ by \eqref{new zeta}, and for $-1<\textup{Re}(s)\leq1, s\neq1$, by \eqref{hermiteeqn}. The latter is inferred from the paragraph following \eqref{hermiteeqn}.}. Let $w\in\mathbb{C}$, $-1<$ \textup{Re}$(z)<1$ and $x>0$. Define $\varphi_w(z, x)$ by
\begin{align}\label{genramhureqeq}
\varphi_w(z, x)&:=\zeta_w(z+1, x+1)+\frac{1}{2}A_w(z)x^{-z-1}-A_{iw}(-z)\frac{x^{-z}}{z},
\end{align}
where $A_w(z)$ is defined in \eqref{ab}.
Then for $\a, \b>0$ and $\a\b=1$, we have
\begin{align}\label{genramhureqeqn}
&\a^{\frac{z+1}{2}}\left(\sum_{m=1}^{\infty}\varphi_w(z,m\a)-\frac{\zeta(z+1)A_w(z)}{2\alpha^{z+1}}-\frac{\zeta(z)A_w(-z)}{\alpha z}\right)\nonumber\\
&=\b^{\frac{z+1}{2}}\left(\sum_{m=1}^{\infty}\varphi_{iw}(z,m\b)-\frac{\zeta(z+1)A_{iw}(z)}{2\beta^{z+1}}-\frac{\zeta(z)A_{iw}(-z)}{\beta z}\right).
\end{align}
\end{theorem}
The above theorem leads to the following integral evaluation.
\begin{theorem}\label{xiintgenramhurthm}
Let $\a>0$, $w\in\mathbb{C}$ and $-1<$ \textup{Re}$(z)<1$. Let $A_w(z)$ and $\varphi_w(z, x)$ be defined in \eqref{ab} and \eqref{genramhureqeq} respectively. Define 
\begin{align}\label{Del}
\Delta_{2}(x, z, w, s)&:=\omega(x, z, w, s)+\omega(x, z, w, 1-s),\nonumber\\
\omega(x, z, w, s)&:=e^{\frac{w^2}{4}}x^{\frac{1}{2}-s}{}_1F_{1}\left(1-\frac{s+z}{2};\frac{3}{2};-\frac{w^2}{4}\right){}_1F_{1}\left(1-\frac{s-z}{2};\frac{3}{2};-\frac{w^2}{4}\right).
\end{align}
Then
\begin{align}\label{xiintgenramhurthmeqn}
&\frac{2^{z-1}\pi^{\frac{z-3}{2}}}{\Gamma(z+1)}\int_0^\infty \Gamma\left(\frac{z-1+it}{4}\right)\Gamma\left(\frac{z-1-it}{4}\right)\Xi\left(\frac{t+iz}{2}\right)\Xi\left(\frac{t-iz}{2}\right)\frac{\Delta_2\left(\alpha,\frac{z}{2},w,\frac{1+it}{2}\right)}{(z+1)^2+t^2}dt\nonumber\\
&=\a^{\frac{z+1}{2}}\left(\sum_{m=1}^{\infty}\varphi_w(z,m\a)-\frac{\zeta(z+1)A_w(z)}{2\alpha^{z+1}}-\frac{\zeta(z)A_w(-z)}{\alpha z}\right).
\end{align}
\end{theorem}
\begin{remark}
Employing \eqref{kft} below, it is easy to see that $\Delta_2\left(\beta,\frac{z}{2},iw,\frac{1+it}{2}\right)=\Delta_2\left(\alpha,\frac{z}{2},w,\frac{1+it}{2}\right)$.
\end{remark}
The analysis which leads to the generalized modular relation and the integral evaluation in Theorems \ref{genramhureq} and \ref{xiintgenramhurthm} rests upon two theories - first, the theory of another new special function which is a new generalization of the modified Bessel function $K_z(x)$ (different from \eqref{kzw}) and second, the theory of functions reciprocal in a kernel consisting of Bessel functions. The latter is developed in Section \ref{seckosh}. Note that the modified Bessel function of the second kind of order $z$ is defined by \cite[p.~78, eq.~(6)]{watson-1966a}
\begin{equation}\label{knus}
K_{z}(x):=\frac{\pi}{2}\frac{\left(I_{-z}(x)-I_{z}(x)\right)}{\sin\pi z},
\end{equation}
where $I_{z}(x)$ is the modified Bessel function of the first kind of order $z$ given by \cite[p.~77]{watson-1966a}
\begin{equation*}
I_{z}(x):=
\begin{cases}
e^{-\frac{1}{2}\pi z i}J_{z}(e^{\frac{1}{2}\pi i}x), & \text{if $-\pi<$ arg $x\leq\frac{\pi}{2}$,}\\
e^{\frac{3}{2}\pi z i}J_{\nu}(e^{-\frac{3}{2}\pi i}x), & \text{if $\frac{\pi}{2}<$ arg $x\leq \pi$}.
\end{cases}
\end{equation*}
Our generalization of $K_{z}(x)$ is defined for $z, w \in\mathbb{C}$, $x\in\mathbb{C}\backslash\{x\in\mathbb{R}: x\leq 0\}$ and $c:=$Re$(s)>-1\pm$Re$(z)$ by
\begin{align}\label{def}
{}_1K_{z,w}(x)&:=\frac{1}{2\pi i}\int_{(c)}\Gamma\left(\frac{1+s-z}{2}\right)\Gamma\left(\frac{1+s+z}{2}\right) {}_1F_1\left(\frac{1+s-z}{2};\frac{3}{2};-\frac{w^2}{4}\right)\nonumber\\
&\qquad\qquad\times {}_1F_1\left(\frac{1+s+z}{2};\frac{3}{2};-\frac{w^2}{4}\right)2^{s-1}x^{-s}\, ds,
\end{align}
where, here, and throughout the paper, the notation $\int_{(c)}$ is used to denote the line integral $\int_{c-i\infty}^{c+i\infty}$, $c\in\mathbb{R}$. It is clear that ${}_1K_{z,w}(x)$ is an even function in both variables $z$ and $w$. When $w=0$,
\begin{align}\label{Basicform}\nonumber
{}_1K_{z,0}(x)&=\frac{1}{2\pi i}\int_{(c)}\Gamma\left(\frac{1+s-z}{2}\right)\Gamma\left(\frac{1+s+z}{2}\right)2^{s-1}x^{-s}\, ds\\ 
&=xK_z(x),
\end{align}
where the last equality follows by replacing $s$ by $s+1$ in the well-known formula \cite[p. 115, Formula 11.1]{ober}, namely, for $c:=$Re$(s)>\pm$Re$(z)$,
\begin{align}\label{ober1}
\frac{1}{2\pi i}\int_{(c)}2^{s-2}\Gamma\left(\frac{s}{2}-\frac{z}{2}\right)\Gamma\left(\frac{s}{2}+\frac{z}{2}\right)x^{-s}ds=K_z(x).
\end{align}
While the theory of ${}_1K_{w}(x)$ is developed in Section \ref{1kzw}, we record below an integral representation for it.
\begin{theorem}\label{integralRepr}
 For $z,w\in\mathbb{C}$ and $|arg\ x|<\frac{\pi}{4}$, we have
\begin{equation}\label{integralRepreqn}
{}_1K_{z,w}(2x)=\frac{2x^{-z}}{w^2}\int_0^\infty u^{2z-1}e^{-u^2-x^2/u^2}\sin(wu)\sin\left(\frac{wx}{u}\right)\, du.
\end{equation}
\end{theorem}
The two new special functions, $\zeta_w(s, a)$ and ${}_1K_{z,w}(x)$, that we have introduced in this paper are connected by means of the following relation.
\begin{theorem}\label{relation b/w zetaw and 1Kzw}
For $w\in\mathbb{C}$, $a>0$ and $1<\mathrm{Re}(s)<2$, we have
\begin{align*}
\zeta_w(s,a+1)=\frac{2^{s+1}\pi^{\frac{s}{2}}e^{-\frac{w^2}{4}}}{a^{\frac{s-1}{2}}\Gamma\left(1-\frac{s}{2}\right)\Gamma(s)}\int_1^\infty\int_0^\infty \frac{(tx)^{\frac{s-1}{2}}{}_1K_{\frac{s-1}{2},iw}(2\pi atx)}{(t^2-1)^{\frac{s}{2}}(e^{2\pi x}-1)}\, dxdt.
\end{align*}
\end{theorem}

\section{Motivation behind the project}\label{mr}
\setcounter{subsection}{1}
The transformation of the type \eqref{thetaab} which is equivalent to the result obtained by squaring the functional equation of $\zeta(s)$ was found by Koshliakov \cite[p.~32]{koshliakov}. For
\begin{equation*}
\mathfrak{F}(\a):=\sqrt{\alpha}\left(\frac{\gamma-\log (4\pi\alpha)}{\alpha}-4\sum_{n=1}^{\infty}d(n)K_{0}(2\pi n\alpha)\right),
\end{equation*}
it is given by
\begin{align}\label{mainagain30}
 \mathfrak{F}(\a)=\mathfrak{F}(\b)=-\frac{32}{\pi}\int_{0}^{\infty}\Xi^{2}\left(\frac{t}{2}\right)\frac{\cos\left(\frac{1}{2}t\log\alpha\right)\, dt}{(1+t^2)^2}.
\end{align}
Its generalization is the famous Ramanujan-Guinand formula, which, along with the integral involving the Riemann $\Xi$-function, is given by \cite[Theorem 1.4]{transf}
\begin{align}\label{mainagain3}
\mathcal{F}(z,\a)=\mathcal{F}(z,\b)=-\frac{32}{\pi}\int_{0}^{\infty}\Xi\left(\frac{t+iz}{2}\right)\Xi\left(\frac{t-iz}{2}\right)\frac{\cos\left(\frac{1}{2}t\log\alpha\right)}{(t^2+(z+1)^2)(t^2+(z-1)^2)}\, dt,
\end{align}
where
\begin{align}\label{fzalrg}
\mathcal{F}(z, \a):=\sqrt{\alpha}\left(\alpha^{\frac{z}{2}-1}\pi^{\frac{-z}{2}}\Gamma\left(\frac{z}{2}\right)\zeta(z)+\alpha^{-\frac{z}{2}-1}\pi^{\frac{z}{2}}\Gamma\left(\frac{-z}{2}\right)\zeta(-z)-4\sum_{n=1}^{\infty}\sigma_{-z}(n)n^{z/2}K_{\frac{z}{2}}\left(2n\pi\alpha\right)\right).
\end{align}
Here $-1<$ Re$(z)<1$, $\alpha, \beta>0$ (though one can analytically continue the result for Re$(\a)>0$, Re$(\b)>0$), $\a\b=1$, $\sigma_{z}(n)=\sum_{d|n}d^z$ and $K_{z}(x)$ is defined in \eqref{knus}.

It is well-known \cite[p.~60]{cohen} that the first equality in \eqref{mainagain3} is equivalent to the functional equation of the non-holomorphic Eisenstein series on $\textup{SL}_{2}(\mathbb{Z})$. Koshliakov's aforementioned formula is obtained by letting $z\to 0$ in \eqref{mainagain3}.

The natural question that could be asked here is, can one superimpose the theta structure on \eqref{mainagain3}? In other words, does there exist a transformation of the form $F(z, w, \a)=F(z,iw,\b)$ whose special case is \eqref{mainagain3}?

This question was affirmatively answered in \cite[Theorems 1.3, 1.5]{dkmt}, where the following result was obtained for $w\in\mathbb{C}$ and $-1<\textup{Re}(z)<1$\footnote{The first equality in \eqref{genrgeqeqn} is actually valid for $z\in\mathbb{C}\backslash\{-1,1\}$. The condition $-1<$Re$(z)<1$ is required only for both equalities to hold. Same is the case with \eqref{mainagain3}.}:
\begin{align}\label{genrgeqeqn}
\mathbb{F}(z, w, \a)=\mathbb{F}(z, iw, \b)=-\frac{16}{\pi} \int_{0}^{\infty} \Xi\left( \frac{t+iz}{2} \right) \Xi\left( \frac{t-iz}{2} \right)
\frac{\nabla_{2}\left(\a,\tfrac{z}{2},w,\tfrac{1+it}{2}\right)\, dt}{\left(t^2+(z+1)^2\right)\left(t^2+(z-1)^2\right)}, 
\end{align}
where
\begin{align}\label{fzalrggen}
\mathbb{F}(z, w, \a)&=\sqrt{\a}  \Bigg( 
\Gamma\left(\frac{z}{2}\right) \zeta(z)  \pi^{-\frac{z}{2}} \a^{\frac{z}{2}-1} 
\,_1F_1\left(\frac{1-z}{2};\frac{1}{2};\frac{w^2}{4}\right)-4 \sum_{n=1}^{\infty} \sigma_{-z}(n) n^{\frac{z}{2}}  e^{-\frac{w^2}{4}} K_{\frac{z}{2},iw}(2  n \pi\a)
  \nonumber\\
&\quad\quad+ \Gamma\left(-\frac{z}{2}\right) \zeta(-z) \pi^{\frac{z}{2}}  \a^{-\frac{z}{2}-1} \,_1F_1\left(\frac{1+z}{2};\frac{1}{2};\frac{w^2}{4} \right)   
\Bigg).
\end{align}
Here $K_{z,w}(x)$ is the generalized modified Bessel function defined for $z, w \in\mathbb{C}$, $x\in\mathbb{C}\backslash\{x\in\mathbb{R}: x\leq 0\}$ and $c:=$Re$(s)>\pm$Re$(z)$ by \cite[Equation (1.3)]{dkmt}
\begin{equation}\label{kzw}
K_{z,w}(x) := 
\frac{1}{2\pi i} \int_{(c)}\Gamma\bigg(\frac{s-z}{2}\bigg) \Gamma\bigg(\frac{s+z}{2}\bigg) \,_1F_1\bigg(\frac{s-z}{2};\frac{1}{2};\frac{-w^2}{4}\bigg)  
\,_1F_1\bigg(\frac{s+z}{2};\frac{1}{2};\frac{-w^2}{4}\bigg) 
 2^{s-2}x^{-s}ds,
\end{equation}
and
\begin{align*}
\nabla_{2}(x, z, w, s)&:=\rho(x, z, w, s)+\rho(x, z, w, 1-s),\\
\rho(x, z, w, s)&:=e^{\frac{w^2}{4}}x^{\frac{1}{2}-s}{}_1F_{1}\left(\frac{1-s-z}{2};\frac{1}{2};-\frac{w^2}{4}\right){}_1F_{1}\left(\frac{1-s+z}{2};\frac{1}{2};-\frac{w^2}{4}\right).
\end{align*}
The surprising thing is, each of the expressions in \eqref{fzalrg} is just the first term in the series expansion of the corresponding one in \eqref{fzalrggen}. In the same paper \cite{dkmt}, the new special function $K_{z,w}(x)$ is shown to satisfy a theory whose richness is comparable to that of the modified Bessel function $K_{z}(x)$.

Thus, in view of the fact that \eqref{ramfor} and \eqref{hurwitzeqn} are analogous to \eqref{mainagain30} and \eqref{mainagain3}, we are motivated to search for a corresponding analogue of \eqref{genrgeqeqn}.
\section{Preliminaries}\label{prelim}
\setcounter{subsection}{1}
Here we collect some results from the literature that are used throughout the sequel.

The duplication and reflection formulas for the gamma function are given by \cite[p.~46, Equations (3.4), (3.5)]{temme}
\begin{align}
\G(s)\G\left(s+\frac{1}{2}\right)&=\frac{\sqrt{\pi}}{2^{2s-1}}\G(2s),\label{dup}\\
\G(s)\G(1-s)&=\frac{\pi}{\sin(\pi s)}\hspace{4mm}(s\notin\mathbb{Z}).\label{refl}
\end{align}
For $a>0$, $b\in\mathbb{C}$ fixed and $|\arg(s)|<\pi-\delta$, where $\delta>0$, Stirling's formula for the gamma function \cite[p. 141, Formula \textbf{5.11.7}]{nist} is given by 
\begin{equation}\label{strivert0}
\G(as+b)\sim\sqrt{2\pi}e^{-as}(as)^{as+b-1/2}
\end{equation}
as $s\to\infty$. In the vertical strip $p\leq\sigma\leq q$, we also have \cite[p.~224]{cop}
\begin{equation}\label{strivert}
  |\Gamma(s)|=\sqrt{2\pi}|t|^{\sigma-\frac{1}{2}}e^{-\frac{1}{2}\pi |t|}\left(1+O\left(\frac{1}{|t|}\right)\right)
\end{equation}
as $|t|\to \infty$.
The functional equation \eqref{zetafe} of $\zeta(s)$ can be alternatively written in the form 
\begin{equation}\label{zetaalt}
\xi(s)=\xi(1-s),
\end{equation}
where $\xi(s)$ is defined in \eqref{xii}. An asymmetric version of \eqref{zetafe} is \cite[p.~259, Theorem 12.7]{apostol}
\begin{equation}\label{zetafealt}
\zeta(1-s)=2^{1-s}\pi^{-s}\G(s)\zeta(s)\cos\left(\frac{\pi s}{2}\right).
\end{equation}
Kummer's formula for the confluent hypergeometric function is given by
\begin{equation}\label{kft}
{}_1F_{1}(a;c;z)=e^z{}_1F_{1}(c-a;c;-z).
\end{equation}
The error functions $\textup{erfi}(w)$ and $\textup{erf}(w)$ in \eqref{erferfi} and \eqref{erf} are related by
\begin{equation}\label{erferfirel}
\textup{erf}(iw)=i\textup{erfi}(w).
\end{equation}
If $\mathcal{G}(s)$ and $\mathcal{H}(s)$ are the Mellin transforms of $g(x)$ and $h(x)$ respectively, then Parseval's identity \cite[p. 83, Equation (3.1.13)]{kp} reads
\begin{align}\label{Persval}
\int_0^\infty g(x)h\left(\frac{u}{x}\right)\frac{dx}{x}=\frac{1}{2\pi i}\int_{(c)} \mathcal{G}(s)\mathcal{H}(s) u^{-s} ds,
\end{align}
whenever the following conditions hold with $t=\textup{Im}(s)$:
\begin{equation}\label{conditions}
\int_{0}^{\infty}x^{-c-it}h(u/x)\, dx\in L(-\infty, \infty)\hspace{2mm}\textup{and}\hspace{2mm} x^{c-1}g(x)\in L(0, \infty).
\end{equation}
A variant of this formula is \cite[p.~83, Equation (3.1.11)]{kp}
\begin{equation}\label{par}
\int_{0}^{\infty}g(x)h(x)\, dx=\frac{1}{2\pi i}\int_{(c)}\mathcal{G}(1-s)\mathcal{H}(s)\, ds.
\end{equation}
A useful inverse Mellin transform which we will need is \cite[p.~507, Formula \textbf{3.952.7}]{grn}
\begin{equation}\label{mellin transform of gamma 1F1 1}
\frac{1}{2\pi i}\int_{c-i\infty}^{c+i\infty}\frac{b}{2}a^{-\frac{1}{2}-\frac{s}{2}}e^{-\frac{b^2}{4a}}\Gamma\left(\frac{s+1}{2}\right){}_1F_{1}\left(1-\frac{s}{2};\frac{3}{2};\frac{b^2}{4a}\right)t^{-s}\, ds=e^{-at^2}\sin(bt)
\end{equation}
for $c:=\mathrm{Re}(s)>-1$ and $\mathrm{Re}(a)>0$. Equivalently,
\begin{equation}\label{equivv}
\int_{0}^{\infty}t^{s-1}e^{-at^2}\sin(bt)\, dt=\frac{b}{2}a^{-\frac{1}{2}-\frac{s}{2}}e^{-\frac{b^2}{4a}}\Gamma\left(\frac{s+1}{2}\right){}_1F_{1}\left(1-\frac{s}{2};\frac{3}{2};\frac{b^2}{4a}\right).
\end{equation}
The following lemma will be used while proving Theorem \ref{integr}. 
\begin{lemma}\label{IMT of 2F1}
For $\textup{Re}$$(s)<\mathrm{Re}(z)$ and $0<\mathrm{Re}(s)<2$, we have
\begin{align*}
\int_0^\infty {}_2F_1\left(1,\frac{z}{2};\frac{1}{2};-\frac{x^2}{\pi^2n^2}\right)x^{s-1}dx=\frac{n^s\pi^{\frac{3}{2}+s}}{2\sin\left(\frac{\pi s}{2}\right)}\frac{\Gamma\left(\frac{z-s}{2}\right)}{\Gamma\left(\frac{1-s}{2}\right)\Gamma\left(\frac{z}{2}\right)}.
\end{align*}
\end{lemma}
\begin{proof}
This follows from \cite[p. 398, Formula \textbf{15.14.1}]{nist}
\begin{align*}
\int_0^\infty{}_2F_1(a,b;c;-x)x^{s-1}dx=\frac{\Gamma(s)\Gamma(a-s)\Gamma(b-s)\Gamma(c)}{\Gamma(a)\Gamma(b)\Gamma(c-s)},
\end{align*}
which is valid for $\min(\mathrm{Re}(a),\mathrm{Re}(b))>\mathrm{Re}(s)>0$.
\end{proof}
The following two results will be used several times in the sequel.
\begin{theorem}\cite[p.~30, Theorem 2.1]{temme}\label{ldcttemme}
Let $(a, b)$ be a given finite or an infinite interval, and let $\{C_n(t)\}_{n=1}^{\infty}$ be a sequence of real or complex valued continuous functions, which satisfy the following conditions:

\textup{(1)} $\sum_{n=0}^{\infty}C_n(t)$ converges uniformly on any compact interval in $(a,b)$.

\textup{(2)} At least one of the following two quantities is finite:

\begin{equation*}
\int_{a}^{b}\sum_{n=0}^{\infty}|C_n(t)|\, dt, \sum_{n=0}^{\infty}\int_{a}^{b}|C_n(t)|\, dt.
\end{equation*}
Then we have 
\begin{equation*}
\int_{a}^{b}\sum_{n=0}^{\infty}C_n(t)\, dt=\sum_{n=0}^{\infty}\int_{a}^{b}C_n(t)\, dt.
\end{equation*}
\end{theorem}
\begin{theorem}\cite[p.~30, Theorem 2.1]{temme}\label{ldcttemme1}
Let $t$ be a real variable ranging over a finite or infinite interval $(a, b)$ and $z$ a complex variable ranging over a domain $\Omega$. Assume that the function $f:(\Omega\times(a, b))\to\mathbb{C}$ satisfies the following conditions:

\textup{(i)} $f$ is a continuous function of both variables.

\textup{(ii)} For each fixed value of $t$, $f(\cdot, t)$ is a holomorphic function of the first variable.

\textup{(iii)} The integral $F(z)=\int_{a}^{b}f(z, t)\, dt, z\in\Omega$, converges uniformly at both limits in any compact set in $\Omega$. 

Then $F(z)$ is holomorphic in $\Omega$, and its derivatives of all orders may be found by differentiating under the sign of integration.
\end{theorem}
We will be also needing the following result from Titchmarsh's text \cite[p.~25]{titchcomplex}.
\begin{theorem}\label{titchcomplexthm}
If $f(x, y)$ is continuous in the rectangle $c\leq x\leq d$, $m\leq y\leq n$, for all values of $d$, and the integral $\kappa(y):=\int_{c}^{\infty}f(x, y)\, dx$ converges uniformly with respect to $y$ in the interval $(m, n)$, then $\kappa(y)$ is a continuous function of $y$ in this interval.
\end{theorem}
The following result from calculus of several variables \cite[p.~87-88, Proposition 3.3 (iii)]{ghorlim} is essential too.
\begin{theorem}\label{multivariate}
For any $f:[a, b]\times[c, d]\to\mathbb{R}$, if the partial derivatives $\frac{\partial f}{\partial x}$ and $\frac{\partial f}{\partial y}$ exist and one of them is bounded on $[a, b]\times[c, d]$, then $f$ is continuous on $[a, b]\times[c, d]$.
\end{theorem}
Leibnitz's rule for differentiating under the integral sign \cite[p.~337--339]{lang1} is given below.
\begin{theorem}\label{lang}
Let $f$ be a function of two variables $(v,x)$ defined for $v\geq a$ and $x$ in some interval $J=[c,d],\ c<d$. Assume that $\frac{\partial f}{\partial x}$ exists, and that both $f$ and $\frac{\partial f}{\partial x}$ are continuous.  Assume that there are functions $\phi(v)$ and $\psi(v)$ which are $\geq 0$, such that 
	 $|f(v,x)|\leq \phi(v)$ and $|\frac{\partial }{\partial x}f(v,x)|\leq\psi(v)$,
for all $v,x$, and such that the integrals
\begin{align*}
\int_{a}^{\infty}\phi(v)dv\qquad \textup{and}\qquad \int_{a}^{\infty}\psi(v)dv
\end{align*}
converge. Then $\displaystyle\int_{a}^{\infty}f(v,x)dv$ is differentiable, and  
\begin{align*}
\frac{d}{dx}\int_{a}^{\infty}f(v,x)dv=\int_{a}^{\infty}\frac{\partial}{\partial x}f(v,x)dv.
\end{align*}
\end{theorem}

\section{Theory of the generalized Hurwitz zeta function $\zeta_w(s, a)$}\label{nzf}
The analytic properties of $\zeta_w(s, a)$ are established here. We first consider the simpler case $a=1$.

\subsection{The generalized Riemann zeta function $\zeta_w(s, 1)$}\label{seczws1}
\hfill\\

This subsection is devoted to proving properties of the function $\zeta_w(s, 1)$.\\

\textbf{Proof of \eqref{a=1caserz}:} Make a change of variable $u\to y/v$ in \eqref{int1} and recall the definition of $A_{w}(z)$ in \eqref{ab} so that
\begin{equation}\label{d1}
\int_{0}^{\infty}\int_{0}^{\infty}\frac{y^z}{v}e^{-(y^2+v^2)}\sin(wv)\sinh(wy)dydv=\frac{\pi w}{4}\Gamma\left(1+\frac{z}{2}\right)\mathrm{erf}\left(\frac{w}{2}\right){}_1F_1\left(1+\frac{z}{2};\frac{3}{2};\frac{w^2}{4}\right).
\end{equation}
Now switch the roles of $y$ and $v$ in the above double integral and note that 
\begin{align}\label{d2}
\int_{0}^{\infty}\int_{0}^{\infty}\frac{y^z}{v}e^{-(y^2+v^2)}\sin(wv)\sinh(wy)dydv&=\int_{0}^{\infty}\int_{0}^{\infty}\frac{v^z}{y}e^{-(v^2+y^2)}\sin(wy)\sinh(wv)dvdy\nonumber\\
&=\int_{0}^{\infty}\int_{0}^{\infty}\frac{v^z}{y}e^{-(v^2+y^2)}\sin(wy)\sinh(wv)dydv,
\end{align}
where in the last step we applied Fubini's theorem owing to the absolute convergence of the double integral. Equating the right-hand sides of \eqref{d1} and \eqref{d2}, multiplying both sides by $\frac{4}{w^2\sqrt{\pi}\G\left(1+\frac{z}{2}\right)}$, letting $z=s-1$, replacing $w$ by $iw$, and using \eqref{erferfirel}, we deduce that
\begin{align*}
\frac{4}{w^2\sqrt{\pi}\G\left(\frac{1+s}{2}\right)}\int_0^\infty\int_0^\infty\frac{v^{s-1}}{y}e^{-(v^2+y^2)}\sin(wv)\sinh(wy)dydv=\frac{\sqrt{\pi}}{w}\ \mathrm{erfi}\left(\frac{w}{2}\right){}_1F_1\left(\frac{1+s}{2};\frac{3}{2};-\frac{w^2}{4}\right).
\end{align*} 
This proves the second equality in \eqref{a=1caserz}.
\qed

\begin{remark}
From \eqref{a=1caserz} and the facts that ${}_1F_1\left(\frac{1+s}{2};\frac{3}{2};-\frac{w^2}{4}\right)$ is an entire function of $s$ and $\zeta(s)$ has analytic continuation in $\mathbb{C}\backslash\{1\}$, it is clear that $\zeta_w(s, 1)$ has analytic continuation in the whole $s$-complex plane except $s=1$ given by the right-hand side of \eqref{a=1caserz}.
\end{remark}
It is easy to obtain the functional equation for $\zeta_w(s, 1)$. It is given in the following theorem.
\begin{theorem}
For $s, w\in\mathbb{C}$,
\begin{equation}\label{fezws1}
\zeta_w(1-s, 1)=2e^{-\frac{w^2}{4}}\frac{\textup{erfi}\left(\frac{w}{2}\right)}{\textup{erf}\left(\frac{w}{2}\right)}\frac{\G(s)}{(2\pi)^{s}}\cos\left(\frac{\pi s}{2}\right)\zeta_{iw}(s, 1).
\end{equation}
\end{theorem}
\begin{proof}
First let $s\neq0, 1$. Since \eqref{a=1caserz} holds for all $s\in\mathbb{C}\backslash\{1\}$,
\begin{align*}
\zeta_w(1-s,1)&=\frac{\sqrt{\pi}}{w}\ \mathrm{erfi}\left(\frac{w}{2}\right){}_1F_1\left(1-\frac{s}{2};\frac{3}{2};-\frac{w^2}{4}\right)\zeta(1-s)\nonumber\\
&=e^{-\frac{w^2}{4}}\frac{\sqrt{\pi}}{w}\ \mathrm{erfi}\left(\frac{w}{2}\right){}_1F_1\left(\frac{1+s}{2};\frac{3}{2};\frac{w^2}{4}\right)2^{1-s}\pi^{-s}\G(s)\cos\left(\frac{\pi s}{2}\right)\zeta(s)\nonumber\\
&=2e^{-\frac{w^2}{4}}\frac{\textup{erfi}\left(\frac{w}{2}\right)}{\textup{erf}\left(\frac{w}{2}\right)}\frac{\G(s)}{(2\pi)^{s}}\cos\left(\frac{\pi s}{2}\right)\zeta_{iw}(s, 1),
\end{align*}
where in the second step, we used \eqref{zetafealt} as well as \eqref{kft}. Now both sides of \eqref{fezws1} have simple pole at $s=0$ (due to $\zeta_w(s, 1)$ and $\G(s)$). Also, even though $\zeta_{iw}(s, 1)$ has a pole at $s=1$, $\cos\left(\frac{\pi s}{2}\right)$ has a zero there. Hence we conclude that $\eqref{fezws1}$ holds for all $s\in\mathbb{C}$.
\end{proof}
\subsection{The generalized Hurwitz zeta function $\zeta_w(s, a)$}\label{seczwsa}
\hfill\\

We begin this subsection by proving Theorem \ref{zwsares1}. However, we first need a lemma.
\begin{lemma}\label{double integral is analytic}
Let $w\in\mathbb{C}, n\in\mathbb{N}$ and $a\in\mathbb{R}$. Then the integral 
\begin{align*}
\int_0^\infty\int_0^\infty\frac{(uv)^{s-1}e^{-(u^2+v^2)}\sin(wv)\sinh(wu)}{(n^2u^2+(a-1)^2v^2)^{s/2}}\ dudv
\end{align*}
is analytic for $\mathrm{Re}(s)>1$.
\end{lemma}
\begin{proof}
This is achieved by repeatedly employing Theorem \ref{ldcttemme1}. Let 
\begin{align*}
f_v(s,u):=\frac{u^{s-1}e^{-u^2}\sinh(wu)}{(n^2u^2+(a-1)^2v^2)^{s/2}}.
\end{align*}
We first show that $\int_{0}^{\infty}f_v(s,u)\ du$ is analytic for $\mathrm{Re}(s)>1$. It is easy to see that $f_v(s,u)$ is a continuous function in both $s$ and $u$ and also an analytic function of $s$ in $\mathrm{Re}(s)>1$. Now consider a compact subset of $\mathrm{Re}(s)>1$ given by
\begin{equation*}
\Omega:=\{s\in\mathbb{C}: 1<\theta_1\leq\textup{Re}(s)\leq\theta_2, \gamma_1\leq\textup{Im}(s)\leq\gamma_2\}.
\end{equation*}
We next show that $\int_0^\infty f_v(s,u)\ du$ converges uniformly at both the limits on compact subsets $\Omega$ of $\mathrm{Re}(s)>1$. Let $\epsilon>0$ be given. Let $0<A<B<1$. Then
\begin{align*}
\left|\int_A^Bf_v(s,u)\ du\right|&\leq \int_A^B \frac{u^{\mathrm{Re}(s)}e^{-u^2}}{(nu)^{\mathrm{Re}(s)}}\frac{|\sinh(wu)|}{u}\ du\nonumber\\
&\leq \frac{\sinh(|w|)}{n^{\theta_1}}\int_A^Be^{-u^2}\ du \nonumber\\
&=\frac{\sinh(|w|)}{n^{\theta_1}}\frac{\sqrt{\pi}}{2}(\mathrm{erf}(B)-\mathrm{erf}(A)),
\end{align*}
for all $s\in\Omega$. Thus we can choose $\lambda,\ 0<\lambda<1$, such that $\frac{\sqrt{\pi}\sinh(|w|)}{2n^{\theta_1}}(\mathrm{erf}(B)-\mathrm{erf}(A))<\epsilon$ for $|A-B|<\lambda$. The uniform convergence at the upper limit can be similarly observed.
Therefore by Theorem \ref{ldcttemme1}, integral $\int_0^\infty f_v(s,u)\ du$ is an analytic function of $s$ in $\mathrm{Re}(s)>1$.

Now we will show $\int_0^\infty\int_0^\infty v^{s-1}e^{-v^2}\sin(wv)f_v(s,u)\ dudv$ is analytic in Re$(s)>1$.

It is clear that $\int_0^\infty v^{s-1}e^{-v^2}\sin(wv)f_v(s,u)\ du$ is continuous in $s$ from what we just proved. To show the continuity of this integral in $v$ on $(0, \infty)$ as well, we make use of the Theorem \ref{titchcomplexthm}. It is easy to see that $v^{s-1}e^{-v^2}\sin(wv)f_v(s,u)$ is continuous on $0\leq u\leq \kappa,\ 0\leq v\leq\delta$. Thus we need only show $\int_0^\infty v^{s-1}e^{-v^2}\sin(wv)f_v(s,u)\ du$ is uniformly convergent in $v$ on $(0,\delta)$. Let $\epsilon>0$. We find an $X_0$ depending only on $\epsilon$ such that for all $v\in(0,\delta)$ and $X\geq X_0$,
\begin{align*}
\left|\int_0^\infty v^{s-1}e^{-v^2}\sin(wv)f_v(s,u)\ du-\int_0^X v^{s-1}e^{-v^2}\sin(wv)f_v(s,u)\ du\right|<\epsilon.
\end{align*}
To that end,
\begin{align*}
\left|\int_X^\infty v^{s-1}e^{-v^2}\sin(wv)f_v(s,u)\ du\right|&\leq e^{-v^2}v^{\mathrm{Re}(s)}\frac{|\sin(wv)|}{v}\int_X^\infty \frac{u^{\mathrm{Re}(s)-1}e^{-u^2}|\sin(wu)|}{(n^2u^2)^{\mathrm{Re}(s)/2}}\ du\nonumber\\
&\leq \frac{\delta^{\mathrm{Re}(s)-1}}{n^{\mathrm{Re}(s)}}\sinh(|w|\delta)\int_X^\infty u^{-1}e^{-u^2}|\sin(wu)|\ du, 
\end{align*}
for all $v\in(0,\delta)$ and $\mathrm{Re}(s)>1$. Note that we can always find an $X_0$ large enough such that for $X\geq X_0$, we have
\begin{align*}
\int_X^\infty u^{-1}e^{-u^2}|\sin(wu)|\ du<\epsilon.\frac{n^{\mathrm{Re}(s)}\sinh(|w|\delta)}{\delta^{\mathrm{Re}(s)-1}}.
\end{align*} 
Therefore the integral $\int_0^\infty v^{s-1}e^{-v^2}\sin(wv)f_v(s,u)\ du$ is uniformly convergent on $v\in(0,\delta)$. Hence by Theorem \ref{titchcomplexthm}, this integral is continuous in $v$ on $(0,\delta)$ for every $\delta>0$, and hence on $(0,\infty)$. It is easy to see that condition (ii) of Theorem \ref{ldcttemme1} is satisfied for $\mathrm{Re}(s)>1$. To show that the condition (iii) of Theorem \ref{ldcttemme1} is also satisfied, let $\epsilon>0$. Let $0<A<B<1$. Then 
\begin{align*}
&\left|\int_A^B\int_0^\infty v^{s-1}e^{-v^2}\sin(wv)f_v(s,u)\ dudv\right| \nonumber\\
&\leq\int_A^B \int_0^\infty \frac{(uv)^{\mathrm{Re}(s)-1}e^{-(u^2+v^2)}|\sinh(wu)\sin(wv)|}{|n^2u^2+(a-1)^2v^2|^{\textup{Re}(s)/2}}\ dudv\nonumber\\
&\leq\int_A^B v^{\mathrm{Re}(s)-1}e^{-v^2}|\sin(wv)|\int_{0}^{\infty}e^{-u^2}\frac{|\sinh(wu)|}{u}\ dudv\nonumber\\
&\leq C_{w}\int_A^B v^{\mathrm{Re}(\theta_1)}\ dv\nonumber\\
&\leq C_{w}\frac{(B^{\theta_1+1}-A^{\theta_1+1})}{\theta_1+1},
\end{align*}
for all $s\in\Omega$, where $C_{w}$ is a constant depending on $w$. Therefore we can choose $\delta$, $0<\delta<1$, such that $C_{w}\frac{(B^{\theta_1+1}-A^{\theta_1+1})}{\theta_1+1}<\epsilon$ for $|B-A|<\delta$. The uniform convergence at the upper limit can be similarly observed. Therefore by another application of Theorem \ref{ldcttemme1}, we see that the double integral $\int_0^\infty\int_0^\infty\frac{(uv)^{s-1}e^{-(u^2+v^2)}\sin(wv)\sinh(wu)}{(n^2u^2+(a-1)^2v^2)^{s/2}}\ dudv$ is analytic in Re$(s)>1$.
\end{proof}
\begin{proof}[Theorem \textup{\ref{zwsares1}}][]
As stated in the introduction, it is easy to see that
\begin{equation*}
\displaystyle\int_0^\infty\int_0^\infty\frac{(uv)^{s-1}e^{-(u^2+v^2)}\sin(wv)\sinh(wu)}{(n^2u^2+(a-1)^2v^2)^{s/2}}\ dudv=O_{w, a}\left(n^{-\textup{Re}(s)}\right)
\end{equation*}
as $n\to\infty$. Hence the series in \eqref{new zeta} converges uniformly on Re$(s)>1$. Since Lemma \ref{double integral is analytic} implies that the above double integral is analytic in Re$(s)>1$, by Weierstrass' theorem on analytic functions, we conclude that $\zeta_w(s, a)$ is analytic in Re$(s)>1$.
\end{proof}
Next, we prove a lemma which is crucial in proving not only \eqref{a=1caserz} but also Theorems \ref{lse} and \ref{xiintgenramhurthm}.
\begin{lemma}\label{int}
Let $A_w(z)$ be defined in \eqref{ab}. For $\mathrm{Re}(z)>-2$ and $ w\in\mathbb{C}$, we have
\begin{align}\label{int1}
\int_0^\infty\int_0^\infty u^z v^z e^{-v^2(u^2+1)}\sin(wv)\sinh(wuv)dudv=\frac{\sqrt{\pi} w^2}{4}\Gamma\left(1+\frac{z}{2}\right)A_w(z),
\end{align}
and, for $\textup{Re}(z)>-1$ and $w\in\mathbb{C}$,
\begin{align}\label{int2}
\int_0^\infty\int_0^\infty u^{z-1} v^z e^{-v^2(u^2+1)}\sin(wv)\sinh(wuv)dudv=\frac{w^2}{4}\Gamma\left(\frac{1+z}{2}\right)A_{iw}(-z).
\end{align}
\end{lemma}

\begin{proof}
Let $s=z+1, a=v^2$, and $b=iwv$ in \eqref{equivv} so that for Re$(z)>-2$ and Re$(v^2)>0$,
\begin{align}\label{prud2}
\int_0^\infty u^ze^{-u^2v^2}\sinh(wuv)du
=\frac{1}{2}wv^{-1-z}\Gamma\left(1+\frac{z}{2}\right){}_1F_1\left(1+\frac{z}{2};\frac{3}{2};\frac{w^2}{4}\right),
\end{align}%
where we used \eqref{kft} for simplification. Now multiply the above equation by $v^ze^{-v^2}\sin(wv)$ and integrate the resulting equation with respect to real $v$ from $0$ to $\infty$ so that
\begin{align*}
&\int_0^\infty\int_0^\infty u^z v^z e^{-v^2(u^2+1)}\sin(wv)\sinh(wuv)dudv\\
&=\frac{1}{2}w\Gamma\left(1+\frac{z}{2}\right){}_1F_1\left(1+\frac{z}{2};\frac{3}{2};\frac{w^2}{4}\right)\int_0^\infty e^{-v^2}\sin(wv)\frac{dv}{v}\\
&=\frac{\pi w}{4}\Gamma\left(1+\frac{z}{2}\right)\mathrm{erf}\left(\frac{w}{2}\right){}_1F_1\left(1+\frac{z}{2};\frac{3}{2};\frac{w^2}{4}\right),
\end{align*}
which follows from employing \eqref{equivv} with $s=0, b=w$, and $a=1$, and then using \eqref{erf}.
Equation \eqref{int2} can be similarly proved.
\end{proof}

To prove Theorem \ref{hurwitzsc}, we begin with a lemma.
\begin{lemma}\label{hurwitz as a special case}
For $\mathrm{Re}(s)>0, n\in\mathbb{R}^{+}$ and $a\in\mathbb{C}\backslash\{\xi: \textup{Re}(\xi)=1, \textup{Im}(\xi)\neq0\}$, we have
\begin{align}\label{aaaa}
 \frac{4}{\sqrt{\pi}}\frac{1}{\Gamma\left(\frac{1+s}{2}\right)} \int_0^\infty\int_0^\infty\frac{(uv)^{s}e^{-(u^2+v^2)}}{(n^2u^2+(a-1)^2v^2)^{\frac{s}{2}}}dudv =
  \begin{cases}
   (n+a-1)^{-s} & \text{if\ $\mathrm{Re}(a)>1$}, \\
   (n-a+1)^{-s} & \text{if\ $\mathrm{Re}(a)<1$}, \\
	n^{-s} & \text{if\ $a=1$}.
  \end{cases}
\end{align}
\end{lemma}
\begin{proof}
The $a=1$ case can be easily proved by noting that
\begin{align*}
\int_{0}^{\infty}e^{-u^2}\, du&=\frac{\sqrt{\pi}}{2},\\
\int_{0}^{\infty}v^se^{-v^2}\, dv&=\frac{1}{2}\G\left(\frac{1+s}{2}\right),
\end{align*}
the first of which is a well-known result and the second results from the integral representation of the Gamma function. Now assume $a\neq 1$. 

We first prove the result for $1<a<n\sqrt{2}+1$ and later extend it to Re$(a)>1$ by analytic continuation. Employing change of variable $u=\frac{(a-1)v}{n}\sqrt{t}$ in the second step below, we see that
\begin{align*}
\int_0^\infty \frac{u^se^{-u^2}}{(n^2u^2+(a-1)^2v^2)^{\frac{s}{2}}}du&=\frac{1}{n^s}\int_0^\infty \frac{e^{-u^2}}{\left(1+\frac{(a-1)^2v^2}{n^2u^2}\right)^{\frac{s}{2}}}du\nonumber\\
&=\frac{(a-1)v}{2n^{s+1}}\int_0^\infty \frac{t^{\frac{s-1}{2}}}{(1+t)^{\frac{s}{2}}}e^{-t\frac{(a-1)^2v^2}{n^2}}\, dt.
\end{align*}
From \cite[p. 326, Equation (13.4.4)]{nist}, for $\mathrm{Re}(b)>0$ and $|\arg(\xi)|<\frac{\pi}{2}$,
\begin{align*}
U(b;c;\xi)=\frac{1}{\Gamma(b)}\int_0^\infty e^{-\xi t}t^{b-1}(1+t)^{c-b-1}dt,
\end{align*}
where $U(b,c;\xi)$ is the Tricomi confluent hypergeometric function. Using this result with $b=\frac{s+1}{2},\ c=\frac{3}{2}$ and $\xi=\frac{(a-1)^2v^2}{n^2}$ so that for $a>1$\footnote{This result is, in fact, true for Re$(a)>1$.} and Re$(s)>0$,
\begin{align*}
\int_0^\infty \frac{u^se^{-u^2}}{(n^2u^2+(a-1)^2v^2)^{\frac{s}{2}}}du&=\frac{(a-1)v}{2n^{s+1}}\Gamma\left(\frac{s+1}{2}\right)U\left(\frac{s+1}{2};\frac{3}{2};\frac{(a-1)^2v^2}{n^2}\right).
\end{align*} 
This gives
\begin{align}\label{hw15}
&\int_0^\infty\int_0^\infty\frac{(uv)^se^{-(u^2+v^2)}}{(n^2u^2+(a-1)^2v^2)^{\frac{s}{2}}}dudv\nonumber\\
&=\frac{(a-1)}{2n^{s+1}}\Gamma\left(\frac{s+1}{2}\right)\int_0^\infty v^{s+1}e^{-v^2}U\left(\frac{s+1}{2};\frac{3}{2};\frac{(a-1)^2v^2}{n^2}\right)dv\nonumber\\
&=\frac{1}{4n^{s}}\left(\frac{n}{a-1}\right)^{s+1}\Gamma\left(\frac{s+1}{2}\right)\int_0^\infty y^{\frac{s}{2}}e^{-\frac{n^2y}{(a-1)^2}}U\left(\frac{s+1}{2};\frac{3}{2};y\right)dy,
\end{align}
where we again employed change of variable $v=\frac{n}{(a-1)}\sqrt{y}$. Next, from \cite[p. 830, Equation (7.621.6)]{grn}, for $\mathrm{Re}(\xi)>\frac{1}{2}$,
\begin{align*}
\int_0^\infty t^{b-1}U(a;c;t)e^{-\xi t}dt=\frac{\Gamma(b)\Gamma(b-c+1)}{\Gamma(a+b-c+1)}\xi^{-b}{}_2F_1\left(a,b;a+b-c+1;1-\frac{1}{\xi}\right).
\end{align*}
Now let $b=\frac{s}{2}+1,\ a=\frac{s+1}{2},\ c=\frac{3}{2}$ and $\xi=\frac{n^2}{(a-1)^2}$ in the above identity, substitute the resultant in \eqref{hw15}, and then use the duplication formula \eqref{dup} for simplification so that for $\frac{n}{a-1}>\frac{1}{\sqrt{2}}$,
\begin{align}\label{hw13}
&\int_0^\infty\int_0^\infty\frac{(uv)^se^{-(u^2+v^2)}}{(n^2u^2+(a-1)^2v^2)^{\frac{s}{2}}}dudv\nonumber\\
&=\frac{(a-1)\sqrt{\pi}}{2^{s+2}n^{s+1}}\Gamma\left(\frac{s+1}{2}\right){}_2F_1\left(\frac{s}{2}+1,\frac{s+1}{2};s+1;1-\frac{(a-1)^2}{n^2}\right),\nonumber\\
&=\frac{\sqrt{\pi}}{2^{s+2}(a-1)^s}\Gamma\left(\frac{s+1}{2}\right){}_2F_1\left(\frac{s}{2},\frac{s+1}{2};s+1;1-\frac{n^2}{(a-1)^2}\right),
\end{align}
where in the last step, we invoked Pfaff's transformation \cite[p.~110, Equation (5.5)]{temme} valid for $|\arg(1-z)|<\pi$:
\begin{align*}
{}_2F_1(a,b;c;z)=(1-z)^{-b}{}_2F_1\left(c-a,b;c;\frac{z}{z-1}\right).
\end{align*}
From \cite[p. 130, Exercise 5.7, Equation (3)]{temme}, for $\textup{Re}(b)>0$ and $|\arg(1-z)|<\pi$,
\begin{equation*}
{}_2F_1\left(a, b; 2b; z\right)=\left(\frac{1}{2}+\frac{1}{2}\sqrt{1-z}\right)^{-2a}{}_2F_{1}\left(a, a-b+\frac{1}{2};b+\frac{1}{2};\left(\frac{1-\sqrt{1-z}}{1+\sqrt{1-z}}\right)^2\right).
\end{equation*}
Letting $a=s/2, b=(s+1)/2$ and $z=1-n^2/(a-1)^2$ in the above equation, we get
\begin{equation}\label{h2}
{}_2F_1\left(\frac{s}{2},\frac{s+1}{2};s+1;1-\frac{n^2}{(a-1)^2}\right)=2^s\left(1+\frac{n}{a-1}\right)^{-s}.
\end{equation}
Thus from \eqref{hw13} and \eqref{h2}, for $1<a<n\sqrt{2}+1$ and $\mathrm{Re}(s)>0$,
\begin{align}\label{a>1}
\frac{4}{\sqrt{\pi}}\frac{1}{\Gamma\left(\frac{1+s}{2}\right)} \int_0^\infty\int_0^\infty\frac{(uv)^{s}e^{-(u^2+v^2)}}{(n^2u^2+(a-1)^2v^2)^{\frac{s}{2}}}dudv=\frac{1}{(n+a-1)^s}.
\end{align}
Now by a repeated application of Theorem \ref{ldcttemme1}, it can be seen\footnote{We do not give the details of the argument here as a similar one is established in full detail in Lemma \ref{tia}.} that the integral \newline $\displaystyle\int_0^\infty\int_0^\infty \frac{(uv)^se^{-(u^2+v^2)}}{(n^2u^2+(a-1)^2v^2)^{\frac{s}{2}}}dudv$ is analytic for Re$(a)>1$. Since the right-hand side of \eqref{a>1} is also analytic for Re$(a)>1$, we conclude that \eqref{a>1} holds for Re$(a)>1$. To prove the remaining case when Re$(a)<1$, replace $a$ by $2-a$ and invoke the evaluation just proved.
\end{proof}
\begin{proof}[Theorem \textup{\ref{hurwitzsc}}][]
Let $n\in\mathbb{N}$ in \eqref{aaaa}. Assume Re$(s)>1$ and sum both sides of \eqref{aaaa} over $n$ from $1$ to $\infty$ whence
\begin{equation*}
\frac{4}{\sqrt{\pi}}\frac{1}{\Gamma\left(\frac{1+s}{2}\right)}\sum_{n=1}^{\infty}\int_0^\infty\int_0^\infty\frac{(uv)^{s}e^{-(u^2+v^2)}}{(n^2u^2+(a-1)^2v^2)^{\frac{s}{2}}}dudv
=\begin{cases}
\zeta(s,a) & \text{if\ $\mathrm{Re}(a)>1$}, \\
   \zeta(s, 2-a) & \text{if\ $\mathrm{Re}(a)<1$}.
\end{cases}
\end{equation*}
Since the left-hand side of the above equation is $\zeta_0(s, a)$, the proof is complete.
\end{proof}

The Mellin transform representations of $\zeta(s)$ and $\zeta(s, a)$ \cite[p.~251]{apostol} are central to developing their theories. An analogous representation for the generalized Hurwitz zeta function $\zeta_w(s, a)$ is now presented. The main ingredient towards obtaining it is the following result of Kapteyn \cite[p.~5]{kapteyn} (see also \cite[Equation (9), p.~386]{watson-1966a}, \cite[p. 710, Formula \textbf{6.24.7}]{grn}) for $\mathrm{Re}(\nu)>0$ and $|\mathrm{Im}(b)|<\pi$:
\begin{align}\label{grad result1}
\sum_{n=1}^\infty\frac{1}{(n^2\pi^2+b^2)^{\nu+\frac{1}{2}}}=\frac{\sqrt{\pi}}{(2b)^{\nu}\Gamma\left(\nu+\frac{1}{2}\right)}\int_0^\infty \frac{J_\nu(bx)x^\nu}{e^{\pi x}-1}\, dx.
\end{align}
\begin{theorem}\label{mtr}
For $a\in\mathbb{R}$ and $\textup{Re}(s)>1$, 
\begin{align}\label{double integral with bessel j}
\zeta_w(s,a)
&=\frac{2^{\frac{s+3}{2}}}{w^2\sqrt{\pi}(a-1)^{\frac{s-1}{2}}\Gamma(s)}\int_0^\infty\int_0^\infty\int_0^\infty \frac{(uv)^{\frac{s-1}{2}}}{u}\frac{e^{-(u^2+v^2)}\sin(wv)\sinh(wu)}{e^{y}-1}\nonumber\\
&\qquad\qquad\qquad\qquad\qquad\qquad\quad\times J_{\frac{s-1}{2}}\left(\frac{(a-1)vy}{u}\right)y^{\frac{s-1}{2}}dydudv.
\end{align}
\end{theorem}
\begin{proof}
Let $b=(a-1)\pi v/u,\ \nu=(s-1)/2$ in \eqref{grad result1} and employ the change of variable $x=y/\pi$ to obtain for $\mathrm{Re}(s)>1$,
\begin{align*}
\sum_{n=1}^\infty \frac{1}{(n^2u^2+(a-1)^2v^2)^{\frac{s}{2}}}=\frac{\sqrt{\pi} u^{-\frac{(s+1)}{2}}}{2^{\frac{s-1}{2}}\Gamma\left(\frac{s}{2}\right)((a-1) v)^{\frac{s-1}{2}}}\int_0^\infty 
\frac{y^{\frac{s-1}{2}}}{e^{y}-1}J_{\frac{s-1}{2}}\left(\frac{(a-1)vy}{u}\right)\ dy.
\end{align*}
Now interchange the order of summation and integration in the definition \eqref{new zeta} of $\zeta_w(s, a)$ by invoking Theorem \ref{ldcttemme} twice, substitute the above representation for the infinite sum, and simplify to arrive at \eqref{double integral with bessel j}.
\end{proof}
\begin{remark}
Note that the function $(a-1)^{\frac{1-s}{2}}J_{\frac{s-1}{2}}\left(\frac{(a-1)vy}{u}\right)$ has a removable singularity at $a=1$. Moreover, when we let $a\to 1$ in \eqref{double integral with bessel j}, we can pass the limit through the triple integral by appealing to Theorem \ref{titchcomplexthm}. This results in \eqref{a=1caserz}.
\end{remark}
\begin{remark}
If we let $w\to 0$ in Theorem \textup{\ref{mtr}}, then the double integral over $u$ and $v$ can be explicitly evaluated leading to the Mellin transform representation of $\zeta(s, a)$, namely \cite[p.~251]{apostol}, whose equivalent form is given in Corollary \textup{\ref{getsplcor}} below.
\end{remark}

We now derive a representation for $\zeta_w(s, a)$ which is imperative in obtaining the modular-type transformation stated in Theorem \ref{genramhureq}.
\begin{theorem}\label{befhermite}
Let $A_w(z)$ be defined in \eqref{ab}. For $a>0$ and $1<\textup{Re}(s)<2$,
\begin{align}\label{hm90}
\zeta_w(s, a+1)&=-\frac{a^{-s}}{2}A_w(s-1)+\frac{a^{1-s}}{s-1}A_{iw}(1-s)+\frac{2^{\frac{s+3}{2}}a^{\frac{1-s}{2}}}{w^2\sqrt{\pi}\Gamma(s)}\int_0^\infty\int_0^\infty\int_0^\infty u^{\frac{s-3}{2}}v^{\frac{s-1}{2}}\nonumber\\
&\quad\times e^{-(u^2+v^2)}\sin(wv)\sinh(wu)J_{\frac{s-1}{2}}\left(\frac{avx}{u}\right)x^{\frac{s-1}{2}}\left(\frac{1}{e^{x}-1}-\frac{1}{x}+\frac{1}{2}\right)dxdudv.
\end{align}
\end{theorem}
\begin{proof}
Replace $a$ by $a+1$ in Theorem \ref{mtr} and rewrite it in the form
\begin{align}\label{mtrmasale}
\zeta_w(s,a+1)&=\frac{2^{\frac{s+3}{2}}a^{\frac{1-s}{2}}}{w^2\sqrt{\pi}\Gamma(s)}\bigg\{\int_0^\infty\int_0^\infty\int_0^\infty u^{\frac{s-3}{2}}v^{\frac{s-1}{2}}e^{-(u^2+v^2)}\sin(wv)\sinh(wu)y^{\frac{s-1}{2}}\nonumber\\
&\quad\times J_{\frac{s-1}{2}}\left(\frac{avy}{u}\right)\left(\frac{1}{e^{y}-1}-\frac{1}{y}+\frac{1}{2}\right)\, dydudv+I_1(a, s, w)-I_2(a, s, w)\bigg\},
\end{align}
where
\begin{align*}
I_1(a, s, w)&:=\int_0^\infty\int_0^\infty\int_0^\infty u^{\frac{s-3}{2}}v^{\frac{s-1}{2}}e^{-(u^2+v^2)}\sin(wv)\sinh(wu)\nonumber\\
&\quad\times J_{\frac{s-1}{2}}\left(\frac{avy}{u}\right)y^{\frac{s-3}{2}}\, dydudv\nonumber\\
I_2(a, s, w)&:=\frac{1}{2}\int_0^\infty\int_0^\infty\int_0^\infty u^{\frac{s-3}{2}}v^{\frac{s-1}{2}}e^{-(u^2+v^2)}\sin(wv)\sinh(wu)\nonumber\\
&\quad\times J_{\frac{s-1}{2}}\left(\frac{avy}{u}\right)y^{\frac{s-1}{2}}\, dydudv.
\end{align*}
We first evaluate $I_2(a, s, w)$ followed by $I_1(a, s, w)$. 

From \cite[p. 684, Formula \textbf{6.561.14}]{grn}, for $-\mathrm{Re}(\nu)-1<\mathrm{Re}(\mu)<\frac{1}{2},\ b>0$,
\begin{align}\label{mt of bessel J}
\int_0^\infty x^\mu J_\nu(bx)\ dx=2^\mu b^{-\mu-1}\frac{\Gamma\left(\frac{1+\nu+\mu}{2}\right)}{\Gamma\left(\frac{1+\nu-\mu}{2}\right)}.
\end{align}
Let $\mu=\frac{s-1}{2},\ \nu=\frac{s-1}{2},\ b=\frac{av}{u}$ in \eqref{mt of bessel J} so that for $0<\mathrm{Re}(s)<2$, $a>0, u>0$, and $v>0$,
\begin{align}\label{bessel J 93}
\int_0^\infty J_{\frac{s-1}{2}}\left(\frac{avx}{u}\right)x^{\frac{s-1}{2}}\ dx=\frac{2^{\frac{s-1}{2}}}{\sqrt{\pi}}\Gamma\left(\frac{s}{2}\right)\left(\frac{u}{av}\right)^{\frac{s+1}{2}},
\end{align}
 and then evaluate the innermost integral in $I_2(a, s, w)$ using \eqref{bessel J 93} whence for $0<\mathrm{Re}(s)<2$ and $a>0$,
\begin{align}\label{i2asw}
I_2(a, s, w)&=\frac{2^{\frac{s-3}{2}}}{\sqrt{\pi}a^{\frac{s+1}{2}}}\Gamma\left(\frac{s}{2}\right)\int_0^\infty\int_0^\infty \frac{u^{s-1}}{v}e^{-(u^2+v^2)}\sin(wv)\sinh(wu)\ dudv \nonumber\\
&=\frac{2^{\frac{s-3}{2}}\sqrt{\pi}w}{4a^{\frac{s+1}{2}}}\Gamma\left(\frac{s}{2}\right)\Gamma\left(\frac{s+1}{2}\right)\mathrm{erf}\left(\frac{w}{2}\right){}_1F_1\left(\frac{s+1}{2};\frac{3}{2};\frac{w^2}{4}\right)\nonumber\\
&=\frac{2^{\frac{s-3}{2}}w^2}{4a^{\frac{s+1}{2}}}\Gamma\left(\frac{s}{2}\right)\Gamma\left(\frac{s+1}{2}\right)A_w(s-1),
\end{align}
where in the last step, we used \eqref{int1} with $u$ replaced by $u/v$ and $z$ by $s-1$. Similarly, substituting \eqref{mt of bessel J}, this time with $\nu=\frac{s-1}{2},\ \mu=\frac{s-3}{2}, b=\frac{av}{u}$, in $I_1(a, s, w)$, gives for $1<\textup{Re}(s)<4$ and $a>0$,
\begin{align}\label{i1asw}
I_1(a, s, w)&=2^{\frac{s-3}{2}}a^{\frac{1-s}{2}}\Gamma\left(\frac{s-1}{2}\right)\Gamma\left(\frac{s}{2}\right)\frac{\sqrt{\pi} w}{4}\mathrm{erfi}\left(\frac{w}{2}\right){}_1F_1\left(\frac{3-s}{2};\frac{3}{2};-\frac{w^2}{4}\right)\nonumber\\
&=2^{\frac{s-7}{2}}a^{\frac{1-s}{2}}\Gamma\left(\frac{s-1}{2}\right)\Gamma\left(\frac{s}{2}\right)w^2A_{iw}(1-s),
\end{align}
where in the first step, we also used \eqref{int2} with $u$ replaced by $u/v$ and $z=s-1$. Finally from \eqref{mtrmasale}, \eqref{i2asw}, \eqref{i1asw} and simplifying using \eqref{dup} twice, we deduce \eqref{hm90} for $1<\textup{Re}(s)<2$ and $a>0$.
\end{proof}
As a corollary of the above result, we obtain a well-known representation of the Hurwitz zeta function \cite[p.~609, Formula \textbf{25.11.27}]{nist}. Before we get to it though, we need a lemma.
\begin{lemma}\label{useful}
For $x>0$. Let $\mathrm{Re}(s)>-1$, then we have
\begin{align}\label{usefuleqn} 
\frac{2^{\frac{s+3}{2}}}{\sqrt{\pi}(ax)^{\frac{s-1}{2}}}\int_0^\infty\int_0^\infty u^{\frac{s-1}{2}}v^{\frac{s+1}{2}} e^{-(u^2+v^2)}J_{\frac{s-1}{2}}\left(\frac{avx}{u}\right)\, dudv=
\begin{cases} 
 e^{-ax},  & \mathrm{Re}(a)>0 \\
      e^{ax}, & \mathrm{Re}(a)<0. \\
   \end{cases}
\end{align}
\end{lemma}
\begin{proof}
We first evaluate the inner integral over $u$. By change of variable $u=1/t$,
\begin{align}\label{int on u}
\int_0^\infty u^{\frac{s-1}{2}}e^{-u^2}J_{\frac{s-1}{2}}\left(\frac{avx}{u}\right)du=\int_0^\infty t^{-\frac{(s+3)}{2}}e^{-1/t^2}J_{\frac{s-1}{2}}(avxt)\ dt.
\end{align}
To evaluate the integral on the above right-hand side, we use \cite[p.~187, Formula \textbf{2.12.9.14}]{prudII} valid \footnote{This formula, which is given for $b>0$ in \cite{prudII}, can be seen to be true for Re$(b)>0$ by analytic continuation.} for $\mathrm{Re}(b)>0,\, \mathrm{Re}(c)>0$ and $\mathrm{Re}(\xi)<\frac{3}{2}$, namely, 
\begin{align}\label{mt of bessel j 1}
\int_0^\infty e^{-c/t^2}J_\nu(bt)t^{\xi-1}\ dt&=\frac{c^{\frac{\xi+\nu}{2}}b^\nu}{2^{\nu+1}}\frac{\Gamma\left(-\frac{(\xi+\nu)}{2}\right)}{\Gamma(\nu+1)}{}_0F_2\left(-;\nu+1,\frac{2+\xi+\nu}{2};\frac{cb^2}{4}\right)\nonumber\\
&\quad+\frac{2^{\xi-1}}{b^{\xi}}\frac{\Gamma\left(\frac{\xi+\nu}{2}\right)}{\Gamma\left(\frac{2-\xi+\nu}{2}\right)}{}_0F_2\left(-;\frac{2-\xi-\nu}{2},\frac{2-\xi+\nu}{2};\frac{cb^2}{4}\right).
\end{align}
Let $b=avx,\ c=1$ and $\nu=\frac{s-1}{2},\ \xi=-\frac{s+1}{2}$ in \eqref{mt of bessel j 1}, so that for Re$(a)>0$, $v>0, x>0$ and $\mathrm{Re}(s)>-4$,
{\allowdisplaybreaks\begin{align}\label{int 2 0f2}
\int_0^\infty t^{-\frac{(s+3)}{2}}e^{-1/t^2}J_{\frac{s-1}{2}}(avxt)\ dt&=\frac{(avx)^{\frac{s-1}{2}}}{2^{\frac{s+1}{2}}}\frac{\sqrt{\pi}}{\Gamma\left(\frac{s+1}{2}\right)}{}_0F_2\left(-;\frac{1}{2},\frac{s+1}{2};\frac{a^2v^2x^2}{4}\right)\nonumber\\
&\quad-\frac{(avx)^{\frac{s+1}{2}}}{2^{\frac{s+1}{2}}}\frac{\sqrt{\pi}}{\Gamma\left(1+\frac{s}{2}\right)}{}_0F_2\left(-;\frac{3}{2},1+\frac{s}{2};\frac{a^2v^2x^2}{4}\right),
\end{align}}
since $\Gamma(-1/2)=-2\sqrt{\pi}$. Hence from \eqref{int on u} and \eqref{int 2 0f2},
\begin{align}\label{double integral on u v}
&\int_0^\infty\int_0^\infty v^{\frac{s+1}{2}}u^{\frac{s-1}{2}}e^{-(u^2+v^2)}J_{\frac{s-1}{2}}\left(\frac{avx}{u}\right)dudv\nonumber\\ &=\frac{(ax)^{\frac{s-1}{2}}}{2^{\frac{s+1}{2}}}\frac{\sqrt{\pi}}{\Gamma\left(\frac{s+1}{2}\right)}\int_0^\infty v^se^{-v^2}{}_0F_2\left(-;\frac{1}{2},\frac{s+1}{2};\frac{a^2v^2x^2}{4}\right)\  dv\nonumber\\
&\qquad-\frac{(ax)^{\frac{s+1}{2}}}{2^{\frac{s+1}{2}}}\frac{\sqrt{\pi}}{\Gamma\left(1+\frac{s}{2}\right)}\int_0^\infty v^{s+1}e^{-v^2}{}_0F_2\left(-;\frac{3}{2},1+\frac{s}{2};\frac{a^2v^2x^2}{4}\right)\ dv.
\end{align} 
Employ change of variable $v=\sqrt{t}$ in both the integrals on the above right-hand side, and evaluate the resulting ones by appropriately specializing the identity \cite[p.~823, Formula \textbf{7.522.5}]{grn}
\begin{align}\label{mt of pfq}
\int_0^\infty e^{-x}x^{\xi-1}{}_pF_q(a_1,...a_p;b_1,...,b_q;cx)\ dx=\Gamma(\xi)\hspace{1mm}{}_{p+1}F_q(\xi,a_1,...a_p;b_1,...,b_q;c),
\end{align}
valid for $p<q$ and Re$(\xi)>0$. Simplify the resultant evaluations using ${}_0F_1\left(-;\frac{1}{2};\frac{a^2x^2}{4}\right)=\cosh(ax)$ and ${}_0F_1\left(-;\frac{3}{2};\frac{a^2x^2}{4}\right)=\frac{\sinh(ax)}{ax}$, and put them together to get, in view of \eqref{double integral on u v},
\begin{align*}
\int_0^\infty\int_0^\infty v^{\frac{s+1}{2}}u^{\frac{s-1}{2}}e^{-(u^2+v^2)}J_{\frac{s-1}{2}}\left(\frac{avx}{u}\right)dudv
=\sqrt{\pi}2^{-\frac{(s+3)}{2}}(ax)^{\frac{s-1}{2}}e^{-ax}
\end{align*}
for $\mathrm{Re}(a)>0$ and $\mathrm{Re}(s)>-1$. This proves \eqref{usefuleqn} for Re$(a)>0$. To derive the second part when Re$(a)<0$, just replace $a$ by $-a$ in the first, and use the well-known relation $J_\nu(-x)=(-1)^\nu J_\nu(x)$.
\end{proof}

\begin{corollary}\label{hurint}
For\footnote{This result actually holds for $\textup{Re}(a)>0$ and $\textup{Re}(s)>-1, s\neq 1$, by analytic continuation.} $a>0$ and $1<\textup{Re}(s)<2$,
\begin{align}\label{hurinteqn}
\zeta(s,a+1)&=-\frac{1}{2}a^{-s}+\frac{a^{1-s}}{s-1}+\frac{1}{\Gamma(s)}\int_0^\infty e^{-ax}\left(\frac{1}{e^x-1}-\frac{1}{x}+\frac{1}{2}\right)x^{s-1}\ dx,
\end{align}
\end{corollary}
\begin{proof}
Let $w\to 0$ on both sides of \eqref{hm90}. Since $a+1>1$, by Theorem \ref{hurwitzsc} and \eqref{zeta0sa}, the left-hand side becomes $\zeta(s, a+1)$. On the other hand, note that the triple integral on the right-hand side of \eqref{hm90} is absolutely convergent because of the presence of exponential function. Hence interchange the order of integration using Fubini's theorem so that the integration is first performed on $u$, then on $v$ and finally on $x$, then interchange the order of limit and integration, permissible due to the uniform convergence of the integral at the upper and lower limits in any compact subsets of the $w$-complex plane, and finally invoke Lemma \ref{useful} so as to arrive at the integral on the right-hand side of \eqref{hurinteqn}.
\end{proof}

\begin{proof}[Theorem \textup{\ref{hermite}}][]
The Legendre's integral for $t>0$ is given by \cite{riemann}
\begin{align}\label{legendre}
\int_0^\infty \frac{\sin(\pi yt)}{e^{\pi y}-1}\ dy=\frac{1}{e^{2\pi t}-1}-\frac{1}{2\pi t}+\frac{1}{2}.
\end{align}
Invoke the above integral evaluation in the first step below to observe that the triple integral in Theorem \ref{befhermite} can be written in the form
\begin{align}\label{hm10}
&\int_0^\infty\int_0^\infty\int_0^\infty u^{\frac{s-3}{2}}v^{\frac{s-1}{2}}e^{-(u^2+v^2)}\sin(wv)\sinh(wu)J_{\frac{s-1}{2}}\left(\frac{avx}{u}\right)x^{\frac{s-1}{2}}\left(\frac{1}{e^{x}-1}-\frac{1}{x}+\frac{1}{2}\right)\ dxdudv\nonumber\\
&=\int_0^\infty\int_0^\infty\int_0^\infty u^{\frac{s-3}{2}}v^{\frac{s-1}{2}}e^{-(u^2+v^2)}\sin(wv)\sinh(wu)J_{\frac{s-1}{2}}\left(\frac{avx}{u}\right)x^{\frac{s-1}{2}}\int_0^\infty \frac{\sin\left(\frac{xy}{2}\right)}{e^{\pi y}-1}\ dydxdudv\nonumber\\
&=\int_0^\infty\int_0^\infty u^{\frac{s-3}{2}}v^{\frac{s-1}{2}}e^{-(u^2+v^2)}\sin(wv)\sinh(wu)dudv\int_0^\infty\frac{dy}{e^{\pi y}-1}\int_0^\infty J_{\frac{s-1}{2}}\left(\frac{avx}{u}\right)x^{\frac{s-1}{2}}\sin\left(\frac{xy}{2}\right)dx,
\end{align}
by interchanging the order of integration, which is easily justified. Now from \cite[p.~738, Formula \textbf{6.699.1}]{grn}, for $-\mathrm{Re}(\nu)-1<1+\mathrm{Re}(\lambda)<\frac{3}{2}$,
\begin{align}\label{hm00}
\int_0^\infty x^\lambda J_\nu(cx)\sin(bx)\ dx&=
\begin{cases}
2^{1+\lambda}c^{-(2+\lambda)}b\frac{\Gamma\left(\frac{2+\lambda+\nu}{2}\right)}{\Gamma\left(\frac{\nu-\lambda}{2}\right)}{}_2F_1\left(\frac{2+\lambda+\nu}{2},\frac{2+\lambda-\nu}{2};\frac{3}{2};\frac{b^2}{c^2}\right),\hspace{1mm}\text{if}\hspace{1mm}0<b<c,\\
\left(\frac{c}{2}\right)^\nu b^{-(\nu+\lambda+1)}\frac{\Gamma(\nu+\lambda+1)}{\Gamma(\nu+1)}\sin\left(\pi\left(\frac{1+\lambda+\nu}{2}\right)\right),\hspace{15mm}\text{if}\hspace{1mm}0<c<b\\
\times{}_2F_1\left(\frac{2+\lambda+\nu}{2},\frac{1+\lambda+\nu}{2};\nu+1;\frac{c^2}{b^2}\right).
\end{cases}
\end{align}
Let $\nu=\frac{s-1}{2},\ c=\frac{av}{u},\ b=\frac{y}{2},\ \lambda=\frac{s-1}{2}$ in \eqref{hm00} so that for $-1<\mathrm{Re}(s)<2$,
\begin{align}\label{MT of bessel j and sin 2}
\int_0^\infty J_{\frac{s-1}{2}}\left(\frac{avx}{u}\right)x^{\frac{s-1}{2}}\sin\left(\frac{xy}{2}\right)dx&=
\begin{cases}
0,\hspace{3mm}\text{if}\hspace{1mm} 0<y<\frac{2av}{u},\\
\left(\frac{av}{2u}\right)^{\frac{s-1}{2}}\frac{y^{-s}\Gamma(s)\sin\left(\frac{\pi s}{2}\right)}{2^{-s}\Gamma\left(\frac{s+1}{2}\right)}{}_1F_0\left(\frac{s}{2};-;\frac{4a^2v^2}{u^2y^2}\right),\hspace{2mm} \text{if}\hspace{1mm}y>\frac{2av}{u}
\end{cases}\nonumber\\
&=\begin{cases}
0,\hspace{3mm}\text{if}\hspace{1mm} 0<y<\frac{2av}{u},\\
\frac{2^{\frac{3s-1}{2}}\sqrt{\pi}}{a^{\frac{1-s}{2}}\Gamma\left(1-\frac{s}{2}\right)}\frac{v^{\frac{s-1}{2}}u^{\frac{s+1}{2}}}{(u^2y^2-4a^2v^2)^{\frac{s}{2}}},\hspace{2mm} \text{if}\hspace{1mm}y>\frac{2av}{u},
\end{cases}
\end{align}
where in the last step, we used \eqref{dup}, \eqref{refl} and the generalized binomial theorem \cite[p.~108]{temme}. Thus from \eqref{hm10} and \eqref{MT of bessel j and sin 2},
\begin{align}\label{anot}
&\int_0^\infty\int_0^\infty\int_0^\infty u^{\frac{s-3}{2}}v^{\frac{s-1}{2}}e^{-(u^2+v^2)}\sin(wv)\sinh(wu)J_{\frac{s-1}{2}}\left(\frac{avx}{u}\right)x^{\frac{s-1}{2}}\left(\frac{1}{e^{x}-1}-\frac{1}{x}+\frac{1}{2}\right)\ dxdudv\nonumber\\
&=\frac{2^{\frac{3s-1}{2}}\sqrt{\pi}}{a^{\frac{1-s}{2}}\Gamma\left(1-\frac{s}{2}\right)}\int_0^\infty\int_0^\infty\int_{\frac{2av}{u}}^\infty \frac{(uv)^{s-1}e^{-(u^2+v^2)}\sin(wv)\sinh(wu)}{(e^{\pi y}-1)(u^2y^2-4a^2v^2)^{\frac{s}{2}}}\ dydudv\nonumber\\
&=\frac{2^{\frac{s+1}{2}}\sqrt{\pi}a^{\frac{1-s}{2}}}{\Gamma\left(1-\frac{s}{2}\right)}\int_0^\infty\int_0^\infty\int_{1}^\infty \frac{u^{s-2}e^{-(u^2+v^2)}\sin(wv)\sinh(wu)}{(e^{\frac{2\pi avt}{u}}-1)(t^2-1)^{\frac{s}{2}}}\ dtdudv\nonumber\\
&=\frac{2^{\frac{s+1}{2}}\sqrt{\pi}a^{\frac{1-s}{2}}}{\Gamma\left(1-\frac{s}{2}\right)}\int_0^\infty\int_1^\infty\int_{0}^\infty \frac{u^{s-2}e^{-(u^2+v^2)}\sin(wv)\sinh(wu)}{(e^{\frac{2\pi avt}{u}}-1)(t^2-1)^{\frac{s}{2}}}\ dvdtdu,
\end{align}
where, in the last step, we interchanged the order of integration using Fubini's theorem. This is now justified. 

For brevity, we define
\begin{align*}
f(u,v):=f(u, v, a, s, w):=\int_1^\infty \frac{u^{s-2}e^{-(u^2+v^2)}\sin(wv)\sinh(wu)}{\left(e^{\frac{2\pi avt}{u}}-1\right)(t^2-1)^{\frac{s}{2}}}\ dt.
\end{align*}
We first show
\begin{align}\label{fubini2}
\int_0^\infty\int_0^\infty f(u,v)\ dudv=\int_0^\infty\int_0^\infty f(u,v)\ dvdu.
\end{align}
Upon deriving \eqref{fubini2}, the interchange in the last step in \eqref{anot} will follow provided we show
\begin{align}\label{fubini6}
\int_0^\infty\int_1^\infty \frac{e^{-v^2}\sin(wv)}{\left(e^{\frac{2\pi avt}{u}}-1\right)(t^2-1)^{\frac{s}{2}}}\ dtdv=\int_1^\infty\int_0^\infty \frac{e^{-v^2}\sin(wv)}{\left(e^{\frac{2\pi avt}{u}}-1\right)(t^2-1)^{\frac{s}{2}}}\ dvdt.
\end{align}
To prove \eqref{fubini2}, it suffices to show
\begin{align*}
I_1:=\left|\int_0^\infty\int_0^\infty \int_1^\infty \frac{u^{s-2}e^{-(u^2+v^2)}\sin(wv)\sinh(wu)}{\left(e^{\frac{2\pi avt}{u}}-1\right)(t^2-1)^{\frac{s}{2}}}\ dtdudv\right|<\infty
\end{align*}
and 
\begin{align*}
I_2:=\left|\int_0^\infty\int_0^\infty \int_1^\infty \frac{u^{s-2}e^{-(u^2+v^2)}\sin(wv)\sinh(wu)}{\left(e^{\frac{2\pi avt}{u}}-1\right)(t^2-1)^{\frac{s}{2}}}\ dtdvdu\right|<\infty.
\end{align*}
Using $e^{\frac{2\pi a vt}{u}}-1\geq\frac{2\pi a vt}{u}$, we see that for $1<\mathrm{Re}(s)<2$,  
\begin{align*}
I_1&\leq\frac{1}{2\pi a}\int_0^\infty u^{\mathrm{Re}(s)-1}e^{-u^2}|\sinh(wu)|\ du \int_0^\infty v^{-1}e^{-v^2}|\sin(wv)|\ dv\int_1^\infty \frac{1}{t(t^2-1)^{\frac{\mathrm{Re}(s)}{2}}}\ dt\nonumber\\
&<\infty,
\end{align*}
as all integrals are finite. Similarly $I_2<\infty$. This implies \eqref{fubini2}.

To prove \eqref{fubini6}, we need to show
\begin{align*}
J_1:=\left|\int_0^\infty\int_1^\infty \frac{e^{-v^2}\sin(wv)}{\left(e^{\frac{2\pi avt}{u}}-1\right)(t^2-1)^{\frac{s}{2}}}\ dtdv\right|<\infty
\end{align*}
and 
\begin{align*}
J_2:=\left|\int_1^\infty\int_0^\infty \frac{e^{-v^2}\sin(wv)}{\left(e^{\frac{2\pi avt}{u}}-1\right)(t^2-1)^{\frac{s}{2}}}\ dvdt\right|<\infty.
\end{align*}
This follows from using $e^{\frac{2\pi a vt}{u}}-1\geq\frac{2\pi a vt}{u}$ again since
\begin{align*}
J_1\leq\frac{u}{2\pi a}\int_0^\infty v^{-1}e^{-v^2}|\sin(wv)|\ dv\int_1^\infty\frac{1}{t(t^2-1)^{\frac{\mathrm{Re}(s)}{2}}}\ dtdv<\infty,
\end{align*}
as $1<\mathrm{Re}(s)<2$. Similarly, $J_2<\infty$. This proves \eqref{fubini6}, and so the last step in \eqref{anot} follows from \eqref{fubini2} and \eqref{fubini6}.

Upon substituting \eqref{anot} in Theorem \ref{befhermite}, we arrive at \eqref{hermiteeqn}.
\end{proof}
Before we derive Hermite's formula as a special case of Theorem \ref{hermite}, we require a lemma which is interesting in itself since it requires Ramanujan's transformation \eqref{mrram} in its derivation.
\begin{lemma}\label{triple integral to one integral}
For\footnote{According to Remark \ref{validityoflemma4.9} below, this result is actually valid for $-1<\mathrm{Re}(s)<2$.} $1<\mathrm{Re}(s)<2$ and $a>0$, we have
\begin{align}\label{3in1}
\int_0^\infty\int_1^\infty\int_0^\infty \frac{u^{s-1}v\ e^{-(u^2+v^2)}}{(e^{\frac{2\pi avt}{u}}-1)(t^2-1)^{\frac{s}{2}}}\ dvdtdu=\frac{\Gamma(s)\Gamma\left(1-\frac{s}{2}\right)}{2^{s+1}a^{1-s}}\int_0^\infty \frac{\sin\left(s\tan^{-1}\left(\frac{y}{a}\right)\right)}{\left(e^{2\pi y}-1\right)(y^2+a^2)^{\frac{s}{2}}}\ dy.
\end{align}
\end{lemma}
\begin{proof}
To transform the innermost integral on the left-hand side of \eqref{3in1}, we invoke \eqref{mrram} with $\a=\frac{u}{\sqrt{\pi}at}$, $\a\b=1$, so as to obtain
\begin{align*}
\int_0^\infty \frac{x\ e^{-\frac{u^2x^2}{a^2t^2}}}{e^{2\pi x}-1}\ dx&=\frac{at}{4\sqrt{\pi}u}-\frac{a^2t^2}{4u^2}+\frac{\pi^{\frac{3}{2}}a^3t^3}{u^3}\int_0^\infty \frac{x\ e^{-(\pi atx)^2/u^2}}{e^{2\pi x}-1}\ dx.
\end{align*}
The change of variable $x=atv/u$ on the above left-hand side gives
\begin{align}\label{integral v2}
\int_0^\infty \frac{v\ e^{-v^2}}{e^{\frac{2\pi avt}{u}}-1}\ dv&=\frac{u}{4\sqrt{\pi}at}-\frac{1}{4}+\frac{\pi^{\frac{3}{2}}at}{u}\int_0^\infty \frac{x\ e^{-(\pi atx)^2/u^2}}{e^{2\pi x}-1}\ dx.
\end{align}
From \eqref{integral v2},
\begin{align}\label{triple integral1}
\int_0^\infty\int_1^\infty\int_0^\infty \frac{u^{s-1}v \ e^{-(u^2+v^2)}}{(e^{\frac{2\pi avt}{u}}-1)(t^2-1)^{\frac{s}{2}}}\ dvdtdu&=\mathfrak{I}_{1}(s,a)+\mathfrak{I}_{2}(s,a)+\mathfrak{I}_{3}(s,a),
\end{align}
where
\begin{align}\label{i1i2i3}
\mathfrak{I}_{1}(s,a)&:=\frac{1}{4\sqrt{\pi}a}\int_{0}^{\infty}\int_{1}^{\infty}\frac{u^se^{-u^2}}{t(t^2-1)^{\frac{s}{2}}}\, dtdu,\nonumber\\
\mathfrak{I}_{2}(s,a)&:=-\frac{1}{4}\int_{0}^{\infty}\int_{1}^{\infty}\frac{u^{s-1}e^{-u^2}}{(t^2-1)^{\frac{s}{2}}}\, dtdu,\nonumber\\
\mathfrak{I}_{3}(s,a)&:=\pi^{\frac{3}{2}}a\int_{0}^{\infty}\int_{1}^{\infty}\frac{tu^{s-2}e^{-u^2}}{(t^2-1)^{\frac{s}{2}}}\int_0^\infty \frac{x\ e^{-(\pi atx)^2/u^2}}{e^{2\pi x}-1}\ dxdtdu.
\end{align}
Since, for $\mathrm{Re}(s)>-1$,
\begin{align*}
\int_0^\infty u^{s}e^{-u^2}du&=\frac{1}{2}\Gamma\left(\frac{s+1}{2}\right),
\end{align*}
and for $0<\mathrm{Re}(s)<2$,
\begin{align}\label{integra210}
\int_1^\infty \frac{1}{t(t^2-1)^{\frac{s}{2}}}\ dt&=\frac{1}{2}\int_0^1 x^{\frac{s}{2}-1}(1-x)^{-\frac{s}{2}}\ dx=\frac{1}{2}\Gamma\left(\frac{s}{2}\right)\Gamma\left(1-\frac{s}{2}\right),
\end{align}
we find that for $0<\mathrm{Re}(s)<2$,
\begin{align}\label{double integral10}
\mathfrak{I}_{1}(s,a)=\frac{1}{2^{s+3}a}\G(s)\Gamma\left(1-\frac{s}{2}\right),
\end{align}
where we used \eqref{dup} for simplification. Similarly for $1<\mathrm{Re}(s)<2$,
\begin{align}\label{double integral11}
\mathfrak{I}_{2}(s,a)=-\frac{1}{2^{s+2}}\Gamma\left(1-\frac{s}{2}\right)\Gamma\left(s-1\right).
\end{align}
Applying Fubini's theorem, this time to the last integral in \eqref{i1i2i3}, we have
\begin{align}\label{tripletrouble}
\mathfrak{I}_{3}(s,a)=\pi^{\frac{3}{2}}a\int_{0}^{\infty}\frac{x\, dx}{e^{2\pi x}-1}\int_{1}^{\infty}\frac{t\, dt}{(t^2-1)^{\frac{s}{2}}}\int_{0}^{\infty}u^{s-2}e^{-u^2-(\pi atx)^2/u^2}\, du.
\end{align}
The formula \cite[p. 370, Formula \textbf{3.471.9}]{grn}
\begin{align*}
\int_0^\infty y^{\nu-1}e^{-\frac{\beta}{y}-\gamma y}\, dy=2\left(\frac{\beta}{\gamma}\right)^{\frac{\nu}{2}}K_\nu(2\sqrt{\beta\gamma}),
\end{align*}
valid for $\mathrm{Re}(\beta)>0$ and $\mathrm{Re}(\gamma)>0$, gives upon letting $\nu=\frac{s-1}{2},\ \gamma=1,\ \beta=(\pi atx)^2$ and employing change of variable $y=u^2$,
\begin{align}\label{bessel k integral}
\int_0^\infty u^{s-2}e^{-u^2}e^{-\frac{\pi^2a^2t^2x^2}{u^2}}\ du=(\pi atx)^{\frac{s-1}{2}}K_{\frac{s-1}{2}}(2a\pi tx).
\end{align}
Moreover, from \cite[p.~346, Formula \textbf{2.16.3.8}]{prudII}, we have for $r>0$, Re$(c)>0$ and Re$(b)>0$,
\begin{equation*}
\int_{r}^{\infty}t^{\nu+1}(t^2-r^2)^{b-1}K_{\nu}(ct)\, dt=2^{b-1}r^{b+\nu}c^{-b}\G(b)K_{\nu+b}(rc).
\end{equation*}
Letting $r=1$, $b=1-s/2$, $\nu=(s-1)/2$ and $c=2\pi ax$ in the above equation, we note that for Re$(s)<2$, $a>0$ and $x>0$,  
\begin{align}\label{K bessel integral}
\int_1^\infty \frac{t^{\frac{s+1}{2}}}{(t^2-1)^{\frac{s}{2}}}K_{\frac{s-1}{2}}(2a\pi tx)\ dt=\frac{1}{4}\pi^{\frac{s}{2}-1}(ax)^{\frac{s-3}{2}}e^{-2\pi a x}\Gamma\left(1-\frac{s}{2}\right).
\end{align}
From \eqref{tripletrouble}, \eqref{bessel k integral} and \eqref{K bessel integral}, we deduce that for $1<\textup{Re}(s)<2$,
\begin{align}\label{tripletrouble1}
\mathfrak{I}_{3}(s,a)&=\frac{1}{4}\pi^sa^{s-1}\G\left(1-\frac{s}{2}\right)\int_{0}^{\infty}\frac{x^{s-1}e^{-2\pi ax}}{e^{2\pi x}-1}\, dx\nonumber\\
&=\frac{a^{s-1}}{2^{s+2}}\G\left(1-\frac{s}{2}\right)\int_0^\infty x^{s-1}e^{-ax}\left(\frac{1}{e^{x}-1}-\frac{1}{x}+\frac{1}{2}\right)\ dx-\mathfrak{I}_{2}(s,a)-\mathfrak{I}_{1}(s,a),
\end{align}
as can be seen from \eqref{double integral10} and \eqref{double integral11}. Thus combining \eqref{triple integral1}, \eqref{tripletrouble1}, and then using \eqref{legendre} in the second step, we observe that for $1<\textup{Re}(s)<2$,
\begin{align}\label{3in1o}
\int_0^\infty\int_1^\infty\int_0^\infty \frac{u^{s-1}v\ e^{-(u^2+v^2)}dvdtdu}{(e^{\frac{2\pi avt}{u}}-1)(t^2-1)^{\frac{s}{2}}}&=\frac{\Gamma\left(1-\frac{s}{2}\right)}{2^{s+2}a^{1-s}}\int_0^\infty x^{s-1}e^{-ax}\left(\frac{1}{e^{x}-1}-\frac{1}{x}+\frac{1}{2}\right)\ dx\nonumber\\
&=\frac{\Gamma\left(1-\frac{s}{2}\right)}{2^{s+1}a^{1-s}}\int_0^\infty\int_0^\infty \frac{e^{-ax}x^{s-1}\sin(yx)}{e^{2\pi y}-1}\ dydx\nonumber\\
&=\frac{\Gamma\left(1-\frac{s}{2}\right)}{2^{s+1}a^{1-s}}\int_0^\infty\int_0^\infty \frac{e^{-ax}x^{s-1}\sin(yx)}{e^{2\pi y}-1}\ dxdy,
\end{align}
where in the last step we again applied Fubini's theorem. For $\mathrm{Re}(\mu)>-1$ and $\mathrm{Re}(\beta)>|\mathrm{Im}(\delta)|$, we have \cite[p.~502, Formula \textbf{3.944.5}]{grn}
\begin{align*}
\int_0^\infty e^{-\beta x}x^{\mu-1}\sin(\delta x)dx=\Gamma(\mu)\frac{\sin\left(\mu\tan^{-1}(\delta/\beta)\right)}{(\beta^2+\delta^2)^{\frac{s}{2}}}.
\end{align*}
Set $\mu=s,\ \beta=a$ and $\delta=y$ in the above identity, then substitute the resulting evaluation for the inner integral over $x$ on the right-hand side of \eqref{3in1o} and simplify to arrive at \eqref{3in1}.
\end{proof}
With the help of the lemma just proved, it is straightforward to obtain Hermite's formula \eqref{hermiteor} as a special case of Theorem \ref{hermite}.
\begin{proof}[\textup{\eqref{hermiteor}}][]
Let $w\rightarrow 0$ in Theorem \ref{hermite}, interchange the order of limit and triple integral and then use Lemma \ref{triple integral to one integral} and \eqref{zeta0sa}.
\end{proof}

\subsection{A relation between $\zeta_w(s, a)$ and ${}_1K_{z,w}(x)$}\label{relation}
\hfill\\

The two new special functions introduced in this paper, namely $\zeta_w(s, a)$ and ${}_1K_{z,w}(x)$, are related by a nice relation given in Theorem \ref{relation b/w zetaw and 1Kzw} which we now prove.
\begin{proof}[Theorem \textup{\ref{relation b/w zetaw and 1Kzw}}][]
We first transform the triple integral in \eqref{hermiteeqn}. To do this, however, we need the following generalization of the first equality in \eqref{mrram} which was established in \cite[Theorem 1.1]{drz5}, and which is analogous to \eqref{genthetatr}, that is, it is a transformation of the form $F(w, \a)=F(iw, \b)$:
\begin{align}\label{mrramg}
&\sqrt{\alpha}e^{\frac{w^2}{8}}\left(\frac{\sqrt{\pi}}{w}\textup{erf}\left(\frac{w}{2}\right)-\frac{4\sqrt{\pi}}{w}\int_{0}^{\infty}\frac{e^{-\pi\alpha^2 x^2}\sin(\sqrt{\pi}\alpha xw)}{e^{2\pi x}-1}\, dx\right)\nonumber\\
&=\sqrt{\beta}e^{\frac{-w^2}{8}}\left(\frac{\sqrt{\pi}}{w}\textup{erfi}\left(\frac{w}{2}\right)-\frac{4\sqrt{\pi}}{w}\int_{0}^{\infty}\frac{e^{-\pi\beta^2 x^2}\sinh(\sqrt{\pi}\beta xw)}{e^{2\pi x}-1}\, dx\right).
\end{align}
Let $\a=\frac{u}{\sqrt{\pi}at}$ in \eqref{mrramg}, simplify, and then employ change of variable $x=atv/u$ in the integral containing sine function so that
\begin{align}\label{dixRoyZah}
\int_0^\infty \frac{e^{-v^2}\sin\left(wv\right)}{e^{2\pi atv/u}-1} \ dv&=\frac{u}{4at}\mathrm{erf}\left(\frac{w}{2}\right)-\frac{\sqrt{\pi} }{4}e^{-\frac{w^2}{4}}\mathrm{erfi}\left(\frac{w}{2}\right)\nonumber\\
&\quad+\sqrt{\pi}\ e^{-\frac{w^2}{4}}\int_0^\infty \frac{e^{-(\pi atx/u)^2}\sinh\left(\frac{\pi atxw}{u}\right)}{e^{2\pi x}-1}\ dx.
\end{align}
Substituting \eqref{dixRoyZah} for the innermost integral of the triple integral in \eqref{hermiteeqn}, invoking \eqref{equivv} with appropriate specializations, \eqref{integra210} and the corresponding evaluation for $\int_{1}^{\infty}(t^2-1)^{-s/2}\, dt$ to evaluate the two double integrals on the right-hand side, and then simplifying, we see that for $1<\textup{Re}(s)<2$,
\begin{align}\label{dt}
&\frac{1}{w^2}\int_0^\infty\int_1^\infty\int_{0}^\infty \frac{u^{s-2}e^{-(u^2+v^2)}\sin(wv)\sinh(wu)}{(e^{\frac{2\pi avt}{u}}-1)(t^2-1)^{\frac{s}{2}}}\ dvdtdu\nonumber\\
&=\frac{1}{2^{s+3}a}\G(s)\G\left(1-\frac{s}{2}\right)A_w(s-1)-\frac{1}{2^{s+2}}\G(s-1)\G\left(1-\frac{s}{2}\right)A_{iw}(1-s)\nonumber\\
&\quad+\frac{\sqrt{\pi}\ e^{-\frac{w^2}{4}}}{w^2}\int_0^\infty\int_1^\infty\int_0^\infty\frac{u^{s-2}e^{-u^2-(\pi atx/u)^2}\sinh(wu)\sinh\left(\frac{\pi atxw}{u}\right)}{(t^2-1)^{\frac{s}{2}}(e^{2\pi x}-1)}\, dxdtdu,
\end{align}
where $A_w(z)$ is defined in \eqref{ab}. In the triple integral on the above right-hand side, apply Fubini's theorem again so as to integrate first on $u$, then on $x$ and finally on $t$. Then evaluate the innermost integral on $u$ by letting $z=(s-1)/2$ in Theorem \ref{integralRepr} as well as replacing $w$ and $x$ by $iw$ and $\pi a t x$ respectively so that
\begin{align*}
&\int_0^\infty\int_1^\infty\int_0^\infty\frac{u^{s-2}e^{-u^2-(\pi atx/u)^2}\sinh(wu)\sinh\left(\frac{\pi atxw}{u}\right)}{(t^2-1)^{\frac{s}{2}}(e^{2\pi x}-1)}\, dxdtdu\nonumber\\
&=\frac{w^2(\pi a)^{\frac{s-1}{2}}}{2}\int_1^\infty\int_0^\infty\frac{(tx)^{\frac{s-1}{2}}{}_1K_{\frac{s-1}{2},iw}(2\pi atx)}{(t^2-1)^{\frac{s}{2}}(e^{2\pi x}-1)}\, dxdt
\end{align*}
Now substitute the above equation in \eqref{dt}, and the resulting equation, in turn, in Theorem \ref{hermite} to arrive at Theorem \ref{relation b/w zetaw and 1Kzw}.
\end{proof}
Theorem \ref{relation b/w zetaw and 1Kzw} gives the well-known Mellin transform representation of $\zeta(s, a)$ as its special case.
\begin{corollary}\label{getsplcor}
For $\textup{Re}(a)>0$ and $\textup{Re}(s)>1$,
\begin{equation}\label{getspl}
\zeta(s, a+1)=\frac{1}{\G(s)}\int_{0}^{\infty}\frac{x^{s-1}e^{-ax}}{e^{x}-1}\, dx.
\end{equation}
\end{corollary}
\begin{proof}
Interchange the order of integration using Fubini's theorem, then let $w\to 0$ in Theorem \ref{relation b/w zetaw and 1Kzw}, interchange the order of limit and integration, apply \eqref{K bessel integral} to evaluate the resulting inner integral, simplify and then employ change of variable $x\to x/(2\pi)$ in the last step so as to get \eqref{getspl} for $1<\textup{Re}(s)<2$ and $a>0$. The result now follows for all $\textup{Re}(s)>1$ and $\textup{Re}(a)>0$ by analytic continuation.
\end{proof}



\section{Analytic continuation of $\zeta_w(s, a)$}\label{lseproof}

\subsection{Analyticity of a triple integral in $-1<\textup{Re}(s)<2$}\label{anallemma}
\hfill\\

Before proving Theorem \ref{lse}, we prove a crucial lemma.
\begin{lemma}\label{tia}
Let $w\in\mathbb{C}$ and $a>0$ be fixed. Then the triple integral 
\begin{equation}\label{ti}
\int_0^\infty\int_1^\infty\int_{0}^\infty \frac{u^{s-2}e^{-(u^2+v^2)}\sin(wv)\sinh(wu)}{(e^{\frac{2\pi avt}{u}}-1)(t^2-1)^{\frac{s}{2}}}\ dvdtdu
\end{equation}
is an analytic function of $s$ in $-1<\textup{Re}(s)<2$.
\end{lemma}
\begin{proof}
This is achieved by a repeated application of Theorem \ref{ldcttemme1}. We first show that the integral
\begin{equation*}
\int_0^\infty\int_1^\infty\int_{0}^\infty \frac{u^{s-2}e^{-(u^2+v^2)}\sin(wv)\sinh(wu)}{(e^{\frac{2\pi avt}{u}}-1)(t^2-1)^{\frac{s}{2}}}\ dvdtdu
\end{equation*}
exists for $-1<\textup{Re}(s)<2$. (Note that Theorem \ref{hermite} implies, in particular, that the above integral exists, but only in $1<\textup{Re}(s)<2$.)

Note that from \cite[Theorem 1.4]{drz5}, as $\a\to 0$,
\begin{align*}
\frac{4}{w}\alpha^{1/4}e^{-w^2/8}\int_0^\infty\frac{e^{-\alpha x^2}\sinh(\sqrt{\alpha}xw)}{e^{2\pi x}-1}\ dx\sim\frac{1}{6}e^{-w^2/8}\alpha^{3/4}.
\end{align*}
Replacing $w$ by $iw$, we see that as $\a\to 0$,
\begin{align}\label{befb}
\int_0^\infty\frac{e^{-\alpha x^2}\sin(\sqrt{\alpha}xw)}{e^{2\pi x}-1}\ dx&\sim\frac{w\sqrt{\alpha}}{24}.
\end{align}
Let $\sqrt{\a}=u/(at)$ in \eqref{befb}, where $u$ and $a$ are fixed. Then employ the change of variable $x=atv/u$ in the integral on the resulting left-hand side so that as $t\to\infty$,
\begin{align}\label{bound for a}
\int_0^\infty \frac{e^{-v^2}\sin(wv)}{e^{\frac{2\pi avt}{u}}-1}\ dv\sim\frac{wu^2}{24a^2t^2}.
\end{align}
which implies, in particular, that for $t$ large enough,
\begin{align}\label{bound for a'}
\int_0^\infty \frac{e^{-v^2}\sin(wv)}{e^{\frac{2\pi avt}{u}}-1}\ dv\leq\frac{c|w|u^2}{24a^2t^2}
\end{align}
for a positive constant $c$. Now let $M$ be large enough. Then
{\allowdisplaybreaks\begin{align*}
&\left|\int_0^\infty\int_1^\infty\int_0^\infty \frac{u^{s-2}e^{-(u^2+v^2)}\sin(wv)\sinh(wu)}{\left(e^{\frac{2\pi avt}{u}}-1\right)(t^2-1)^{\frac{s}{2}}}\ dvdtdu\right|\\
&\leq\int_0^\infty\int_1^M\int_0^\infty \frac{u^{\mathrm{Re}(s)-2}e^{-(u^2+v^2)}|\sin(wv)\sinh(wu)|}{\left(e^{\frac{2\pi avt}{u}}-1\right)(t^2-1)^{\frac{\mathrm{Re}(s)}{2}}}\ dvdtdu\\
&\quad+\int_0^\infty\int_M^\infty\int_0^\infty \frac{u^{\mathrm{Re}(s)-2}e^{-(u^2+v^2)}|\sin(wv)\sinh(wu)|}{\left(e^{\frac{2\pi avt}{u}}-1\right)(t^2-1)^{\frac{\mathrm{Re}(s)}{2}}}\ dvdtdu\\
&\leq\frac{1}{2\pi a}\int_0^\infty u^{\mathrm{Re}(s)-1}e^{-u^2}|\sinh(wu)|\ du\int_1^M\frac{1}{t(t^2-1)^{\frac{\mathrm{Re}(s)}{2}}}\ dt\int_0^\infty v^{-1}e^{-v^2} |\sin(wv)|\ dv\\
&\quad+\frac{c|w|}{24a^2}\int_0^\infty u^{\mathrm{Re}(s)}e^{-u^2}|\sinh(wu)|\ du\int_M^\infty \frac{1}{t^2(t^2-1)^{\frac{\mathrm{Re}(s)}{2}}}\ dt.
\end{align*}}
where in the last step we used \eqref{bound for a'} as well as the fact that $e^{\frac{2\pi avt}{u}}-1\geq \frac{2\pi avt}{u}$. Note that all of the integrals in the last expression exist and are finite as long as $-1<\mathrm{Re}(s)<2$. Hence the integral in \eqref{ti} exists for $-1<\mathrm{Re}(s)<2$.

We now show that the triple integral in \eqref{ti} is analytic in $-1<\mathrm{Re}(s)<2$ too. This is done in several steps below.\\

(A) \textbf{Analyticity of $\displaystyle\int_1^\infty\int_0^\infty\frac{e^{-v^2}\sin(wv)}{\left(e^{\frac{2\pi avt}{u}}-1\right)(t^2-1)^{\frac{s}{2}}}\ dvdt$ in $-1<\textup{Re}(s)<2$:} 
Let 
\begin{align*}
f(s,t):=\frac{1}{(t^2-1)^{\frac{s}{2}}}\int_0^\infty\frac{e^{-v^2}\sin(wv)}{e^{\frac{2\pi avt}{u}}-1}\ dv.
\end{align*}
Clearly, $f(s,t)$ is continuous in $s$ on $-1<\mathrm{Re}(s)<2 $. We wish to show that $f(s, t)$ is also continuous in $t$ on $(1,\infty)$ so that hypothesis (i) of Theorem \ref{ldcttemme1} is satisfied. It suffices to show that $\int_0^\infty\frac{e^{-v^2}\sin(wv)}{e^{\frac{2\pi avt}{u}}-1}\ dv$ is continuous function of $t$ on $(1,\infty)$. This is done by appealing to Theorem \ref{titchcomplexthm}. Consider the rectangle given by $0\leq v\leq b,\ 1\leq t\leq D$. It is easy to see that $\frac{e^{-v^2}\sin(wv)}{e^{\frac{2\pi avt}{u}}-1}$ is continuous on the rectangle. Next the uniform convergence of the integral $\int_0^\infty\frac{e^{-v^2}\sin(wv)}{e^{\frac{2\pi avt}{u}}-1}\ dv$ on $(1, D)$ is proved, that is, given $\epsilon>0$, we find an $X_0$ depending only on $\epsilon$ such that for all $t\in(1,D)$ and $X\geq X_0$, 
\begin{align}\label{uc10}
\left|\int_0^\infty\frac{e^{-v^2}\sin(wv)}{e^{\frac{2\pi avt}{u}}-1}\ dv-\int_0^X \frac{e^{-v^2}\sin(wv)}{e^{\frac{2\pi avt}{u}}-1}dv\right|<\epsilon.
\end{align} 
To that end, observe that
\begin{align*}
\left|\int_0^\infty\frac{e^{-v^2}\sin(wv)}{e^{\frac{2\pi avt}{u}}-1}\ dv-\int_0^X \frac{e^{-v^2}\sin(wv)}{e^{\frac{2\pi avt}{u}}-1}dv\right|&<\int_X^\infty \frac{e^{-v^2}
|\sin(wv)|}{e^{\frac{2\pi avt}{u}}-1}dv\\
&<\frac{u}{2\pi at}\int_X^\infty e^{-v^2}\frac{|\sin(wv)|}{v}\ dv\\
&<\frac{uA_w}{2\pi a}\int_X^\infty e^{-v^2/2}dv,
\end{align*}
where, the fact that $e^{-v^2/2}|\sin(wv)|v^{-1}$ is continuous and goes to zero as $v\rightarrow\infty$ implies that there exists a constant $A_w$ such that $|e^{-v^2/2}\sin(wv)v^{-1}|<A_w$ for all large enough $v$. Therefore, one can always choose $X_0$ large enough so that for all $X\geq X_0$,
\begin{align*}
\int_X^\infty e^{-v^2/2}dv<\epsilon.\frac{2\pi a}{uA_w}.
\end{align*}
This implies \eqref{uc10} and hence the uniform convergence of $\int_0^\infty\frac{e^{-v^2}\sin(wv)}{e^{\frac{2\pi avt}{u}}-1}\ dv$ with respect to $t$ in the interval $[1,D]$. Hence by Theorem \ref{titchcomplexthm}, $f(s,t)$ is continuous on $(1,D)$ for every $D>1$, and hence on $(1,\infty)$. Thus, part (i) of Theorem \ref{ldcttemme1} is satisfied.

Part (ii) of Theorem \ref{ldcttemme1} follows easily as $f(s,t)$ is analytic in $-1<\mathrm{Re}(s)<2$ for every fixed $t\in(1, \infty)$.

To prove part (iii) of Theorem \ref{ldcttemme1}, consider the compact subset $\Omega$ of $-1<\textup{Re}(s)<2$ given by 
\begin{equation}\label{compact}
\Omega:=\{s\in\mathbb{C}: -1<\theta_1\leq\textup{Re}(s)\leq\theta_2<2, \gamma_1\leq\textup{Im}(s)\leq\gamma_2\}.
\end{equation}

Let $\epsilon>0$ be given. Let $1<A<B<1+\tau$ for $\tau$ a small positive number. Then
\begin{align*}
\left|\int_A^Bf(s,t)\ dt\right|&\leq\int_A^B\frac{1}{(t^2-1)^{\theta_2/2}}\frac{u}{2\pi at}\int_0^\infty e^{-v^2}\frac{|\sin(wv)|}{v}\ dv\nonumber\\
&\leq\frac{uB_w}{2\pi a}\int_A^B\frac{1}{t(t^2-1)^{\theta_2/2}}\ dt
\end{align*}
for all $s\in\Omega$, where $B_w$ is a positive constant depending on $w$. Thus we can choose a $\delta, 0<\delta<1,$ such that $\frac{uB_w}{2\pi a}\int_A^B\frac{1}{t(t^2-1)^{\theta_2/2}}\ dt<\epsilon$ for $|A-B|<\delta$. This proves that integral $\int_1^\infty\ f(s,t)dt$ converges uniformly at the lower limit on the compact subset $\Omega$ of $-1<\mathrm{Re}(s)<2$. 

The uniform convergence at the upper limit is now shown. From \eqref{bound for a}, we have \newline $\int_0^\infty\frac{e^{-v^2}\sin(wv)}{e^{\frac{2\pi avt}{u}}-1}\ dv\leq\frac{C_{a, u, w}}{t^{2}}$ for a positive constant $C_{a, u, w}$ depending on $a, u$ and $w$. Hence for $B>A>\kappa$, where $\kappa$ is large enough,
\begin{align*}
\left|\int_A^B f(s,t)\ dt\right|&\leq\int_A^B\frac{1}{(t^2-1)^{\theta_1/2}}\int_0^\infty\frac{e^{-v^2}\sin(wv)}{e^{\frac{2\pi avt}{u}}-1}\ dvdt\nonumber\\
&\leq C_{a, u, w}\int_A^B\frac{1}{t^2(t^2-1)^{\theta_1/2}}\ dt
\end{align*}
for all $s\in\Omega$. Thus, given an $\epsilon>0$, there is an number $\kappa>1$ such that $|\int_A^B\frac{1}{t^2(t^2-1)^{\theta_1}}\ dt|<\epsilon/C_{a, u, w}$ whenever $B>A>\kappa$. 

By invoking Theorem \ref{ldcttemme1}, we thus conclude that $\int_{1}^{\infty}f(s, t)\, dt$ is analytic in $-1<\textup{Re}(s)<2$.

Let 
\begin{align}\label{def g}
g(s,u):=u^{s-2}e^{-u^2}\sinh(wu)\int_1^\infty\int_0^\infty\frac{1}{(t^2-1)^{\frac{s}{2}}}\frac{e^{-v^2}\sin(wv)}{\left(e^{\frac{2\pi avt}{u}}-1\right)}\ dvdt.
\end{align}

(B) \textbf{Analyticity of $\int_{0}^{\infty}g(s, u)\, du$ in $-1<\textup{Re}(s)<2$:}
 We use Theorem \ref{ldcttemme1} again. Note that (A) implies that $g(s,u)$ is continuous function of $s$ in $-1<\mathrm{Re}(s)<2$. We want to show that $g(s,u)$ is also continuous in $u$ on $(0,\infty)$ so that hypothesis (i) of Theorem \ref{ldcttemme1} is satisfied. To show this we will use Theorem \ref{titchcomplexthm}. However, this requires $\frac{u^{s-2}e^{-u^2}}{(t^2-1)^{\frac{s}{2}}}\int_0^\infty\frac{e^{-v^2}\sin(wv)\sinh(wu)}{e^{\frac{2\pi avt}{u}}-1}\ dv$ to be continuous on the rectangle $0\leq u\leq c,\ 1\leq t\leq d$ for all values of $c$. But the function is not continuous at the point $(u,t)=(0,1)$. This problem is now circumvented by first proving with the help of Theorem \ref{titchcomplexthm} that for any $n\in\mathbb{N}$,
\begin{align*}
g_n(s,u):=\int_{1+\frac{1}{n}}^\infty \int_0^\infty\frac{e^{-v^2}\sin(wv)\sinh(wu)}{(t^2-1)^{\frac{s}{2}}\left(e^{\frac{2\pi avt}{u}}-1\right)}\ dvdt
\end{align*}
is continuous for $u\in(0,c)$ for all values of $c$ and then showing that the sequence $\{g_n(s,u)\}_{n=1}^{\infty}$ converges uniformly to $\int_{1}^\infty \int_0^\infty\frac{e^{-v^2}\sin(wv)\sinh(wu)}{(t^2-1)^{\frac{s}{2}}\left(e^{\frac{2\pi avt}{u}}-1\right)}\ dvdt$, whence the latter will be continuous function of $u$ in $(0, c)$.

Consider the rectangle given by $0\leq u\leq c,\ 1+\frac{1}{n}\leq t\leq d$. To show the continuity of $\frac{1}{(t^2-1)^{\frac{s}{2}}}\int_0^\infty\frac{e^{-v^2}\sin(wv)\sinh(wu)}{\left(e^{\frac{2\pi avt}{u}}-1\right)} \ dv$ on the rectangle, we apply Theorem \ref{multivariate}. However, this requires the existence of partial derivatives of the above function with respect to $u$ and $t$ as well as the boundedness of one of the partial derivatives. Note that it suffices to show the existence of the partial derivatives of $\int_0^\infty\frac{e^{-v^2}\sin(wv)\sinh(wu)}{\left(e^{\frac{2\pi avt}{u}}-1\right)} \ dv$, for then, the existence of the partial derivatives of $\frac{1}{(t^2-1)^{\frac{s}{2}}}\int_0^\infty\frac{e^{-v^2}\sin(wv)\sinh(wu)}{\left(e^{\frac{2\pi avt}{u}}-1\right)} \ dv$ follows obviously. To show the existence of partial derivatives of the former with respect to $u$ and $t$, we need to apply Theorem \ref{lang} twice. 

Let 
\begin{align*}
f_{t, u}(v):=\frac{e^{-v^2}\sin(wv)\sinh(wu)}{e^{\frac{2\pi avt}{u}}-1}.
\end{align*}
To meet the hypotheses of Theorem \ref{lang}, apart from the continuity of $f_{t, u}(v), \frac{\partial }{\partial t}f_{t, u}(v)$ and $\frac{\partial }{\partial u}f_{t, u}(v)$, we need to find the functions
$\phi_u(v), \psi_u(v)$ and $\tilde{\phi}_t(v), \tilde{\psi}_t(v)$ such that for all $v>0$ $|f_{t, u}(v)|\leq \phi_u(v)$, $|\frac{\partial }{\partial u}f_{t, u}(v)|\leq \psi_u(v)$ for $u\in[0, c]$, and $|f_{t, u}(v)|\leq \tilde{\phi}_t(v)$, $|\frac{\partial }{\partial t}f_{t, u}(v)|\leq \tilde{\psi}_t(v)$  for $t\in[1+1/n, d]$, and such that the integrals $\int_{0}^{\infty}\phi_u(v)\, dv$, $\int_{0}^{\infty}\psi_u(v)\, dv$, $\int_{0}^{\infty}\tilde{\phi}_t(v)\, dv$ and $\int_{0}^{\infty}\tilde{\psi}_t(v)\, dv$ converge. It is tedious, albeit not difficult, to see that the following choices of $\phi_u(v), \psi_u(v), \tilde{\phi}_t(v)$ and $\tilde{\psi}_t(v)$ work:
\begin{align*}
\phi_u(v)&:=\frac{u}{2\pi av(1+1/n)}e^{-v^2}|\sinh(wu)\sin(wv)|,\nonumber\\
\psi_u(v)&:=\frac{e^{-v^2}}{u}|\sin(wv)\sinh(wu)|+\frac{e^{-v^2}|\sin(wv)\sinh(wu)|}{2\pi av\left(1+1/n\right)}+\frac{ue^{-v^2}|w\sin(wv)|\cosh(|w|u)}{2\pi av\left(1+1/n\right)},\nonumber\\
\tilde{\phi}_t(v)&:=\frac{c}{2\pi atv}e^{-v^2}|\sin(wv)|\sinh(|w|c),\nonumber\\
\tilde{\psi}_t(v)&:=\frac{1}{t}e^{-v^2}|\sin(wv)|\sinh(|w|c)+\frac{c}{2\pi at^2}e^{-v^2}|\sin(wv)|\sinh(|w|c).
\end{align*}
Also, one can check that $\frac{\partial }{\partial u}f_{t, u}(v)$ is bounded on the rectangle $0\leq u\leq c,\ 1+\frac{1}{n}\leq t\leq d$. The details of this verification are left to the reader. Thus we conclude that $\frac{1}{(t^2-1)^{\frac{s}{2}}}\int_0^\infty\frac{e^{-v^2}\sin(wv)\sinh(wu)}{\left(e^{\frac{2\pi avt}{u}}-1\right)} \ dv$ is continuous on the above rectangle. Thus, the first hypothesis of Theorem \ref{titchcomplexthm} is satisfied. To see that the other holds too, we now show that $g_n(s,u)$ converges uniformly with respect to $u$ in $(0,c)$. Thus, given an $\epsilon>0$, we find an $X_0$ (depending only on $\epsilon$) such that for all $u\in(0,c)$ and $X\geq X_0$,
\begin{align}\label{uc12}
&\left|g_n(s,u)-\int_{1+\frac{1}{n}}^X\int_0^\infty \frac{e^{-v^2}\sinh(wu)\sin(wv)}{(t^2-1)^{\frac{s}{2}}\left(e^{\frac{2\pi avt}{u}}-1\right)}\ dvdt\right|<\epsilon.
\end{align}
To that end, 
\begin{align*}
\left|\int_{X}^{\infty}\int_0^\infty \frac{e^{-v^2}\sinh(wu)\sin(wv)}{(t^2-1)^{\frac{s}{2}}\left(e^{\frac{2\pi avt}{u}}-1\right)}\ dvdt\right|
&\leq\frac{u|\sin(wu)|}{2\pi a}\int_X^\infty \frac{1}{t(t^2-1)^{\frac{\textup{Re}(s)}{2}}}\ dt\int_0^\infty \frac{e^{-v^2}|\sin(wv)|}{v}\ dv\nonumber\\
&\leq \frac{cD_w}{2\pi a}\int_X^\infty \frac{1}{t(t^2-1)^{\frac{\textup{Re}(s)}{2}}}\ dt,
\end{align*}
where $D_w$ is constant and depends on $w$. Note that we can choose $X_0$ large enough so that for all $X\geq X_0$, 
\begin{align*}
\int_X^\infty \frac{1}{t(t^2-1)^{\frac{\textup{Re}(s)}{2}}}\ dt<\epsilon.\frac{2\pi a}{cD_w}.
\end{align*}
This implies \eqref{uc12} and hence the uniform convergence of $g_n(s,u)$ with respect to $u$ in $(0,c)$. Hence by Theorem \ref{titchcomplexthm}, $g_n(s,u)$ is continuous on $(0,c)$ for every $c>0$. Next, we show $g_n(s,u)$ converges uniformly to 
\begin{equation*}
\mathscr{G}(s,u):=\int_{1}^\infty\int_0^\infty\frac{e^{-v^2}\sinh(wu)\sin(wv)}{(t^2-1)^{\frac{s}{2}}\left(e^{\frac{2\pi avt}{u}}-1\right)}\ dvdt
\end{equation*}
on $(0,c)$ whence $\mathscr{G}(s,u)$ will be continuous in $u$ in $(0, c)$ for every $c>0$. To do this, we need to show that given an $\epsilon>0$, there exists $n_0$ which depends only on $\epsilon$ (and not on $u$) such that for all $n\geq n_0$,  
\begin{align}\label{uc13}
|g_n(s,u)-\mathscr{G}(s,u)|<\epsilon.
\end{align}
This indeed is true since
\begin{align*}
|g_n(s,u)-\mathscr{G}(s,u)|
&\leq \int_1^{1+1/n}\int_0^\infty\frac{e^{-v^2}|\sinh(wu)\sin(wv)|}{(t^2-1)^{\frac{\textup{Re}(s)}{2}}\left(e^{\frac{2\pi avt}{u}}-1\right)}\ dvdt\nonumber\\
&\leq  \frac{u|\sinh(wu)|}{2\pi a}\int_1^{1+1/n}\frac{1}{t(t^2-1)^{\frac{\textup{Re}(s)}{2}}}\ dt\int_0^\infty \frac{e^{-v^2}|\sin(wv)|}{v}\ dv \nonumber\\
&\leq F_{w,c,a}\int_1^{1+1/n}\frac{1}{t(t^2-1)^{\frac{\textup{Re}(s)}{2}}}\ dt,
\end{align*}
where $F_{w,c,a}$ is a positive constant depending on $w, c$ and $a$. It is now clear that one can choose $n_0$ large enough such that \eqref{uc13} holds. Therefore, by \cite[p.~150, Theorem 7.12]{rudin}, we see that $\mathscr{G}(s,u)$ is continuous in $u\in(0,c)$ for every $c>0$ and hence on $(0,\infty)$. This, in turn, implies the continuity of $g(s, u)$, defined in \eqref{def g}, as a function of $u$, on $(0, \infty)$. Thus hypothesis (i) of Theorem \ref{ldcttemme1} holds. Clearly, (ii) of Theorem \ref{ldcttemme1} holds since for any fixed $u\in(0, \infty)$, $g(s, u)$ is analytic in $-1<\textup{Re}(s)<2$ as can be seen from part (A). Thus we need only show that $\int_0^\infty g(s,u)\ du$ converges uniformly at both the limits on $\Omega$ defined in \eqref{compact}. This is now proved.

Let $\epsilon>0$ be given. Let $0<A<B<1$. Using \eqref{bound for a} in the second step below with $a$ and $t$ fixed and $u$ very small, we have
\begin{align*}
\left|\int_A^B g(s,u)\ du\right|&\leq \int_A^Bu^{\theta_1-2}e^{-u^2}|\sinh(wu)|\int_1^\infty\frac{1}{(t^2-1)^{\textup{Re}(s)/2}}\int_0^\infty\frac{e^{-v^2}|\sin(wv)|}{e^{\frac{2\pi avt}{u}}-1}\ dvdtdu\nonumber\\
&\leq c_1\int_A^B u^{\theta_1-2}e^{-u^2}|\sinh(wu)|\int_1^\infty \frac{1}{(t^2-1)^{\frac{1}{2}\textup{Re}(s)}}\frac{|w|u^2}{24a^2t^2}\ dtdu\nonumber\\
&\leq \frac{c_1|w|}{24a^2}\int_A^B u^{\theta_1}e^{-u^2}|\sinh(wu)|\ du\left(\int_1^{\sqrt{2}}\frac{1}{t^2(t^2-1)^{\theta_2/2}}+\int_{\sqrt{2}}^{\infty}\frac{1}{t^2(t^2-1)^{\theta_1/2}}\right)\nonumber\\
&\leq \frac{G}{24a^2}\int_A^B u^{\theta_1+1}\ du\nonumber\\
&\leq \frac{G}{24a^2}\left(B^{\theta_1+2}-A^{\theta_1+2}\right),
\end{align*}
for all $s\in\Omega$, where $G:=G_{w,\theta_1, \theta_2}$ is a positive constant. Thus one can choose a $\delta$, $0<\delta<1$, such that $\frac{G}{24a^2}(B^{\theta_1+2}-A^{\theta_1+2})<\epsilon$ for $|A-B|<\delta$. Therefore $\int_0^\infty g(s,u)\ du$ converges uniformly at the lower limit. Similarly one can show the uniform convergence at the upper limit. 

Finally, since all of the hypotheses of Theorem \ref{ldcttemme1} are satisfied, we conclude that $\int_0^\infty g(s,u)\ du$, that is, the triple integral in \eqref{ti}, is analytic in $-1<\textup{Re}(s)<2$.

\end{proof}
\begin{remark}\label{validityoflemma4.9}
The above lemma can be used to show that Lemma \textup{\ref{triple integral to one integral}} actually holds for $-1<\textup{Re}(s)<2$. To see this, first note that from \cite[p.~270]{ww} the integral on the right-hand side of \eqref{3in1} is an entire function of $s$. Hence the right-hand side itself is analytic in $\mathbb{C}\backslash(2\mathbb{N}\cup\mathbb{Z}^{-})$. However, Lemma \textup{\ref{tia}} implies, in particular that the triple integral on the left-hand side of \eqref{3in1} is analytic in $-1<\textup{Re}(s)<2$. Hence by the principle of analytic continuation, Lemma \textup{\ref{triple integral to one integral}} is valid for $-1<\textup{Re}(s)<2$.
\end{remark}

\subsection{Analytic continuation of $\zeta_w(s, a)$ in $\textup{Re}(s)>-1, s\neq1$: Proof of Theorem \ref{lse}}
\hfill\\

As explained in the introduction, $\zeta_w(s, a)$, as is originally defined in \eqref{new zeta}, is analytic in Re$(s)>1$. From what we just proved in Lemma \ref{tia}, it is clear that the right-hand side of \eqref{hermiteeqn} is analytic for $-1<\textup{Re}(s)<2$ except for a simple pole at $s=1$. Thus we get a meromorphic continuation of $\zeta_w(s, a)$ in the region $-1<\textup{Re}(s)\leq 1$ after we \underline{\emph{define}} $\zeta_w(s, a+1)$ to be the right-hand side of \eqref{hermiteeqn} for $-1<\textup{Re}(s)\leq 1, s\neq1$. This gives meromorphic continuation of $\zeta_w(s, a+1)$ (and hence of $\zeta_w(s, a)$) in Re$(s)>-1$.

As for the residue at $s=1$, multiply both sides of \eqref{hermiteeqn} by $s-1$, let $s\to 1$, and then use the analyticity of the triple integral in the neighborhood of $s=1$ to see that
\begin{align*}
\lim_{s\to 1}(s-1)\zeta_w(s,a+1)&=A_{iw}(0)=\frac{\pi}{w^2}e^{-\frac{w^2}{4}}\textup{erfi}^{2}\left(\frac{w}{2}\right)=e^{\frac{w^2}{4}}{}_1F_{1}^{2}\left(1;\frac{3}{2};-\frac{w^2}{4}\right),
\end{align*}
as can be seen from \eqref{erferfi} and \eqref{ab}. It remains to prove only \eqref{luslus}. This is done by showing that for Re$(s)>1$, we have
\begin{equation}\label{laurentse}
\zeta_w(s, a+1)=\frac{e^{\frac{w^2}{4}}{}_1F_{1}^{2}\left(1;\frac{3}{2};-\frac{w^2}{4}\right)}{s-1}-\psi_w(a+1) +\sum_{n=1}^{\infty}c_n(w, a)(s-1)^n,
\end{equation}
which, using analytic continuation of $\zeta_w(s, a)$ in Re$(s)>-1, s\neq 1$ just proved above, is easily seen to be true in the deleted neighborhood of $s=1$. Here, $c_n(w, a)$ are constants depending only on $w, a$ and $n$. To that end, first write
\begin{align}\label{l1}
\zeta_w(s,a)-\frac{e^{\frac{w^2}{4}}{}_1F_{1}^{2}\left(1;\frac{3}{2};-\frac{w^2}{4}\right)}{s-1}
=\frac{1}{\G(s)}\left\{\Gamma(s)\zeta_w(s,a)-\Gamma(s-1)e^{\frac{w^2}{4}}{}_1F_{1}^{2}\left(1;\frac{3}{2};-\frac{w^2}{4}\right)\right\}.
\end{align}
Since Re$(s)>1$, we can represent $\G(s-1)$ as an integral
\begin{equation}\label{onee}
\G(s-1)=\int_{0}^{\infty}e^{-x}x^{s-2}\, dx.
\end{equation}
Also, replacing $u$ by $u/v$ in \eqref{int2} and then letting $z=0$, we have
\begin{align}\label{twoo}
\frac{4}{w^2\sqrt{\pi}}\int_0^\infty\int_0^\infty\frac{e^{-(u^2+v^2)}}{u}\sin(wv)\sinh(wu)dudv=e^{\frac{w^2}{4}}{}_1F_{1}^{2}\left(1;\frac{3}{2};-\frac{w^2}{4}\right).
\end{align}
Hence, combining \eqref{onee} and \eqref{twoo} to write $\Gamma(s-1)e^{\frac{w^2}{4}}{}_1F_{1}^{2}\left(1;\frac{3}{2};-\frac{w^2}{4}\right)$ as a triple integral, and also using Theorem \ref{mtr} to write $\Gamma(s)\zeta_w(s,a)$ as a triple integral, we find using \eqref{l1} that
\begin{align}\label{threee}
\zeta_w(s,a)-\frac{e^{\frac{w^2}{4}}{}_1F_{1}^{2}\left(1;\frac{3}{2};-\frac{w^2}{4}\right)}{s-1}&=\frac{4}{w^2\sqrt{\pi}\Gamma(s)}\int_0^\infty\int_0^\infty\int_0^\infty \frac{e^{-(u^2+v^2)}}{u}\sin(wv)\sinh(wu)\nonumber\\
&\quad\times\left\{\left(\frac{2xuv}{a-1}\right)^{\frac{s-1}{2}}\frac{J_{\frac{s-1}{2}}\left(\frac{(a-1)vx}{u}\right)}{e^{x}-1}-e^{-x}x^{s-2}\right\}\, dxdudv,
\end{align}
Since $s=1$ is the only (simple) pole of $\zeta_w(s, a)$ in Re$(s)>-1$, we also have
\begin{align}\label{fourr}
\zeta_w(s,a)-\frac{e^{\frac{w^2}{4}}{}_1F_{1}^{2}\left(1;\frac{3}{2};-\frac{w^2}{4}\right)}{s-1}=\sum_{n=0}^{\infty}c_n(w, a)(s-1)^n.
\end{align}
Now let $s\to 1$ in \eqref{threee} and \eqref{fourr} and compare their resulting right-hand sides so as to get
\begin{align*}
c_0(w, a)=\frac{-4}{w^2\sqrt{\pi}}\int_0^\infty\int_0^\infty\int_0^\infty \frac{e^{-(u^2+v^2+x)}}{u}\sin(wv)\sinh(wu)\left(\frac{1}{x}-\frac{J_{0}\left((a-1)\frac{vx}{u}\right)}{1-e^{-x}}\right)dxdudv.
\end{align*}
We refrain from justifying the interchange of the order of limit and the triple integral which can be done using standard methods. Hence from the above equation and \eqref{new psi function}, we see that $c_0(w, a)=-\psi_w(a)$. After replacing $a$ by $a+1$ in \eqref{fourr}, this proves \eqref{laurentse} for Re$(s)>1$, and hence in the neighborhood of $s=1$ as argued after \eqref{laurentse}. Thus \eqref{luslus} is established. This proves Theorem \ref{lse}.

\qed
\begin{remark}
The residue at the simple pole $s=1$ of $\zeta(s)$, and also of $\zeta(s, a)$, is usually calculated using Euler's summation formula. However, in the setting of $\zeta_w(s, a)$, Euler's formula leads to a cumbersome expression involving a double integral and a triple integral whose analyticity beyond $\textup{Re}(s)>1$ needs to be established first. On the other hand, calculating the residue using the generalization of Hermite's formula \eqref{hermiteeqn}, as shown in the above proof, necessitates analyticity of only one triple integral which is established in Lemma \textup{\ref{tia}}.
\end{remark}
We now obtain a well-known result \cite[p.~152]{berndtrocky1972} for $\zeta(s, a)$ as a corollary of Theorem \ref{lse}.
\begin{corollary}
Let $0<a<1$. In the neighborhood of $s=1$, we have
\begin{equation}\label{lsespl}
\zeta(s, a+1)=\frac{1}{s-1}-\psi(a+1)+O_{a}(|s-1|).
\end{equation}
\end{corollary}
\begin{proof}
Let $w=0$ in \eqref{laurentse}. From Remark \ref{ext}, we see that $\zeta_0(s, a+1)=\zeta(s, a+1)$. Also, from \eqref{limitpsi}, we have $\psi_0(a+1)=\psi(a+1)$. This leads to
\begin{equation*}
\zeta(s, a+1)=\frac{1}{s-1}-\psi(a+1) +\sum_{n=1}^{\infty}c_n(0, a)(s-1)^n,
\end{equation*}
which proves \eqref{lsespl}.
\end{proof}

\subsection{Asymptotic estimate of $\zeta_w(s, a)$ as $a\to\infty$}\label{aezwsa}
\hfill\\

Here we prove a result which is instrumental in proving that the infinite series in \eqref{genramhureqeqn} are absolutely and uniformly convergent in $-1<\textup{Re}(z)<1$.
\begin{theorem}
For $-1<\textup{Re}(s)<2$, as $a\to\infty$,
\begin{align}\label{hermiteeqn1}
\zeta_w(s,a+1)=-\frac{a^{-s}}{2}A_w(s-1)+\frac{a^{1-s}}{s-1}A_{iw}(1-s)+O_{s,w}\left(a^{-\textup{Re}(s)-1}\right).
\end{align}
\end{theorem}
\begin{proof}
By analytic continuation, we know that \eqref{hermiteeqn} is valid for $-1<\textup{Re}(s)<2$. In view of this, it suffices to show that as $a\to\infty$,
\begin{align}\label{bbound}
\frac{2^{s+2}a^{1-s}}{w^2\Gamma\left(1-\frac{s}{2}\right)\Gamma(s)}\int_0^\infty\int_1^\infty\int_{0}^\infty \frac{u^{s-2}e^{-(u^2+v^2)}\sin(wv)\sinh(wu)}{(e^{\frac{2\pi avt}{u}}-1)(t^2-1)^{\frac{s}{2}}}\ dvdtdu=O_{s, w}\left(a^{-\textup{Re}(s)-1}\right).
\end{align}
Observe that \eqref{bound for a} holds also when $t$ and $u$ are fixed and $a\to\infty$. 
Moreover, since $-1<\textup{Re}(s)<2$, the double integral
\begin{align*}
\int_0^\infty\int_{1}^\infty \frac{u^{\textup{Re}(s)}e^{-u^2}|\sinh(wu)|}{t^2(t^2-1)^{\frac{\textup{Re}(s)}{2}}}\ dtdu
\end{align*}
converges. Hence along with \eqref{bound for a}, we arrive at \eqref{bbound}.
Substituting \eqref{bbound} in \eqref{hermiteeqn} results in \eqref{hermiteeqn1}.
\end{proof}

\subsection{Limitations of \eqref{hermiteeqn} in analytically continuing $\zeta_w(s, a)$ in Re$(s)\leq-1$}\label{nomorehermite}
\hfill\\

It is well-known that Hermite's formula \eqref{hermiteor} provides analytic continuation of $\zeta(s, a)$ in $\mathbb{C}\backslash\{1\}$. In contrast to this, our generalization of Hermite's formula for $\zeta_w(s, a)$, that is \eqref{hermiteeqn}, does not give analytic continuation in $\mathbb{C}\backslash\{1\}$ as will now be proved by showing that the triple integral on the right-hand side of \eqref{hermiteeqn} is not analytic at least at $s=-1, -2, -3, \cdots$. We show this only for $s=-1$. 

Suppose the triple integral was analytic at $s=-1$. Then we could \emph{define} $\zeta_w(-1, a+1)$ by the right-hand side of \eqref{hermiteeqn} at $s=-1$. Since $\G(s)$ has a pole at $s=-1$, this would imply
\begin{align*}
\zeta_w(-1, a+1)&=-\frac{a}{2}A_{w}(-2)-\frac{a^2}{2}A_{iw}(2)\nonumber\\
&=-\frac{a}{2}e^{-\frac{w^2}{4}}{}_1F_{1}\left(1;\frac{3}{2};\frac{w^2}{4}\right)-\frac{a^2}{2}e^{\frac{w^2}{4}}{}_1F_{1}\left(1;\frac{3}{2};-\frac{w^2}{4}\right){}_1F_{1}\left(2;\frac{3}{2};-\frac{w^2}{4}\right).
\end{align*}
In particular, for $w=0$, we would have
\begin{equation*}
\zeta_0(-1,a+1)=\zeta(-1,a+1)=-\frac{a}{2}-\frac{a^2}{2},
\end{equation*}
which is incorrect since \cite[p.~264, Theorem 12.13]{apostol}
\begin{equation*}
\zeta(-1,a+1)=-\frac{a}{2}-\frac{a^2}{2}-\frac{1}{12}.
\end{equation*}
This shows that the triple integral cannot be analytic at $s=-1$, and that it very well contributes towards $\zeta_w(-1,a+1)$. 

It can be similarly seen that the triple integral is not analytic also at $s=2$.

\section{Some properties of the generalized digamma function $\psi_w(a)$}\label{psiwa}
\setcounter{subsection}{1}
We begin this subsection by obtaining a special value of $\psi_w(a)$.
\begin{theorem}
Let $w\in\mathbb{C}$. We have
\begin{equation*}
\psi_w(1)=-\gamma\frac{\pi e^{-\frac{w^2}{4}}}{w^2}\mathrm{erfi^2}\left(\frac{w}{2}\right).
\end{equation*}
\end{theorem}
\begin{proof}
Note that the definition of $\psi_w(a)$ in \eqref{new psi function} holds for any $a\in\mathbb{C}$ as can be seen from the asymptotic expansion of $J_{0}(x)$ \cite[p.~360, Formula \textbf{9.1.7}]{as}. Let $a=1$ in this definition and use the fact $J_{0}(0)=1$ to observe that
\begin{align*}
\psi_w(1)&=\frac{4}{w^2\sqrt{\pi}}\int_0^\infty\int_0^\infty \frac{e^{-(u^2+v^2)}}{u}\sin(wv)\sinh(wu)\int_0^\infty e^{-x}\left(\frac{1}{x}-\frac{1}{1-e^{-x}}\right)dxdudv \nonumber\\
&=-\gamma\frac{4}{w^2\sqrt{\pi}}\int_0^\infty\int_0^\infty \frac{e^{-(u^2+v^2)}}{u}\sin(wv)\sinh(wu)dudv \nonumber\\
&=-\gamma\frac{\pi}{w^2}e^{-\frac{w^2}{4}}\ \mathrm{erfi}^2\left(\frac{w}{2}\right),
\end{align*}
where in the first step, we used \eqref{psiintegral} and the well-known result $\psi(1)=-\g$, and in the second step, we used \eqref{twoo}.
\end{proof}
We now show how $\psi_w(a)$ reduces to the digamma function $\psi(a)$ when $w=0$.
\begin{proof}[\textup{\eqref{limitpsi}}][]
Let $w\to 0$ in \eqref{new psi function} and interchange the order of limit and the triple integral to obtain
\begin{align}\label{psi as a special case}
\lim_{w\rightarrow 0}\psi_w(a)&=\frac{4}{\sqrt{\pi}}\int_0^\infty\int_0^\infty\int_0^\infty v e^{-(u^2+v^2+x)}\left(\frac{1}{x}-\frac{J_{0}\left(\frac{(a-1)vx}{u}\right)}{1-e^{-x}}\right)\, dxdudv \nonumber\\
&=\frac{4}{\sqrt{\pi}}\int_0^\infty e^{-x}\Bigg(\frac{\sqrt{\pi}}{4x}-\frac{1}{1-e^{-x}}\int_0^\infty\int_0^\infty v e^{-(u^2+v^2)}J_{0}\left(\frac{(a-1)vx}{u}\right)dudv\Bigg)\, dx,
\end{align}
where we used the fact that
\begin{align*}
\int_0^\infty\int_0^\infty ve^{-(u^2+v^2)}dudv=\frac{\sqrt{\pi}}{4}.
\end{align*}
Now replacing $a$ by $a-1$ and letting $s=1$ in Lemma \ref{useful}, we find that
\begin{equation}\label{psi as a special case1}
\int_0^\infty\int_0^\infty v e^{-(u^2+v^2)}J_{0}\left(\frac{(a-1)vx}{u}\right)dudv=
\begin{cases} 
 \frac{\sqrt{\pi}}{4}e^{-(a-1)x},  & \mathrm{Re}(a)>1, \\
      \frac{\sqrt{\pi}}{4}e^{-(1-a)x}, & \mathrm{Re}(a)<1. \\
   \end{cases}
\end{equation}
From \eqref{psi as a special case} and \eqref{psi as a special case1}, for $\mathrm{Re}(a)>1$,
\begin{align*}
\lim_{w\rightarrow 0}\psi_w(a)&=\frac{4}{\sqrt{\pi}}\int_0^\infty e^{-x}\left(\frac{\sqrt{\pi}}{4x}-\frac{\sqrt{\pi}}{4}\frac{e^{-(a-1)x}}{1-e^{-x}}\right)dx\\
&=\psi(a),
\end{align*} 
where we used \eqref{psiintegral}. Similarly for $\mathrm{Re}(a)<1$, we use \eqref{psiintegral} to derive that $\lim_{w\rightarrow 0}\psi_w(a)=\psi(2-a)$.
\end{proof} 
We now derive a result which gives asymptotic behavior of $\psi_w(a+1)$ as $a\to\infty$.
\begin{theorem}\label{generalization of psi func intl}
For $w\in\mathbb{C},\ 0<a\leq 1$, we have
\begin{align}\label{generalization of psi func intleqn}
\psi_w(a+1)&=\frac{\pi}{2aw^2}e^{\frac{w^2}{4}}\mathrm{erf}^2\left(\frac{w}{2}\right)+\frac{\pi}{w^2}e^{-\frac{w^2}{4}}\mathrm{erfi}^2\left(\frac{w}{2}\right)\log a \nonumber\\
&\qquad+\frac{\sqrt{\pi}}{2w}\mathrm{erfi}\left(\frac{w}{2}\right)\sum_{n=0}^\infty\frac{(-w^2/4)^n}{(3/2)_n}(\psi(n+1)+\gamma)\nonumber\\
&\qquad-\frac{8}{w^2\sqrt{\pi}}\int_0^\infty\int_1^\infty\int_0^\infty \frac{e^{-(u^2+v^2)}\sin(wv)\sinh(wu)}{u\left(e^{\frac{2\pi avt}{u}}-1\right)\sqrt{t^2-1}}\, dvdtdu.
\end{align}
\end{theorem}
\begin{proof}
Subtracting $\frac{1}{(s-1)}\frac{\pi}{w^2}e^{-\frac{w^2}{4}}\textup{erfi}^{2}\left(\frac{w}{2}\right)$ from both sides of \eqref{hermiteeqn} and then letting $s\to 1$, we observe that
\begin{align*}
&\lim_{s\rightarrow1}\left(\zeta_w(s,a+1)-\frac{e^{\frac{w^2}{4}}{}_1F_1^2\left(1;\frac{3}{2};-\frac{w^2}{4}\right)}{s-1}\right)\nonumber\\
&=-\frac{1}{2a}A_w(0)+\lim_{s\rightarrow1}\left(\frac{a^{1-s}}{s-1}A_{iw}(1-s)-\frac{e^{\frac{w^2}{4}}{}_1F_1^2\left(1;\frac{3}{2};-\frac{w^2}{4}\right)}{s-1}\right)\nonumber\\
&\qquad+\frac{8}{w^2\sqrt{\pi}}\int_0^\infty\int_1^\infty\int_0^\infty \frac{e^{-(u^2+v^2)}\sin(wv)\sinh(wu)}{u\left(e^{\frac{2\pi avt}{u}}-1\right)\sqrt{t^2-1}}\, dvdtdu.
\end{align*}
Hence using \eqref{lse1} and \eqref{luslus}, we have
\begin{align}\label{psiw0}
\psi_w(a+1)&=\frac{\pi}{2aw^2}e^{\frac{w^2}{4}}\mathrm{erf}^2\left(\frac{w}{2}\right)-\lim_{s\rightarrow1}\left(\frac{a^{1-s}}{s-1}A_{iw}(1-s)-\frac{e^{\frac{w^2}{4}}{}_1F_1^2\left(1;\frac{3}{2};-\frac{w^2}{4}\right)}{s-1}\right)\nonumber\\
&\qquad-\frac{8}{w^2\sqrt{\pi}}\int_0^\infty\int_1^\infty\int_0^\infty \frac{e^{-(u^2+v^2)}\sin(wv)\sinh(wu)}{u\left(e^{\frac{2\pi avt}{u}}-1\right)\sqrt{t^2-1}}\, dvdtdu.
\end{align}
To evaluate the limit, first note that, as $s\rightarrow1$,
\begin{align*}
a^{-(s-1)}=e^{-(s-1)\log a}=1-(s-1)\log a+\frac{(s-1)^2(\log a)^2}{2!}+O(|(s-1)|^3),
\end{align*}
and 
\begin{align*}
{}_1F_1\left(\frac{3-s}{2};\frac{3}{2};-\frac{w^2}{4}\right)&=\frac{\sqrt{\pi}e^{-\frac{w^2}{4}}\mathrm{erfi}\left(\frac{w}{2}\right)}{w}+(s-1)\frac{d}{ds}{}_1F_1\left(\frac{3-s}{2};\frac{3}{2};-\frac{w^2}{4}\right)\Bigg|_{s=1}+O(|s-1|^2).
\end{align*}
Now
\begin{align}\label{jk}
&\lim_{s\rightarrow1}\left(\frac{a^{1-s}}{s-1}A_{iw}(1-s)-\frac{e^{\frac{w^2}{4}}{}_1F_1^2\left(1;\frac{3}{2};-\frac{w^2}{4}\right)}{s-1}\right)\nonumber\\
&=\frac{\sqrt{\pi}}{w}\textup{erfi}\left(\frac{w}{2}\right)\lim_{s\rightarrow 1}\left\{\frac{a^{1-s}}{s-1}{}_1F_1\left(\frac{s-3}{2};\frac{3}{2};-\frac{w^2}{4}\right)-\frac{1}{s-1}\frac{e^{-w^2/4}\sqrt{\pi}\mathrm{erfi}(w/2)}{w}\right\} \nonumber\\
&=-\frac{\pi}{w^2}e^{-\frac{w^2}{4}}\mathrm{erfi}^2\left(\frac{w}{2}\right)\log a+\frac{\sqrt{\pi}}{w}\textup{erfi}\left(\frac{w}{2}\right)\frac{d}{ds}{}_1F_1\left(\frac{3-s}{2};\frac{3}{2};-\frac{w^2}{4}\right)\Bigg|_{s=1}.
\end{align}
By using series definition of ${}_1F_1$, we see that
\begin{align*}
\frac{d}{ds}{}_1F_1\left(\frac{3-s}{2};\frac{3}{2};-\frac{w^2}{4}\right)=-\frac{1}{2}\sum_{n=0}^\infty \frac{\left(-\frac{w^2}{4}\right)^n}{\left(\frac{3}{2}\right)_nn!}\left(\frac{3-s}{2}\right)_n \sum_{j=0}^{n-1}\frac{1}{\left(\frac{3-s}{2}+j\right)}.
\end{align*}
Since
\begin{align*}
&\lim_{s\rightarrow 1}\left(1+\frac{s-1}{2}\right)\left(2+\frac{s-1}{2}\right)...\left(n+\frac{s-1}{2}\right)\left(\frac{1}{1+\frac{s-1}{2}}+\frac{1}{2+\frac{s-1}{2}}+...+\frac{1}{n+\frac{s-1}{2}}\right) \nonumber\\
&=n!(\psi(n+1)+\gamma),
\end{align*}
this gives
\begin{align}\label{1F1 +}
\frac{d}{ds}{}_1F_1\left(\frac{3-s}{2};\frac{3}{2};-\frac{w^2}{4}\right)\Bigg|_{s=1}&=-\frac{1}{2}\sum_{n=0}^\infty \frac{\left(-\frac{w^2}{4}\right)^n}{\left(\frac{3}{2}\right)_n}\left(\psi(n+1)+\gamma\right).
\end{align}
Finally, from \eqref{psiw0}, \eqref{jk} and \eqref{1F1 +}, we arrive at \eqref{generalization of psi func intleqn}.
\end{proof}
\begin{remark}
Note that either by letting $s=1$ in \eqref{hermiteeqn1} and then proceedings along the same lines as in the proof of Theorem \textup{\ref{generalization of psi func intl}} or by working from scratch with the triple integral in Theorem \textup{\ref{generalization of psi func intl}}, we arrive at
\begin{align*}
\psi_w(a+1)&=\frac{\pi}{2aw^2}e^{\frac{w^2}{4}}\mathrm{erf}^2\left(\frac{w}{2}\right)+\frac{\pi}{w^2}e^{-\frac{w^2}{4}}\mathrm{erfi}^2\left(\frac{w}{2}\right)\log a \nonumber\\
&\qquad+\frac{\sqrt{\pi}}{2w}\mathrm{erfi}\left(\frac{w}{2}\right)\sum_{n=0}^\infty\frac{(-w^2/4)^n}{(3/2)_n}(\psi(n+1)+\gamma)+O_w\left(\frac{1}{a^2}\right)
\end{align*}
as $a\to\infty$. Of course, one can derive a complete asymptotic expansion for $\psi_w(a+1)$ by obtaining the same for the triple integral in \eqref{generalization of psi func intleqn}. Note that when $w=0$, the above estimate gives the well-known result
\begin{equation*}
\psi(a)=-\frac{1}{2a}+\log a+O\left(\frac{1}{a^2}\right).
\end{equation*}
\end{remark}
As a special case of Theorem \ref{generalization of psi func intl}, we get the following well-known formula.
\begin{corollary}
Let $0<a\leq 1$. Then
\begin{align}\label{psipsi-1}
\psi(a)=\log a-\frac{1}{2a}-2\int_0^\infty\frac{y\ dy}{(y^2+a^2)(e^{2\pi y}-1)}.
\end{align}
\end{corollary}
\begin{proof}
Let $w\to 0$ in \eqref{generalization of psi func intleqn}. Using \eqref{limitpsi}, we see that
\begin{align}\label{psipsi0}
\psi(a+1)=\frac{1}{2a}+\log a-\frac{8}{\sqrt{\pi}}\int_0^\infty\int_1^\infty\int_0^\infty \frac{ve^{-(u^2+v^2)}}{\left(e^{\frac{2\pi avt}{u}}-1\right)\sqrt{t^2-1}}\, dvdtdu,
\end{align}
where the interchange of the order of limit and the triple integral is justified using standard methods. To evaluate the triple integral on the right-hand side, let $s\to 1$ in Lemma \ref{triple integral to one integral}. This gives
\begin{align}\label{psipsi}
\int_0^\infty\int_1^\infty\int_0^\infty \frac{ve^{-(u^2+v^2)}}{\left(e^{\frac{2\pi avt}{u}}-1\right)\sqrt{t^2-1}}\, dvdtdu&=\frac{\sqrt{\pi}}{4}\int_0^\infty \frac{\sin\left(\tan^{-1}\left(\frac{y}{a}\right)\right)}{\left(e^{2\pi y}-1\right)\sqrt{y^2+a^2}}\ dy\nonumber\\
&=\frac{\sqrt{\pi}}{4}\int_0^\infty\frac{y\ dy}{(y^2+a^2)(e^{2\pi y}-1)},
\end{align}
since $\sin\left(\tan^{-1}(\theta)\right)=\theta/\sqrt{1+\theta^2}$. Now substitute \eqref{psipsi} in \eqref{psipsi0} and use the functional equation $\psi(a+1)=\psi(a)+1/a$ to arrive at \eqref{psipsi-1}.
\end{proof}
\section{Reciprocal functions in the second Koshliakov kernel}\label{seckosh}
In \cite[Equations (8), (13)]{kosh1938}, Koshliakov obtained two remarkable integral evaluations, namely, for $-\frac{1}{2}<\textup{Re}(z)<\frac{1}{2}$,\footnote{Koshliakov \cite{kosh1938} stated it only for $-\frac{1}{2}<z<\frac{1}{2}$, however, they are easily seen to be true for $-\frac{1}{2}<\textup{Re}(z)<\frac{1}{2}$.}
\begin{equation}\label{koshlyakov-1}
\int_{0}^{\infty} K_{z}(t) \left( \cos(\pi z) M_{2z}(2 \sqrt{xt}) -
\sin(\pi z) J_{2z}(2 \sqrt{xt}) \right)\, dt = K_{z}(x).
\end{equation}
and
\begin{equation}\label{koshlyakov-2}
\int_{0}^{\infty} tK_{z}(t) \left( \sin(\pi z) J_{2 z}(2 \sqrt{xt}) -
\cos(\pi z) L_{2 z}(2 \sqrt{xt}) \right) dt = xK_{z}(x),
\end{equation}
where
\begin{align}\label{kernels}
M_{\nu}(x)=\frac{2}{\pi}K_{\nu}(x)-Y_{\nu}(x)\hspace{4mm}\text{and}\hspace{4mm}L_{\nu}(x)=-\frac{2}{\pi}K_{\nu}(x)-Y_{\nu}(x),
\end{align}
$Y_{\nu}(x)$ being the Bessel function of the second kind of non-integer order $\nu$ defined by \cite[p.~64]{watson-1966a}
\begin{equation*}
Y_{\nu}(x)=\frac{J_{\nu}(x)\cos\left(\nu\pi\right)-J_{-\nu}(x)}{\sin\left(\nu\pi\right)},
\end{equation*}
and $Y_{n}(x)$ for integer $n$ defined by $Y_{n}(x)=\lim_{\nu\to n}Y_{\nu}(x)$.

Owing to \eqref{koshlyakov-1} and \eqref{koshlyakov-2}, we call the two kernels, namely, $ \cos(\pi z) M_{2z}(2 \sqrt{xt}) -\sin(\pi z) J_{2z}(2 \sqrt{xt})$ and $\sin(\pi z) J_{2 z}(2 \sqrt{xt}) -\cos(\pi z) L_{2 z}(2 \sqrt{xt})$, the first and the second Koshliakov kernels respectively, and the integrals \cite[Definition 15.1]{bdrz}
\begin{equation*}
\int_{0}^{\infty}f(t, z)\left( \cos(\pi z) M_{2z}(2 \sqrt{xt}) -\sin(\pi z) J_{2z}(2 \sqrt{xt}) \right)\, dt
\end{equation*}
and
\begin{equation*}
\int_{0}^{\infty} f(t, z)\left(\sin(\pi z) J_{2 z}(2 \sqrt{xt}) -\cos(\pi z) L_{2 z}(2 \sqrt{xt})\right) dt,
\end{equation*}
the first and the second Koshliakov transforms of $f(t, z)$ respectively whenever the integrals converge.

Let $\phi$ and $\psi$ be the second Koshliakov transforms of each other\footnote{Throughout the analysis $z$ and $w$ will be fixed complex numbers in some domains of their respective complex planes.}, that is,
\begin{equation}\label{phii}
\phi(x,z, w)=2\int_0^\infty\psi(t,z, w)\left(\sin(\pi z)J_{2z}(4\sqrt{xt})-\cos(\pi z)L_{2z}(4\sqrt{xt}))\right)dt,
\end{equation}
and
\begin{equation}\label{psii}
\psi(x,z, w)=2\int_0^\infty\phi(t,z, w)\left(\sin(\pi z)J_{2z}(4\sqrt{xt})-\cos(\pi z)L_{2z}(4\sqrt{xt}))\right)dt,
\end{equation}
Let the normalized Mellin transform $Z_1(s,z, w)$ and $Z_2(s,z, w)$ of the functions $\phi(x,z)$
 and $\psi(x,z)$ be defined by
 \begin{equation}\label{defZ1}
 \Gamma\left(\frac{1+s-z}{2}\right) \Gamma\left(\frac{1+s+z}{2}\right)Z_1(s,z, w):=\int_0^\infty x^{s-1}\phi(x,z, w)dx,
 \end{equation}
and 
\begin{align}\label{defZ2}
\Gamma\left(\frac{1+s-z}{2}\right) \Gamma\left(\frac{1+s+z}{2}\right)Z_2(s,z, w):=\int_0^\infty x^{s-1}\psi(x,z, w)dx,
\end{align}
where each of the above equations holds in a particular vertical strip in the complex $s-$plane. Let
\begin{equation}
Z(s,z, w):=Z_1(s,z, w)+Z_2(s,z, w)\ \text{and}\ \Theta(x,z, w):=\phi(x,z, w)+\psi(x,z, w),\label{defTheta}
\end{equation}
whence
\begin{align*}
\Gamma\left(\frac{1+s-z}{2}\right) \Gamma\left(\frac{1+s+z}{2}\right)Z(s,z, w)=\int_0^\infty x^{s-1}\Theta(x,z, w)dx
\end{align*}
for those values of $s$ which lie in the intersection of the aforementioned vertical strips. 

\subsection{Properties of functions reciprocal in the second Koshliakov kernel}\label{propmt}
\hfill\\

A class of functions $\Diamond_{\eta,\omega}$, called the `diamond class', was introduced in \cite{koshkernel}. We reproduce its definition below as the specific choices of the functions $\phi$ and $\psi$, satisfying \eqref{phii} and \eqref{psii}, which are used to prove Theorem \ref{analogueReslutthm}, and hence Theorem \ref{xiintgenramhurthm}, need to be members of this class.
\begin{definition}\label{diamondc}
Let $0<\omega\leq \pi$ and $\eta>0$. Let $z$ be fixed. If $u(s, z, w)$ is such that
\begin{enumerate}
\item[(i)] $u(s, z, w)$ is an analytic function of $s=re^{i\theta}$ regular in the angle given by $r>0$, $|\theta|<\omega$,
\item[(ii)] $u(s, z, w)$ satisfies the bounds
\begin{equation*}
u(s, z, w)=
			\begin{cases}
			O_{z}(|s|^{-\delta}) & \mbox{ if } |s| \le 1,\\
			{O_z(|s|^{-\eta-1-|\textup{Re}(z)|})} & \mbox{ if } |s| > 1,
			\end{cases}
\end{equation*}
\end{enumerate}
for every positive $\delta$ and uniformly in any angle $|\theta|<\omega$, then we say that $u$ belongs to the diamond class $\Diamond_{\eta, \omega}$ and write $u(s, z, w)\in \Diamond_{\eta,\omega}$.
\end{definition}

The following lemma is used to prove Theorem \ref{analogueReslutthm} of which Theorem \ref{xiintgenramhurthm} is a special case.
\begin{lemma}\label{parts12}
Let $\eta>0$, $0<\omega\leq\pi$ and let $-1/4<$ \textup{Re}$(z)<1/4$. Suppose that $\phi,\psi\in\Diamond_{\eta,\omega}$ are reciprocal with respect to the second Koshliakov kernel $\sin(\pi z) J_{2 z}(2 \sqrt{xt}) -\cos(\pi z) L_{2 z}(2 \sqrt{xt})$. Let $Z(s, z, w)$ be defined in \eqref{defTheta}. Then
\begin{enumerate}
\item $Z(s,z, w)=Z(1-s,z, w)\ \forall\ s$ such that $-\eta-|\textup{Re}(z)| < \textup{Re}(s) < 1 + \eta+|\textup{Re}(z)|$.\\
\item $Z(\sigma+it,z, w)\ll_z e^{\left(\frac{\pi}{2}-\omega+\epsilon\right)|t|}$ for every $\epsilon>0$.
\end{enumerate}
\end{lemma}
\begin{proof}
We first show the validity of the first part of the lemma for $\frac{3}{4}<$ Re$(s)<1-|$Re$(z)|$. It is later extended to $-\eta-|\textup{Re}(z)| < \textup{Re}(s) < 1 + \eta+|\textup{Re}(z)|$ by analytic continuation. Define
\begin{align*}
h(s,z):=\Gamma\left(\frac{1+s-z}{2}\right) \Gamma\left(\frac{1+s+z}{2}\right).
\end{align*}
Note that 
\begin{align}\label{functeqn1}
&h(1-s,z)Z_1(1-s,z, w)\nonumber\\
&=\int_0^\infty x^{-s}\phi(x,z, w)dx \nonumber \\
&=2\int_0^\infty x^{-s} \int_0^\infty \psi(t,z, w)\left(\sin(\pi z)J_{2z}(4\sqrt{xt})-\cos(\pi z)L_{2z}(4\sqrt{xt}))\right)dtdx \nonumber\\
&=2\int_0^\infty \psi(t,z, w) \int_0^\infty x^{-s}\left(\sin(\pi z)J_{2z}(4\sqrt{xt})-\cos(\pi z)L_{2z}(4\sqrt{xt}))\right) dxdt\nonumber\\
&=2\pi^{2-2s}\int_0^\infty t^{s-1} \psi(t,z, w) \int_0^\infty u^{-s}\left(\sin(\pi z)J_{2z}(4\pi\sqrt{u})-\cos(\pi z)L_{2z}(4\pi\sqrt{u}))\right)du dt,
\end{align}
where the interchange of the order of integration in the second step above, which is justifiable for $\frac{3}{4}<$Re$(s)<1-|$Re$(z)|$, can be given along the same lines as for that in Lemma 2.1 of \cite{koshkernel}\footnote{The integrals $\int_0^\infty x^{-s} \int_0^\infty \psi(t,z, w)\left(\sin(\pi z)J_{2z}(4\sqrt{xt})-\cos(\pi z)L_{2z}(4\sqrt{xt}))\right)dtdx$ and $\int_0^\infty \psi(t,z, w)$\newline$\times\int_0^\infty x^{-s}\left(\sin(\pi z)J_{2z}(4\sqrt{xt})-\cos(\pi z)L_{2z}(4\sqrt{xt}))\right) dxdt$ are absolutely convergent for $\frac{3}{4}<$Re$(s)<1-|$Re$(z)|$. So the result follows from Fubini's theorem.}.

Replacing $s$ by $1-s ,\ z$ by $2z$ and letting $y=1$ in \cite[Lemma 5.2]{dixitmoll}, we have, for $\frac{1}{4}<$Re$(s)<1\pm$Re$(z)$ and $y>0$,
\begin{align}\label{functeqn2}
\int_0^\infty& x^{-s}\left(\sin(\pi z)J_{2z}(4\pi\sqrt{x})-\cos(\pi z)L_{2z}(4\pi\sqrt{x})\right)dx \nonumber \\
& =\frac{1}{2^{2-2s}\pi^{3-2s}}\Gamma\left(1-s-z\right) \Gamma\left(1-s+z\right)\left(\cos (\pi z)+\cos (\pi s)\right).
\end{align}
Thus from \eqref{functeqn1} and \eqref{functeqn2},
\begin{align*}
h(1-s,z)Z_1(1-s,z, w)&=\frac{2^{2s-1}}{\pi}\Gamma\left(1-s-z\right) \Gamma\left(1-s+z\right)\left(\cos (\pi z)+\cos (\pi s)\right)\int_0^\infty t^{s-1}\psi(t,z, w)dt, \nonumber\\
&=\frac{2^{2s-1}}{\pi}\Gamma\left(1-s-z\right) \Gamma\left(1-s+z\right)\left(\cos (\pi z)+\cos (\pi s)\right) h(s,z) Z_2(s,z, w), \nonumber\\
&= \Gamma\left(\frac{2-s+z}{2}\right) \Gamma\left(\frac{2-s-z}{2}\right)Z_2(s,z, w)\nonumber\\
&=h(1-s,z)Z_2(s,z,w),
\end{align*}
where in the penultimate step, we used \eqref{dup} and \eqref{refl} for simplification. This implies that
\begin{equation*}
Z_1(1-s,z, w)=Z_2(s,z,w).
\end{equation*}
Similarly one can prove 
\begin{align*}
Z_2(1-s,z,w)&=Z_1(s,z,w).
\end{align*}
This, together with \eqref{defTheta}, implies $$Z(1-s,z,w)=Z(s,z,w).$$ This proves the functional equation of $Z(s, z, w)$ for $\frac{3}{4}<$ Re$(s)<1-|$Re$(z)|$. Now, as shown in \cite[p.~1117]{koshkernel}, the functions $Z_1(s, z, w)$ and $Z_2(s, z, w)$ are well-defined and analytic in $s$ in the region $\delta<$ Re$(s)<1+\eta+|$Re$(z)|$ for every $\delta>0$\footnote{There is a typo on page 1117 of \cite{koshkernel}, namely, all instances of $|$Re$(z)|/2$ should be replaced by $|$Re$(z)|$.}. Hence the results $Z_1(1-s,z, w)=Z_2(s,z,w)$, $Z_2(1-s,z, w)=Z_1(s,z,w)$ and $Z(1-s,z, w)=Z(s,z,w)$ are valid in $-\eta-|\textup{Re}(z)| < \textup{Re}(s) < 1 + \eta+|\textup{Re}(z)|$ by analytic continuation. 

The proof of the second part of the lemma is along the similar lines as that of Lemma 2.1 in \cite[p.~1117-1118]{koshkernel} and is hence omitted.
\end{proof}
Next, we obtain a useful inverse Mellin transform that is required for proving Theorem \ref{analogueReslutthm} below.
\begin{lemma}\label{InverseMellOfQuetientOfGamma}
Let $x>0$ and $z$ be fixed such that $-\frac{1}{2}<\mathrm{Re}(z)<\frac{1}{2}$.	For $\pm\mathrm{Re}\left(\frac{z}{2}\right)<c:=\mathrm{Re}(s)<2+\mathrm{Re}\left(\frac{z}{2}\right)$, we have
\begin{align}\label{aconti}
\frac{1}{2\pi i}\int_{(c)}\frac{\Gamma\left(1+\frac{s}{2}-\frac{z}{4}\right)\Gamma\left(\frac{s}{2}+\frac{z}{4}\right)\Gamma\left(\frac{z}{4}-\frac{s}{2}\right)}{\Gamma\left(\frac{s}{2}+\frac{1}{2}-\frac{z}{4}\right)}x^{-s}ds=\frac{2}{\sqrt{\pi}}x^{-\frac{z}{2}} \Gamma\left(\frac{z}{2}\right) \left({}_2F_1\left(1,\frac{z}{2};\frac{1}{2};-x^{-2}\right)-1\right).
\end{align}
\end{lemma}
\begin{proof}
We first prove the result for $0<\textup{Re}(z)<1/2$ and then extend it to $-\frac{1}{2}<\mathrm{Re}(z)<\frac{1}{2}$ by analytic continuation.

From \cite[p.~198, Equation (5.51)]{ober} for $0<d:=$Re$(s)<\min\{\mathrm{Re}(\alpha),\mathrm{Re}(\beta)\}$,
\begin{align*}
\frac{1}{2\pi i}\int_{(d)}\frac{\Gamma(\alpha-s)\Gamma(\beta-s)\Gamma(s)}{\Gamma(\gamma-s)}x^{-s}ds=\frac{\Gamma(\alpha)\Gamma(\beta)}{\Gamma(\gamma)} {}_2F_1(\alpha ,\beta ;\gamma ;-x).
\end{align*}
Now let $\alpha=1, \beta=\frac{z}{2}, \gamma=\frac{1}{2}$ and replace $s$ by $-\left(\frac{s}{2}-\frac{z}{4}\right)$ so that for $-2\min\{1,\mathrm{Re}(\frac{z}{2})\}+\mathrm{Re}(\frac{z}{2})<d':=\mathrm{Re}(s)<\mathrm{Re}(\frac{z}{2})$, that is, for $-\mathrm{Re}(\frac{z}{2})<d':=\mathrm{Re}(s)<\mathrm{Re}(\frac{z}{2})$ (since $\mathrm{Re}(z)<\frac{1}{2}$), one has
\begin{align*}
\frac{1}{2\pi i}\int_{(d')}\frac{\Gamma\left(1+\frac{s}{2}-\frac{z}{4}\right)\Gamma\left(\frac{s}{2}+\frac{z}{4}\right)\Gamma\left(\frac{z}{4}-\frac{s}{2}\right)}{\Gamma\left(\frac{s}{2}+\frac{1}{2}-\frac{z}{4}\right)}x^{\frac{s}{2}-\frac{z}{4}}\frac{ds}{2}=\frac{\Gamma\left(\frac{z}{2}\right)}{\sqrt{\pi}}\ {}_2F_1\left(1,\frac{z}{2};\frac{1}{2};-x\right).
\end{align*}
Replace $x$ by $x^{-2}$ in the above equation to have
\begin{align}\label{d}
\frac{1}{2\pi i}\int_{(d')}\frac{\Gamma\left(1+\frac{s}{2}-\frac{z}{4}\right)\Gamma\left(\frac{s}{2}+\frac{z}{4}\right)\Gamma\left(\frac{z}{4}-\frac{s}{2}\right)}{\Gamma\left(\frac{s}{2}+\frac{1}{2}-\frac{z}{4}\right)}x^{-s}ds=\frac{2}{\sqrt{\pi}}x^{-\frac{z}{2}} \Gamma\left(\frac{z}{2}\right) {}_2F_1\left(1,\frac{z}{2};\frac{1}{2};-x^{-2}\right).
\end{align}
Consider the contour $\mathfrak{C}$ formed by the line segments $[d'-iT,c-iT], [c-iT,c+iT], [c+iT,d'+iT]$ and $[d'+iT, d'-iT]$, oriented in the counter-clockwise direction, where $T>\textup{Im}\left(z/2\right)$ and $\pm\mathrm{Re}\left(\frac{z}{2}\right)<c:=\mathrm{Re}(s)<2+\mathrm{\frac{Re(z)}{2}}$. Note that\\
1. $\Gamma\left(\frac{z}{4}-\frac{s}{2}\right)$ has poles at $s=\frac{z}{2},\  2+\frac{z}{2},\ 4+\frac{z}{2},...$\\
2. $\Gamma\left(1+\frac{s}{2}-\frac{z}{4}\right)$ has poles at $s=\frac{z}{2}-2,\ \frac{z}{2}-4,...$\\
3. $\Gamma\left(\frac{s}{2}+\frac{z}{4}\right)$ has poles at $s=-\frac{z}{2},\ -\frac{z}{2}-2,...$.\\
It is clear that the only $s=z/2$ is the only pole of the integrand in \eqref{d} that lies inside the contour $\mathfrak{C}$. Hence applying Cauchy's residue theorem, letting $T\to\infty$, noting that the integrals along horizontal segments vanish as $T\to\infty$ (due to \eqref{strivert}) and that the residue at $s=z/2$ is $-\frac{2}{\sqrt{\pi}}\Gamma\left(\frac{z}{2}\right)x^{-\frac{z}{2}}$, we find that for $\mathrm{Re}\left(\frac{z}{2}\right)<c:=\mathrm{Re}(s)<2+\mathrm{\frac{Re(z)}{2}}$,
\begin{align}
&\frac{1}{2\pi i} \int_{(c)}\frac{\Gamma\left(1+\frac{s}{2}-\frac{z}{4}\right)\Gamma\left(\frac{s}{2}+\frac{z}{4}\right)\Gamma\left(\frac{z}{4}-\frac{s}{2}\right)}{\Gamma\left(\frac{s}{2}+\frac{1}{2}-\frac{z}{4}\right)}x^{-s}ds\nonumber\\
&=\frac{1}{2\pi i}\int_{(d')}\frac{\Gamma\left(1+\frac{s}{2}-\frac{z}{4}\right)\Gamma\left(\frac{s}{2}+\frac{z}{4}\right)\Gamma\left(\frac{z}{4}-\frac{s}{2}\right)}{\Gamma\left(\frac{s}{2}+\frac{1}{2}-\frac{z}{4}\right)}x^{-s}ds -\frac{2}{\sqrt{\pi}}\Gamma\left(\frac{z}{2}\right)x^{-\frac{z}{2}}\nonumber\\
&=\frac{2}{\sqrt{\pi}}x^{-\frac{z}{2}} \Gamma\left(\frac{z}{2}\right) \left({}_2F_1\left(1,\frac{z}{2};\frac{1}{2};-x^{-2}\right)-1\right),\nonumber
\end{align}
where in the last step we used \eqref{d}. This proves \eqref{aconti} for $0<\textup{Re}(z)<1/2$. Using Theorem \ref{ldcttemme1}, it is easy to see that \eqref{aconti} actually holds for $-\frac{1}{2}<\mathrm{Re}(z)<\frac{1}{2}$.
\end{proof}

\subsection{A general theorem for evaluating an integral involving the Riemann $\Xi$-function}\label{genthmxi}
\hfill\\

We now prove a result of which is Theorem \ref{xiintgenramhurthm} is a special case.
\begin{theorem}\label{analogueReslutthm}
Let $\eta>0$ and let $-\frac{1}{2}<\mathrm{Re}(z)<\frac{1}{2}$. Suppose that $\phi,\psi\in\Diamond_{\eta,\omega}$, and are reciprocal with respect to the second Koshliakov kernel, that is, they satisfy \eqref{phii} and \eqref{psii}. Let $Z(s,z, w)$ and $\Theta(x,z, w)$ be defined in \eqref{defTheta}. Then,
\begin{align}\label{analogueReslut}
&\pi^{\frac{z-3}{2}}\int_0^\infty \Gamma\left(\frac{z-1+it}{4}\right)\Gamma\left(\frac{z-1-it}{4}\right) \Xi\left(\frac{t+iz}{2}\right)\Xi\left(\frac{t-iz}{2}\right)Z\left(\frac{1+it}{2},\frac{z}{2}, w\right)\frac{dt}{(z+1)^2+t^2} \nonumber \\
&=-\frac{1}{\pi}\Gamma\left(\frac{z}{2}\right)\sum_{n=1}^\infty\sigma_{-z}(n) \int_0^\infty \Theta\left(x,\frac{z}{2}, w\right)\left({}_2F_1\left(1,\frac{z}{2};\frac{1}{2};-\frac{x^2}{\pi^2n^2}\right)-1\right)x^{\frac{z-2}{2}}\, dx -S(z, w),
\end{align}
where 
\begin{align}\label{defS}
S(z, w):=2^{-1-z}\Gamma(1+z)\zeta(1+z)Z\left(1+\frac{z}{2},\frac{z}{2}, w\right)+2^{-z}\Gamma(z)\zeta(z)Z\left(1-\frac{z}{2},\frac{z}{2}, w\right).
\end{align}
\end{theorem}
\begin{proof}The convergence of the integral on the left-hand side of \eqref{analogueReslut} follows easily from Stirling's formula \eqref{strivert} and part (2) of Lemma \ref{parts12}. We next show that the series
\begin{equation}\label{serconv}
\sum_{n=1}^\infty\sigma_{-z}(n)\int_0^\infty \Theta\left(x,\frac{z}{2}, w\right)\left({}_2F_1\left(1,\frac{z}{2};\frac{1}{2};-\frac{x^2}{\pi^2n^2}\right)-1\right)x^{\frac{z-2}{2}}\, dx
\end{equation}
converges for $-\frac{1}{2}<\mathrm{Re}(z)<\frac{1}{2}$. Note that $\phi, \psi\in\Diamond_{\eta, \omega}$ imply that $\Theta\in\Diamond_{\eta, \omega}$.

We first split the integral inside the sum as
\begin{align*}
&\int_0^\infty \Theta\left(x,\frac{z}{2}, w\right)\left({}_2F_1\left(1,\frac{z}{2};\frac{1}{2};-\frac{x^2}{\pi^2n^2}\right)-1\right)x^{\frac{z-2}{2}}\, dx\nonumber\\
&=\left[\int_0^1+\int_{1}^{n\pi}+\int_{n\pi}^{\infty}\right]\Theta\left(x,\frac{z}{2}, w\right)\left({}_2F_1\left(1,\frac{z}{2};\frac{1}{2};-\frac{x^2}{\pi^2n^2}\right)-1\right)x^{\frac{z-2}{2}}\, dx\nonumber\\
&=:I_{1}(n, z, w)+I_{2}(n, z, w)+I_{3}(n, z, w).
\end{align*}
Consider $\sum_{n=1}^{\infty}\sigma_{-z}(n)I_{1}(n, z, w)$ first. Since ${}_2F_1\left(1,\frac{z}{2};\frac{1}{2};-\frac{x^2}{\pi^2n^2}\right)-1=O_z\left(\frac{x^2}{n^2}\right)$ for $x\in(0,1)$, we see using $\phi\in\Diamond_{\eta, \omega}$ that
\begin{align*}
I_{1}(n, z, w)\ll_z \frac{1}{n^2}\int_{0}^{1}x^{-\delta+\frac{1}{2}|\textup{Re}(z)|+1}\, dx\ll_z\frac{1}{n^2}.
\end{align*}
This implies that $\sum_{n=1}^{\infty}\sigma_{-z}(n)I_{1}(n, z, w)$ converges since $\mathrm{Re}(z)>-\frac{1}{2}$.

Next, consider $\sum_{n=1}^{\infty}\sigma_{-z}(n)I_{2}(n, z, w)$. Again, ${}_2F_1\left(1,\frac{z}{2};\frac{1}{2};-\frac{x^2}{\pi^2n^2}\right)-1=O_z\left(\frac{x^2}{n^2}\right)$ since $x\in(1, n\pi)$. 
\begin{align*}
I_{2}(n, z, w)\ll_z\frac{1}{n^2}\int_{1}^{n\pi}x^{-\eta-\frac{1}{2}|\textup{Re}(z)|+\frac{1}{2}\textup{Re}(z)}\, dx=\frac{1}{n^2}\left\{(n\pi)^{-\eta-\frac{1}{2}|\textup{Re}(z)|+\frac{1}{2}\textup{Re}(z)+1}-1\right\}.
\end{align*}
Since $\eta>0$ and $-\frac{1}{2}<\mathrm{Re}(z)<\frac{1}{2}$, this implies that $\sum_{n=1}^{\infty}\sigma_{-z}(n)I_{2}(n, z, w)$ converges.

We now show that $\sum_{n=1}^{\infty}\sigma_{-z}(n)I_{3}(n, z, w)$ converges. Note that
\begin{align}\label{ithree}
I_{3}(n, z, w)=(n\pi)^{z/2}\int_{1}^{\infty}\Theta\left(tn\pi,\frac{z}{2}, w\right)\left({}_2F_{1}\left(1,\frac{z}{2};\frac{1}{2};-t^2\right)-1\right)t^{\frac{z-2}{2}}\, dt.
\end{align}
To that end, we use \cite[p.~113, Equation (5.11)]{temme}, namely, for $|\arg(-\xi)|<\pi$,
\begin{align*}
{}_2F_1\left(a,b;c;\xi\right)&=\frac{\Gamma\left(c\right)\Gamma\left(b-a\right)}{\Gamma(b)\Gamma(c-a)}(-\xi)^{-a}{}_2F_1\left(a,1-c+a;1-b+a;\frac{1}{\xi}\right)\nonumber\\
&\qquad\qquad+\frac{\Gamma(c)\Gamma(a-b)}{\Gamma(a)\Gamma(c-b)}(-\xi)^{-b}{}_2F_1\left(b,1-c+b;1-a+b;\frac{1}{\xi}\right).
\end{align*}
Let $a=1, b=\frac{z}{2}, c=\frac{1}{2}$ and $\xi=-t^2$ in the above transformation to get
\begin{align*}
{}_2F_1\left(1,\frac{z}{2};\frac{1}{2};-t^2\right)&=\frac{\G\left(\frac{1}{2}\right)\G\left(\frac{z}{2}-1\right)}{t^2\G\left(\frac{z}{2}\right)\G\left(-\frac{1}{2}\right)}{}_2F_{1}\left(1,\frac{3}{2};2-\frac{z}{2};-\frac{1}{t^2}\right)\nonumber\\
&\quad+\frac{\G\left(\frac{1}{2}\right)\G\left(1-\frac{z}{2}\right)}{t^z\G\left(\frac{1-z}{2}\right)}{}_1F_{0}\left(\frac{1+z}{2};-;-\frac{1}{t^2}\right).
\end{align*}
Since $-\frac{1}{2}<\mathrm{Re}(z)<\frac{1}{2}$, for $t>M$, where $M$ is large enough, we have
\begin{equation*}
{}_2F_1\left(1,\frac{z}{2};\frac{1}{2};-t^2\right)=O_z\left(t^{-\textup{Re}(z)}\right).
\end{equation*}
Employing the above estimate in \eqref{ithree} along with the fact that $\Theta\in\Diamond_{\eta, \omega}$, we find that
\begin{align*}
I_3(n,z, w)&\ll_z n^{\frac{1}{2}\textup{Re}(z)-\eta-1-\frac{1}{2}|\textup{Re}(z)|}\bigg[\int_{1}^Mt^{-\eta-1-\frac{1}{2}|\textup{Re}(z)|}\left|{}_2F_1\left(1,\frac{z}{2};\frac{1}{2};-t^2\right)-1\right|\, dt\nonumber\\
&\quad\quad+\int_{M}^{\infty}\left(t^{-\eta-\frac{1}{2}|\textup{Re}(z)|-\frac{1}{2}\textup{Re}(z)-2}+t^{-\eta-\frac{1}{2}|\textup{Re}(z)|+\frac{1}{2}\textup{Re}(z)-2}\right)\, dt\bigg]\nonumber\\
&\ll_z n^{\frac{1}{2}\textup{Re}(z)-\eta-1-\frac{1}{2}|\textup{Re}(z)|},
\end{align*}
which implies that $\sum_{n=1}^{\infty}\sigma_{-z}(n)I_{3}(n, z, w)$ converges. This completes the proof of the fact that the series in \eqref{serconv} converges for $-\frac{1}{2}<\mathrm{Re}(z)<\frac{1}{2}$.

We now prove \eqref{analogueReslut}. First note that if 
\begin{equation}\label{fdefi}
f(t, z)=\phi(it, z)\phi(-it, z),
\end{equation}
where $\phi(\xi, z)$ is analytic in both $\xi$ and $z$, then
\begin{align}\label{analogue}
&\int_0^\infty f\left(\frac{t}{2}, z\right)\Xi\left(\frac{t+iz}{2}\right)\Xi\left(\frac{t-iz}{2}\right)Z\left(\frac{1+it}{2},\frac{z}{2}, w\right)\, dt\nonumber\\
&=\frac{1}{i}\int_{(\frac{1}{2})}\phi\left(s-\frac{1}{2},z\right)\phi\left(\frac{1}{2}-s,z\right)\xi\left(s-\frac{z}{2}\right)\xi\left(s+\frac{z}{2}\right)Z\left(s,\frac{z}{2}, w\right)\, ds.
\end{align}
This is proved by converting the integral on the left into the one along the whole real line and then letting $s=\frac{1+it}{2}$ to convert it into a line integral. The former is easy to see since the integrand of the integral on the left is an even function of $t$, which can, in turn, be seen from the fact that \eqref{xi} and \eqref{zetaalt} imply
\begin{equation*}
\Xi\left(-t\pm\frac{iz}{2}\right)=\Xi\left(t\mp\frac{iz}{2}\right),
\end{equation*}
and since $f$ and $Z\left(\frac{1+it}{2},\frac{z}{2}, w\right)$ are even functions of $t$. Note that the evenness of $Z\left(\frac{1+it}{2},\frac{z}{2}, w\right)$ follows from part (1) of Lemma \ref{parts12} since $-\frac{1}{2}<\mathrm{Re}(z)<\frac{1}{2}$. 

We now specialize $f$ in \eqref{fdefi} by choosing  
\begin{equation*}
\phi(s,z)=\frac{1}{\left(s+\frac{z+1}{2}\right)}\Gamma\left(\frac{z-1}{4}+\frac{s}{2}\right).
\end{equation*}
Then \eqref{analogue} implies
\begin{align}\label{analogue3}
&4\int_0^\infty\Gamma\left(\frac{z-1+it}{4}\right)\Gamma\left(\frac{z-1-it}{4}\right) \Xi\left(\frac{t+iz}{2}\right)\Xi\left(\frac{t-iz}{2}\right)Z\left(\frac{1+it}{2},\frac{z}{2}\right)\frac{dt}{(z+1)^2+t^2} \nonumber \\
&=-\frac{1}{i}\int_{(\frac{1}{2})}H(s, z, w)\pi^{-s}\, ds,
\end{align}
where
\begin{align*}
H(s, z, w)&:=\Gamma\left(\frac{z}{4}-\frac{s}{2}\right)\Gamma\left(\frac{s}{2}-\frac{z}{4}+1\right)\Gamma\left(\frac{z}{4}+\frac{s}{2}\right)\Gamma\left(\frac{z}{4}+\frac{s}{2}+\frac{1}{2}\right)\nonumber\\
&\qquad\times\zeta\left(s-\frac{z}{2}\right)\zeta\left(s+\frac{z}{2}\right)Z\left(s,\frac{z}{2}, w\right),
\end{align*}
where we used the definition of $\xi(s)$ in \eqref{xii} and the functional equation $\G(\omega+1)=\omega\G(\omega)$ for simplification.

Next, we would like to represent $\zeta\left(s-\frac{z}{2}\right)\zeta\left(s+\frac{z}{2}\right)$ as an infinite series using the well-known Dirichlet series representation \cite[p.~8, Equation (1.3.1)]{titch}
\begin{equation}\label{doublezeta}
\zeta(s)\zeta(s-a) = \sum\limits_{n = 1}^\infty  \frac{{\sigma _{a}(n)}}{n^{s}},
\end{equation}
valid for Re$(s)>\max(1,1+$Re$(a))$. However, note that the conditions $-1/2 < $ Re$(z)<1/2$ and Re$(s)=1/2$ imply $1/4 < $ Re$(s\pm z/2)<3/4$. Thus in order to use \eqref{doublezeta}, with $s$ replaced by $s-z/2$ and $a$ replaced by $-z$ (which is then valid for Re$(s) > 1 \pm $ Re$\left(\frac{z}{2}\right)$), we need to shift the line of integration from Re$(s)=1/2$ to Re$(s)=5/4$. In doing so, we encounter a simple pole at $s=1+z/2$ (due to $\zeta\left(s-\frac{z}{2}\right)$) and a simple pole at $s=1-z/2$ (due to $\zeta\left(s+\frac{z}{2}\right)$). 

After an application of Cauchy's residue theorem by considering the contour formed by the line segments $[1/2-iT, 5/4-iT],\ [5/4-iT, 5/4+iT],\ [5/4+iT, 1/2+iT]$ and $[1/2+iT, 1/2-iT]$ and noting that integral along horizontal lines vanishes as $T\to \infty$ (as can be seen by invoking \eqref{strivert}), we get
\begin{align}\label{analogue4}
\int_{(\frac{1}{2})}H(s, z, w)\pi^{-s}\, ds&=\int_{(\frac{5}{4})}H(s, z, w)\pi^{-s}\, ds-2\pi i\left(R_{1+\frac{z}{2}}+R_{1-\frac{z}{2}}\right),
\end{align}
where, here and throughout the rest of the paper, we use the notation $R_b$ to denote the residue of the integrand of the associated integral at $b$. Now employing \eqref{doublezeta}, with $s$ replaced by $s-z/2$ and $a$ replaced by $-z$, and interchanging the order of summation and integration which is valid because of absolute convergence, we see from \eqref{analogue3} and \eqref{analogue4} that
\begin{align}\label{analogue5}
&4\int_0^\infty \Gamma\left(\frac{z-1+it}{4}\right)\Gamma\left(\frac{z-1-it}{4}\right) \Xi\left(\frac{t+iz}{2}\right)\Xi\left(\frac{t-iz}{2}\right)Z\left(\frac{1+it}{2},\frac{z}{2}\right)\frac{dt}{(z+1)^2+t^2} \nonumber \\
&=-\frac{1}{i}\Bigg(\sum_{n=1}^\infty\sigma_{-z}(n)n^{\frac{z}{2}}\int_{(\frac{5}{4})}\Gamma\left(\frac{z}{4}-\frac{s}{2}\right)\Gamma\left(\frac{s}{2}-\frac{z}{4}+1\right)\Gamma\left(\frac{z}{4}+\frac{s}{2}\right)\Gamma\left(\frac{z}{4}+\frac{s}{2}+\frac{1}{2}\right)  \nonumber\\
&\qquad\quad\times Z\left(s,\frac{z}{2}\right) (n\pi)^{-s} ds-2\pi i\left(R_{1+\frac{z}{2}}+R_{1-\frac{z}{2}}\right)\Bigg).
\end{align}
The residues are easily computed to be
\begin{align}\label{resid}
R_{1+\frac{z}{2}}&=-2^{-z}\pi^{(1-z)/2}\Gamma(1+z)\zeta(1+z)Z\left(1+\frac{z}{2},\frac{z}{2}, w\right),\nonumber\\
R_{1-\frac{z}{2}}&=-2^{1-z}\pi^{(1-z)/2}\Gamma(z)\zeta(z)Z\left(1-\frac{z}{2},\frac{z}{2}, w\right).
\end{align}
We now apply \eqref{Persval} with $g(x)=\Theta(x, \frac{z}{2}, w)$ and $h(x)=\frac{2}{\sqrt{\pi}}x^{-\frac{z}{2}} \Gamma\left(\frac{z}{2}\right) \left({}_2F_1\left(1,\frac{z}{2};\frac{1}{2};-x^{-2}\right)-1\right)$. It is easy to see that the conditions \eqref{conditions} needed for using \eqref{Persval} are satisfied. Thus using Lemma \ref{InverseMellOfQuetientOfGamma},

\begin{align}\label{analogue6}
&\int_{(\frac{5}{4})}\Gamma\left(\frac{z}{4}-\frac{s}{2}\right) \Gamma\left(\frac{s}{2}-\frac{z}{4}+1\right)\Gamma\left(\frac{z}{4}+\frac{s}{2}\right)\Gamma\left(\frac{z}{4}+\frac{s}{2}+\frac{1}{2}\right) Z\left(s,\frac{z}{2}, w\right) (n\pi)^{-s} ds \nonumber \\
&=4\sqrt{\pi}i \Gamma\left(\frac{z}{2}\right) \int_0^\infty \Theta\left(x,\frac{z}{2}, w\right)\left({}_2F_1\left(1,\frac{z}{2};\frac{1}{2};-\frac{x^2}{\pi^2n^2}\right)-1\right)\left(\frac{\pi n}{x}\right)^{-\frac{z}{2}}\frac{dx}{x}.
\end{align}
Substituting \eqref{resid} and \eqref{analogue6} in \eqref{analogue5} and then simplifying leads us to \eqref{analogueReslut}.

\end{proof}

\section{Theory of the generalized modified Bessel function ${}_1K_{z,w}(x)$}\label{1kzw}
The theory of the generalized modified Bessel function ${}_1K_{z,w}(x)$ defined in \eqref{def} is developed in this section. It is essential for evaluating an integral involving the Riemann $\Xi$-function stated in Theorem \ref{xiintgenramhurthm}. The proofs of the properties of ${}_1K_{z,w}(x)$ are similar in nature to those occurring in the theory of another generalization of the modified Bessel function $K_{z,w}(x)$ defined in \eqref{kzw}. Since the latter was developed in detail in \cite{dkmt}, our approach here will be terse. 

\subsection{Series and integral representations for ${}_1K_{z,w}(x)$}\label{sir}
\hfill\\

We begin by proving Theorem \ref{integralRepr}.

\begin{proof}[Theorem \textup{\ref{integralRepr}}][]
Replacing $s$ by $s-z$ and letting $a=1, b=w$ in \eqref{mellin transform of gamma 1F1 1}, we see that for $c_1:=$Re$(s)>-1+$Re$(z)$,
\begin{align}\label{int2two}
\frac{1}{2\pi i}\int_{(c_1)}   \Gamma\left(\frac{1+s-z}{2}\right){}_1F_1\left(\frac{1+s-z}{2};\frac{3}{2};-\frac{w^2}{4}\right)t^{-s}ds=\frac{2t^{-z}}{w}e^{-t^2}\sin (wt).
\end{align}
Similarly, for $c_2:=$Re$(s)>-1-$Re$(z)$,
\begin{align}\label{int3}
\frac{1}{2\pi i}\int_{(c_2)}   \Gamma\left(\frac{1+s+z}{2}\right){}_1F_1\left(\frac{1+s+z}{2};\frac{3}{2};-\frac{w^2}{4}\right)t^{-s}ds=\frac{2t^{z}}{w}e^{-t^2}\sin (wt).
\end{align}
From \eqref{int2two}, \eqref{int3}, \eqref{Persval} and the definition of ${}_1K_{z,w}(x)$ in \eqref{def}, we arrive at \eqref{integralRepreqn}.
\end{proof}
To find the asymptotic behavior of ${}_1K_{z,w}(x)$, we need to first prove the lemma given below and the theorem following it.
\begin{lemma}\label{doubleSeriesRepr}
For $z,w\in\mathbb{C}$ and $|\arg(x)|<\frac{\pi}{4}$, we have
\begin{align*}
{}_1K_{z,w}(2x)=2x\sum_{n=0}^{\infty}\sum_{m=0}^{\infty}\frac{(-w^2x)^{n+m}}{(2n+1)!(2m+1)!}K_{n-m+z}(2x).
\end{align*}
\end{lemma}
\begin{proof}
This is proved by starting with \eqref{integralRepreqn}, writing the sine functions in the form of their Taylor series, interchange the order of integration and double sum using Theorem \ref{ldcttemme}, and then employing the formula \cite[p.~344, Formula \textbf{2.3.16.1}]{Prudnikov}
\begin{equation}\label{PrudnikovFormula}
\int_0^\infty y^{s-1}e^{-py-q/y}dy=2\left(\frac{q}{p}\right)^{s/2} K_s(2\sqrt{pq}),
\end{equation}
valid for Re$(p)>0,$ Re$(q)>0$.
\end{proof}

Basset's integral for the modified Bessel function of the second kind is given by \cite[p.~172]{watson-1966a}
\begin{align*}
K_z(xy)=\frac{\Gamma\left(z+\frac{1}{2}\right)(2x)^z}{y^z \Gamma\left(\frac{1}{2}\right)}\int_0^\infty
\frac{\cos (yu)\, du}{(x^2+u^2)^{z+\frac{1}{2}}},
\end{align*}
where Re$(z)>-1/2$, $y>0$ and $|\arg(x)|<\frac{\pi}{2}$. Performing integration by parts, we obtain
\begin{align}\label{Besset2}
K_z(xy)=\frac{2\Gamma\left(z+\frac{3}{2}\right)(2x)^z}{y^{z+1} \Gamma\left(\frac{1}{2}\right)}\int_0^\infty
\frac{u\sin (yu)}{(x^2+u^2)^{z+\frac{3}{2}}}\, du,
\end{align}
also valid for Re$(z)>-1/2$, $y>0$ and $|\arg(x)|<\frac{\pi}{2}$.
\begin{theorem}\label{Biletral}
For\footnote{The conditions in the corresponding theorem given in \cite{dkmt}, namely Theorem 1.9, are too restrictive. It is actually valid for $z, w\in\mathbb{C}$, and $|\arg(x)|<\pi$.} $z, w\in\mathbb{C}$, and $|\arg(x)|<\pi$,
\begin{align}\label{Biletraleqn}
{}_1K_{z,w}(2x)=\frac{1}{w^2}\sum_{m=-\infty}^\infty (-1)^m K_{m+z}(2x)(I_{2m}(2w\sqrt{x})-J_{2m}(2w\sqrt{x})).
\end{align}
\end{theorem}
\begin{proof}
We first prove the result for $-\frac{1}{2}<\textup{Re}(z)<\frac{1}{2}$ and $|\arg(x)|<\frac{\pi}{4}$, and later extend it to any $z\in\mathbb{C}$ and $x\in\mathbb{C}\backslash\{x\in\mathbb{R}:x\leq 0\}$ by analytic continuation.

From \eqref{Besset2} with $y=2$,
\begin{align}\label{bilateralSeries1}
K_{n-m+z}(2x) = \left\{ \begin{array}{cc}
                \frac{x^{n-m+z}\Gamma\left(n-m+z+\frac{3}{2}\right)}{\sqrt{\pi}}\displaystyle\int_0^\infty
\frac{u\sin (2u)}{(x^2+u^2)^{n-m+z+\frac{3}{2}}}\, du, &  \mathrm{if}\ \mathrm{Re}(n-m+z)> -\frac{1}{2}, \\
\frac{x^{m-n-z}\Gamma\left(m-n-z+\frac{3}{2}\right)}{\sqrt{\pi}}\displaystyle\int_0^\infty
\frac{u\sin (2u)}{(x^2+u^2)^{m-n-z+\frac{3}{2}}}\, du, & \mathrm{if}\ \mathrm{Re}(m-n-z)> -\frac{1}{2}.
                \end{array} \right. 
\end{align}
From Lemma \ref{doubleSeriesRepr},
\begin{align}\label{bilateralSeries2}
{}_1K_{z,w}(2x)&=2x\sum_{n=0}^\infty\sum_{m=0}^n\frac{(-w^2x)^{m+n}}{(2n+1)!(2m+1)!}K_{n-m+z}(2x) \nonumber \\
&\quad+2x\sum_{n=0}^\infty\sum_{m=n+1}^\infty\frac{(-w^2x)^{m+n}}{(2n+1)!(2m+1)!}K_{n-m+z}(2x).
\end{align}
Using the fact that $-\frac{1}{2}<\textup{Re}(z)<\frac{1}{2}$ and invoking \eqref{bilateralSeries1} in \eqref{bilateralSeries2}, we see that
\begin{align}\label{sumOfS1S2}
{}_1K_{z,w}(2x)=:S_1(z,w,x)+S_2(z,w,x),
\end{align}
where
\begin{align}
S_1(z,w,x)&:=\frac{2x}{\sqrt{\pi}}\sum_{n=0}^\infty\sum_{m=0}^n\frac{(-w^2)^{m+n}x^{2n+z}}{(2n+1)!(2m+1)!}\Gamma\left(n-m+z+\frac{3}{2}\right)\int_0^\infty
\frac{u\sin (2u)\, du}{(x^2+u^2)^{n-m+z+\frac{3}{2}}}\nonumber\\
S_2(z,w,x)&:=\frac{2x}{\sqrt{\pi}}\sum_{n=0}^\infty\sum_{m=n+1}^\infty\frac{(-w^2)^{m+n}x^{2m-z}}{(2n+1)!(2m+1)!}\Gamma\left(m-n-z+\frac{3}{2}\right)\int_0^\infty
\frac{u\sin (2u)\, du}{(x^2+u^2)^{m-n-z+\frac{3}{2}}}\label{S2def}.
\end{align}
To evaluate $S_1(z,w,x)$, we first write it in the form of a doubly infinite series as
\begin{align}\label{8.10}
S_1(z,w,x)=\frac{2x}{\sqrt{\pi}}\sum_{m=0}^\infty\sum_{k=0}^\infty\frac{(-w^2)^{2m+k}x^{2m+2k+z}}{(2m+1)!(2m+2k+1)!}\Gamma\left(k+z+\frac{3}{2}\right)\int_0^\infty
\frac{u\sin (2u)}{(x^2+u^2)^{k+z+\frac{3}{2}}}\, du.
\end{align}
Interchanging the order of the double sum, which is justified by absolute convergence, we find that
\begin{align}\label{ess1zwx}
S_1(z,w,x)&=\frac{2x^{1+z}}{\sqrt{\pi}}\sum_{k=0}^\infty(-w^2x^2)^k\G\left(k+z+\frac{3}{2}\right)\sum_{m=1}^{\infty}\frac{w^{4m}x^{2m}}{(2m+1)!(2m+2k+1)!}\nonumber\\
&\quad\times\int_{0}^{\infty}\frac{u\sin (2u)\, du}{(x^2+u^2)^{k+z+\frac{3}{2}}}.
\end{align}
Observe that
\begin{equation}\label{sumoverm}
\sum_{m=0}^{\infty}\frac{w^{4m}x^{2m}}{(2m+1)!(2m+2k+1)!}=\frac{1}{2(w^2x)^{k+1}}\left(I_{2k}(2w\sqrt{x})-J_{2k}(2w\sqrt{x})\right).
\end{equation}
Replace $z$ by $k+z$ and let $y=2$ in \eqref{Besset2} so that for Re$(z)<-k-1/2$,
\begin{equation}\label{bintegral}
\int_{0}^{\infty}\frac{u\sin (2u)\, du}{(x^2+u^2)^{k+z+\frac{3}{2}}}=\frac{\sqrt{\pi}}{x^{z+k}\G(k+z+\frac{3}{2})}K_{k+z}(2x).
\end{equation}
Substituting \eqref{sumoverm} and \eqref{bintegral} in \eqref{ess1zwx}, we are led to
\begin{align}\label{s1expr}
S_1(z,w,x)=\frac{1}{w^2}\sum_{k=0}^{\infty}(-1)^k K_{k+z}(2x)(I_{2k}(2w\sqrt{x})-J_{2k}(2w\sqrt{x}))
\end{align}
for Re$(z)>-1/2$. Replace $m$ by $n+\ell$ in \eqref{S2def} and then use \eqref{8.10} to deduce that
\begin{align}\label{ess2ess}
S_2(z,w,x)&=\frac{2x}{\sqrt{\pi}}\sum_{n=0}^\infty\sum_{\ell =1}^\infty \frac{(-w^2)^{2n+\ell}x^{2n+2\ell-z}\Gamma\left(\ell -z+\frac{3}{2}\right)}{(2n+1)!(2\ell +2n+1)!}\int_0^\infty\frac{u\sin(2u)}{(u^2+x^2)^{\ell-z+\frac{3}{2}}}\, du \nonumber\\
&=S_1(-z,w,x)- \frac{2x}{\sqrt{\pi}}\sum_{n=0}^\infty \frac{(-w^2)^{2n}x^{2n-z}\Gamma\left(\frac{3}{2}-z\right)}{(2n+1)!(2n+1)!}\int_0^\infty\frac{u\sin(2u)}{(u^2+x^2)^{\frac{3}{2}-z}}\, du\nonumber\\
&=S_1(-z,w,x)-\frac{1}{w^2}K_{-z}(x)\left(I_0(2w\sqrt{x})-J_0(2w\sqrt{x})\right),
\end{align}
where we employed \eqref{sumoverm} with $k=0$ and \eqref{Besset2} with $z$ replaced by $-z$ which is legitimate since Re$(z)<1/2$. Thus from \eqref{s1expr} and \eqref{ess2ess},
\begin{align}\label{s2expr}
S_2(z,w,x)&=\frac{1}{w^2}\sum_{k=1}^{\infty}(-1)^k K_{k-z}(2x)(I_{2k}(2w\sqrt{x})-J_{2k}(2w\sqrt{x}))\nonumber\\
&=\frac{1}{w^2}\sum_{k=-\infty}^{-1}(-1)^k K_{k+z}(2x)(I_{2k}(2w\sqrt{x})-J_{2k}(2w\sqrt{x})),
\end{align}
where we used the well-known facts that $J_{-n}(\xi)=(-1)^nJ_{n}(\xi)$ and $I_{-n}(\xi)=I_{n}(\xi)$ and $K_{-\nu}(\xi)=K_{\nu}(\xi)$. The result now follows from substituting \eqref{s1expr} and \eqref{s2expr} in \eqref{sumOfS1S2}. This proves Theorem \ref{Biletral} for $-\frac{1}{2}<\textup{Re}(z)<\frac{1}{2}$ and $|\arg(x)|<\frac{\pi}{4}$. 

Now from the series definitions of $J_{\nu}(\xi)$, $I_{\nu}(\xi)$ and $K_{\nu}(\xi)$, it is easy to see that as $\nu\to\infty, |\arg(\nu)|<\pi-\delta, \delta>0$ and $\xi\neq0$,
\begin{align*}
J_{\nu}(\xi)&\sim\frac{1}{\sqrt{2\pi\nu}}\left(\frac{e\xi}{2\nu}\right)^{\nu}-\left(\frac{z}{2}\right)^{\nu+2}\frac{e^{\nu}}{\nu^{\nu+3/2}\sqrt{2\pi}},\\
I_{\nu}(\xi)&\sim\frac{1}{\sqrt{2\pi\nu}}\left(\frac{e\xi}{2\nu}\right)^{\nu}+\left(\frac{z}{2}\right)^{\nu+2}\frac{e^{\nu}}{\nu^{\nu+3/2}\sqrt{2\pi}},\\
K_{\nu}(\xi)&\sim\frac{\sqrt{\pi}}{2\nu}\left(\frac{e\xi}{2\nu}\right)^{-\nu}.
\end{align*}
These easily show that the bilateral series in \eqref{Biletraleqn} is uniformly convergent not only on compact subsets of $z\in\mathbb{C}$ but also on compact subsets of $x\in\mathbb{C}\backslash\{x\in\mathbb{R}:x\leq 0\}$ and hence represents an analytic function for any complex $z$ and $x\in\mathbb{C}\backslash\{x\in\mathbb{R}:x\leq 0\}$. Since ${}_1K_{z,w}(2x)$ is also analytic in these regions as can be seen from \eqref{def}, we see that Theorem \ref{Biletral} holds for these extended regions in $z$- and $x$-complex planes.
\end{proof}
\begin{remark}
Let $w\to 0$ on both sides of Theorem \textup{\ref{Biletral}}. This results in the trivial identity $2xK_{z}(2x)=2xK_{z}(2x)$ as can be seen from \eqref{Basicform} and the fact that
\begin{equation*}
\lim_{w\to0}\frac{I_{2m}\left(2w\sqrt{x}\right)-J_{2m}\left(2w\sqrt{x}\right)}{w^2}=
\begin{cases}
2x,\hspace{1mm}\text{if}\hspace{1.5mm}m=0,\\
0,\hspace{2.5mm}\text{if}\hspace{1.5mm}m\neq0.\\
\end{cases}
\end{equation*}
 \end{remark}

\subsection{Asymptotic estimates of ${}_1K_{z,w}(x)$ }\label{ae1kzw}
\hfill\\

The asymptotic estimate for ${}_1K_{z,w}(x)$ for large values of $|x|$ is given first.
\begin{theorem}\label{asymptotcs large x}
Let the complex variables $w$ and $z$ belong to compact domains, then for large values of $|x|$, $|\arg(x)|<\frac{1}{4}\pi$, we have
\begin{align}\label{alx}
{}_1K_{z,w}(2x)=-\frac{1}{2w^2}\sqrt{\frac{\pi}{x}}e^{-2x}\left(\cos(2w\sqrt{x})P-\sin(2w\sqrt{x})Q-e^{-\frac{1}{4}w^2}R\right),
\end{align}
where $P,\ Q$ and $R$ have the asymptotic expansions
\begin{align}\label{PQR}
P&=1+\frac{32z^2-3w^2-8}{128x}+O\left(x^{-2}\right),\nonumber\\
Q&=\frac{w}{8\sqrt{x}}+O\left(x^{-\frac{3}{2}}\right),\nonumber\\
R&=1+\frac{(4z^2-1)(2-w^2)}{32x}+O\left(x^{-2}\right).
\end{align}
\end{theorem}
\begin{proof}
Employ change of variable $u=ys$, $y=\sqrt{x}$ in Theorem \ref{integralRepr} and represent the sine functions in terms of exponential function so that
\begin{align}\label{Iii1}
{}_1K_{z,w}(2x)=-\frac{1}{2w^2}\left((I_{i,i}+I_{-i,-i}-\left(I_{i,-i}+I_{-i,i}\right)\right),
\end{align}
\begin{align*}
I_{\sigma,\tau}&=\int_0^\infty e^{-y^2\varphi_{\sigma,\tau}(s)}s^{2z-1}\ ds,\nonumber\\
\varphi_{\sigma,\tau}(s)&=s^2+\frac{1}{s^2}+\frac{w}{y}\left(\sigma s+\frac{\tau}{s}\right),
\end{align*}
and \begin{align*}
\sigma=\pm i,\ \tau=\pm i.
\end{align*}
From \cite[p. 416, Equation A.4,\ A.6]{dkmt}, we have
\begin{align}\label{Iii2}
I_{i,i}+I_{-i,-i}&=\frac{\sqrt{\pi}}{y}e^{-2y^2}\left(\cos(2wy)P-\sin(2wy)Q\right),\nonumber\\
I_{i,-i}+I_{-i,i}&=\frac{\sqrt{\pi}}{y}e^{-2y^2-\frac{1}{4}w^2}R,
\end{align}
where $P, Q$, and $R$ are defined in \eqref{PQR}. Now let $y=\sqrt{x}$ in each of the equations in \eqref{Iii2} and substitute the resulting ones in \eqref{Iii1} so as to obtain \eqref{alx}.
\end{proof}
The above result gives the familiar asymptotic estimate for $K_z(x)$ as $x\to\infty$ as its special case.
\begin{corollary}
For $|x|$ large enough and such that $|\mathrm{arg}(x)|<\frac{\pi}{4}$, we have
\begin{align}\label{kspl}
K_z(2x)=\frac{1}{2}\sqrt{\frac{\pi}{x}}e^{-2x}\left(1+\frac{4z^2-1}{16x}+O\left(\frac{1}{x^2}\right)\right).
\end{align}
\end{corollary}
\begin{proof}
For $|x|$ large enough,
{\allowdisplaybreaks\begin{align}\label{lim1}
&\lim_{w\rightarrow 0}\frac{1}{w^2}\left\{\cos(2w\sqrt{x})P-1\right\}\nonumber\\
&=\lim_{w\rightarrow0}\frac{1}{w^2}\left\{\left(1-\frac{4w^2x}{2!}+\frac{16w^4x^2}{4!}-...\right)\left(1+\frac{32z^2-3w^2-8}{128x}+O\left(x^{-2}\right)\right)-1\right\}\nonumber\\
&=-2x-\frac{4z^2-1}{8}+O\left(\frac{1}{x}\right),
\end{align}}
as well as
\begin{align}\label{lim2}
\lim_{w\rightarrow 0}\frac{1}{w^2}\sin(2w\sqrt{x})Q&=\lim_{w\rightarrow 0}\frac{1}{w^2}\left(2w\sqrt{x}-\frac{(2w\sqrt{x})^3}{3!}+...\right)\left(\frac{w}{8\sqrt{x}}+O\left(x^{-\frac{3}{2}}\right)\right)\nonumber\\
&=\frac{1}{4}+O\left(\frac{1}{x}\right).
\end{align}
Again, for large value of $|x|$,
\begin{align}\label{lim3}
&\lim_{w\rightarrow 0}\frac{1}{w^2}\left\{e^{-\frac{w^2}{4}}R-1\right\}\nonumber\\
&=\lim_{w\rightarrow 0}\frac{1}{w^2}\left\{\left(1-\frac{w^2}{4}+\frac{w^4}{2! 16}-...\right)\left(1+\frac{(4z^2-1)(2-w^2)}{32x}+O\left(x^{-2}\right)\right)-1\right\}\nonumber\\
&=-\frac{1}{4}+O\left(\frac{1}{x}\right).
\end{align}
Upon letting $w\rightarrow0$ in \eqref{alx} and using \eqref{lim1}, \eqref{lim2} and \eqref{lim3}, we see that
\begin{align*}
\lim_{w\rightarrow0}{}_1K_{z,w}(2x)&=-\frac{1}{2}\sqrt{\frac{\pi}{x}}e^{-2x}\left(-2x-\frac{4z^2-1}{8}+O\left(\frac{1}{x}\right)\right),
\end{align*}
which implies \eqref{kspl} upon noting \eqref{Basicform}.
\end{proof}
We now give the asymptotic formula for ${}_1K_{z,w}(x)$ valid for small values of $|x|$.
\begin{theorem}\label{asymSmall}
Let $w\in\mathbb{C}$ be fixed.

\textup{(i)} Consider a fixed $z$ such that \textup{Re}$(z)>0$. Let $\mathfrak{D}=\{x\in\mathbb{C}:|\arg(x)|<\frac{\pi}{4}\}$. Then as $x\to 0$ along any path in $\mathfrak{D}$, we have
\begin{equation*}
{}_1K_{z,w}(x)\sim\Gamma(z)\left(\frac{x}{2}\right)^{1-z}{}_1F_1\left(z;\frac{3}{2};-\frac{w^2}{4}\right).
\end{equation*}
\textup{(ii)} Let $|\arg(x)|<\pi$. As $x\to 0$,
\begin{equation}\label{kzwsmallii}
{}_1K_{0,w}(x)\sim -x\log x-\frac{w^2x}{6}{}_2F_{2}\left(1,1;2,\frac{5}{2};-\frac{w^2}{4}\right).
\end{equation}
\end{theorem}
\begin{proof}
To prove (i), we use \eqref{integralRepreqn} with $x$ replaced by $x/2$ and write it in the form
\begin{align}\label{Asy}
{}_1K_{z,w}(x)=\frac{2^{z+1}x^{1-z}}{w}\frac{1}{xw}\int_0^\infty t^{2z-1} e^{-t^2-\frac{x^2}{4t^2}}\sin(wt)\sin\left(\frac{wx}{2t}\right)\, dt.
\end{align}
Now
\begin{align}\label{asymSmall1}
\lim_{x\to 0}\frac{1}{xw}\int_0^\infty t^{2z-1}e^{-t^2-\frac{x^2}{t^2}}\sin(wt)\sin\left(\frac{wx}{2t}\right)\mathrm{d}t &= \int_0^\infty \lim_{x\to 0} t^{2z-1}e^{-t^2-\frac{x^2}{4t^2}}\sin(wt)\frac{\sin\left(\frac{wx}{2t}\right)}{2t\left(\frac{wx}{2t}\right)}\, dt \nonumber \\
&=\frac{1}{2} \int_0^\infty  t^{2z-2}e^{-t^2}\sin(wt)\, dt\nonumber\\
&=\frac{w}{4}\Gamma(z){}_1F_1\left(z;\frac{3}{2};-\frac{w^2}{4}\right),
\end{align}
where in the last step, we used \eqref{equivv} with $a=1,\ b=w$ and $s=2z-1$, and then used \eqref{kft}. The result now follows from \eqref{Asy} and \eqref{asymSmall1}.

To prove part (ii), it is convenient to work with ${}_1K_{0,w}(x)/x$ initially. Invoking Theorem \ref{Biletral}, we observe that
\begin{align}\label{k0w}
\frac{{}_1K_{0,w}(x)}{x}=\frac{1}{w^2}\frac{K_0(x)}{x}(I_0(w\sqrt{2x})-J_0(w\sqrt{2x}))+\frac{2}{w^2x}\sum_{n=1}^\infty K_n(x)\left(I_{2n}(w\sqrt{2x})-J_{2n}(w\sqrt{2x})\right).
\end{align}
First, using series definitions of $I_{2n}(w\sqrt{2x})$ and $J_{2n}(w\sqrt{2x})$, it is easy to see that for $n\geq0$,
\begin{align}\label{lim34}
\lim_{x\rightarrow0}\frac{1}{x^{n+1}}\left(I_{2n}(w\sqrt{2x})-J_{2n}(w\sqrt{2x})\right)&=\frac{w^{2n+2}}{2^{n}\Gamma(2n+2)}.
\end{align}
Further, from \cite[p.~375, Equations (9.7.1), (9.7.2)]{as}, as $x\to0$,
\begin{equation}\label{kzxasy0}
K_{z}(x)\sim\begin{cases}
 \frac{1}{2}\G(z)\left(\frac{x}{2}\right)^{-z}, \text{if}\hspace{1mm}\textup{Re }z>0,\\
-\log x,\hspace{9mm} \text{if}\hspace{1mm}z=0.
\end{cases}
\end{equation}
Hence using \eqref{lim34} and the first part of \eqref{kzxasy0}, we see that
\begin{align}\label{nonzero part}
&\lim_{x\rightarrow0}\frac{2}{w^2x}\sum_{n=1}^\infty (-1)^nK_n(x)\left(I_{2n}(w\sqrt{2x})-J_{2n}(w\sqrt{2x})\right)\nonumber\\
&=\frac{2}{w^2}\sum_{n=1}^\infty (-1)^n\left(\lim_{x\to0}x^nK_{n}(x)\right)\left(\lim_{x\rightarrow0}\frac{1}{x^{n+1}}\left(I_{2n}(w\sqrt{2x})-J_{2n}(w\sqrt{2x})\right)\right)\nonumber\\
&=\frac{1}{w^2}\sum_{n=1}^\infty\frac{(-1)^n\Gamma(n)w^{2n+2}}{\Gamma(2n+2)}\nonumber\\
&=-\frac{w^2}{6}{}_2F_2\left(1,1;\frac{5}{2},2;-\frac{w^2}{4}\right).
\end{align}
Next, from the $n=0$ case of \eqref{lim34} and the second part of \eqref{kzxasy0},
\begin{align}\label{zeropart}
\lim_{x\to 0}\frac{1}{w^2}\frac{K_0(x)}{x}(I_0(w\sqrt{2x})-J_0(w\sqrt{2x}))=-\log x.
\end{align}
From \eqref{k0w}, \eqref{nonzero part} and \eqref{zeropart}, as $x\to0$,
\begin{align*}
\frac{{}_1K_{0,w}(x)}{x}\sim-\log x-\frac{w^2}{6}{}_2F_2\left(1,1;\frac{5}{2},2;-\frac{w^2}{4}\right),
\end{align*}
which implies \eqref{kzwsmallii}.
\end{proof}

\subsection{Further properties of ${}_1K_{z,w}(x)$}\label{furprop}
\hfill\\

We commence this subsection by generalizing Basset's integral representation \eqref{Besset2} for ${}_1K_{0,w}(x)$. We begin with a lemma.
\begin{lemma}\label{bessetTypeLemma}
For $|\arg(x)|<\frac{\pi}{4}$ and $w\in\mathbb{C}$,
\begin{align*}
\int_0^\infty e^{-t^2-\frac{x^2}{t^2}}\sin(wt)\frac{\, dt}{t^2}=\frac{w}{2}\int_0^\infty \exp\left(-\frac{w^2x^2}{4(x^2+u^2)}\right)\frac{u\sin(2u)}{(x^2+u^2)^{3/2}}\, du
\end{align*}
\end{lemma} 
\begin{proof}
Observe that
\begin{align*}
\int_0^\infty e^{-t^2-\frac{x^2}{t^2}}\sin(wt)\frac{\mathrm{dt}}{t^2}&=\int_0^\infty e^{-t^2-\frac{x^2}{t^2}}\sum_{n=0}^\infty \frac{(-1)^n(wt)^{2n+1}}{(2n+1)!}\frac{\, dt}{t^2}\nonumber\\
&=\sum_{n=0}^\infty \frac{(-1)^nw^{2n+1}}{(2n+1)!}\int_0^\infty e^{-t^2-\frac{x^2}{t^2}} t^{2n-1}\, dt \nonumber\\
&=\sum_{n=0}^\infty \frac{(-1)^n w^{2n+1}x^n}{(2n+1)!}K_n(2x)\hspace{5mm}(\text{from}\eqref{PrudnikovFormula})\nonumber\\
&=\sum_{n=0}^\infty \frac{(-1)^n w^{2n+1}x^n}{(2n+1)!} \frac{\Gamma\left(n+\frac{3}{2}\right)x^n}{\Gamma\left(\frac{1}{2}\right)}\int_0^\infty \frac{u\sin(2u)}{(x^2+u^2)^{n+\frac{3}{2}}}\, du \hspace{5mm}(\text{from} \eqref{Besset2})\nonumber\\
&=\int_0^\infty \frac{u\sin(2u)}{(x^2+u^2)^{3/2}}\sum_{n=0}^\infty \frac{(-1)^nw^{2n+1}x^{2n}}{(x^2+u^2)^n}\frac{\Gamma\left(n+\frac{3}{2}\right)}{(2n+1)!\Gamma\left(\frac{1}{2}\right)}\, du \nonumber\\
&=\frac{w}{2}\int_0^\infty \exp\left(-\frac{w^2x^2}{4(x^2+u^2)}\right)\frac{u\sin(2u)}{(x^2+u^2)^{3/2}}\, du,
\end{align*}
where all interchanges of the order of summation and integration can be justified through Theorem \ref{ldcttemme}.
\end{proof}
\begin{theorem}\label{bessetTypeRepr}
For $|\arg(x)|<\frac{\pi}{4}$ and $w\in\mathbb{C}$, we have
\begin{align}\label{bassset}
{}_1K_{0,w}(x)&=\frac{2}{w^2}\int_0^\infty\exp\left(-\frac{w^2x^2}{2(x^2+u^2)}\right)\sin\left(\frac{w^2xu}{2(x^2+u^2)}\right)\frac{\sin u}{(x^2+u^2)^{1/2}}\mathrm{d}u.
\end{align}
\end{theorem}
\begin{proof}
First assume $x>0$. Define 
\begin{align}\label{defI}
I(x,w):=\frac{1}{2\pi i}\int_{(c)}\Gamma^2\left(\frac{1+s}{2}\right){}_1F_1^2\left(\frac{1+s}{2};\frac{3}{2};-\frac{w^2}{4}\right)x^{-s}\mathrm{d}s.
\end{align}
Then
\begin{equation*}
I\left(\frac{x}{2},w\right)=2\ {}_1K_{0,w}(x).
\end{equation*}
We make use of the following formula from \cite[p. 121, Equation (43)]{burchnallchaundy} to write ${}_1F_1^2$ in \eqref{defI} as a series:
\begin{align*}
{}_1F_1(a;c;u){}_1F_1(a;c;v)=\sum_{n=0}^\infty \frac{(a)_n(c-a)_n}{n!(c)_n(c)_{2n}}(-uv)^n {}_1F_1(a+n;c+2n;u+v).
\end{align*}
Along with \eqref{defI}, it gives
\begin{align}\label{bassettype2}
I(x,w)&=\frac{1}{2\pi i}\int_{(c)}\Gamma^2\left(\frac{1+s}{2}\right)\sum_{n=0}^\infty\frac{\left(\frac{1+s}{2}\right)_n\left(\frac{2-s}{2}\right)_n}{n!\left(\frac{3}{2}\right)_n\left(\frac{3}{2}\right)_{2n}} \left(-\frac{w^4}{16}\right)^n{}_1F_1\left(\frac{1+s}{2}+n;\frac{3}{2}+2n;-\frac{w^2}{2}\right)x^{-s}\ ds.
\end{align}
We would now like to interchange the order of summation and integration. That this can be done is now justified. Employ the change of variable $s=c+it$ in the above equation so that
\begin{align}\label{integral20}
I(x,w)&=\frac{x^{-c}}{2\pi}\int_{-\infty}^\infty\Gamma^2\left(\frac{1+c+it}{2}\right)\sum_{n=0}^\infty\frac{\left(\frac{1+c+it}{2}\right)_n\left(\frac{2-c-it}{2}\right)_n}{n!\left(\frac{3}{2}\right)_n\left(\frac{3}{2}\right)_{2n}} \left(-\frac{w^4}{16}\right)^n \nonumber\\
&\qquad\qquad\times{}_1F_1\left(\frac{1+c+it}{2}+n;\frac{3}{2}+2n;-\frac{w^2}{2}\right)x^{-it}dt
\end{align}
Now
\begin{align*}
\left|\left(\frac{1+c+it}{2}\right)_n\right|
&\leq\left|\left(\frac{1+c+it}{2}+n\right)\right|\left|\left(\frac{1+c+it}{2}+n\right)\right|\cdots\left|\left(\frac{1+c+it}{2}+n\right)\right|\nonumber\\
&=\left(\frac{1+|c|+t}{2}+n\right)^n.
\end{align*}
Therefore, as $n\rightarrow\infty$, 
\begin{align}\label{poch1}
\left|\left(\frac{1+c+it}{2}\right)_n\right|=O\left(n^ne^{\frac{1+|c|+|t|}{2}}\right).
\end{align}
Similarly, 
\begin{align}\label{poch2}
\left|\left(\frac{2-c-it}{2}\right)_n\right|=O\left(n^ne^{\frac{2+|c|+|t|}{2}}\right)
\end{align}
as $n\rightarrow\infty$. By using \eqref{strivert0}, it is easy to see that, as $n\rightarrow\infty$,
\begin{align}\label{poch3}
\Gamma(n+1)&\sim\sqrt{2\pi}n^{n+\frac{1}{2}}e^{-n}\nonumber\\
\Gamma\left(\frac{3}{2}+n\right)&\sim\sqrt{2\pi}n^{n+1}e^{-n}\nonumber\\
\Gamma\left(\frac{3}{2}+2n\right)&\sim\sqrt{2\pi}(2n)^{2n+1}e^{-2n}.
\end{align}
From \cite[p. 318]{mos},
\begin{align}\label{masm}
M_{\frac{1}{4}-\frac{s}{2},\frac{1}{4}+n}\left(-\frac{w^2}{2}\right)=\left(-\frac{w^2}{2}\right)^{n+\frac{3}{4}}\left(1+O\left(\frac{1}{n}\right)\right),
\end{align} 
where $M_{\kappa,\mu}(z)$ is the Whittaker function given by \cite[p. 178, Equation (7.17)]{temme}
\begin{align}\label{m1f1}
M_{\kappa, \mu}(z)=e^{-\frac{1}{2}z}z^{\frac{1}{2}+\mu}{}_1F_1\left(\frac{1}{2}+\mu-\kappa;1+2\mu;z\right).
\end{align}
Thus, as $n\to\infty$,
\begin{align*}
{}_1F_1\left(\frac{1+s}{2}+n;\frac{3}{2}+2n;-\frac{w^2}{2}\right)=e^{-\frac{w^2}{4}}\left(1+O\left(\frac{1}{n}\right)\right).
\end{align*}
From \eqref{poch1}, \eqref{poch2}, \eqref{poch3} and \eqref{masm}, we see that series on the right hand side of \eqref{integral20} is uniform convergent in $t$ on compact subsets of $(-\infty,\infty)$. Again from \eqref{integral20}, for a large enough $M>0$,
\begin{align*}
|I(x,w)|&\leq\frac{x^{-c}}{2\pi}\int_{-\infty}^\infty\left|\Gamma^2\left(\frac{1+c+it}{2}\right)\right|\left[\sum_{n=0}^M+\sum_{M+1}^{\infty}\right]\Bigg|\frac{\left(\frac{1+c+it}{2}\right)_n\left(\frac{2-c-it}{2}\right)_n}{n!\left(\frac{3}{2}\right)_n\left(\frac{3}{2}\right)_{2n}} \left(-\frac{w^4}{16}\right)^n \nonumber\\
&\qquad\qquad\times{}_1F_1\left(\frac{1+c+it}{2}+n;\frac{3}{2}+2n;-\frac{w^2}{2}\right)x^{-it}\Bigg|\, dt\\
&=:I_1(x,w)+I_2(x,w).
\end{align*}
An application of \eqref{strivert} implies that $|I_1(x,w)|<\infty$. From \eqref{poch1}, \eqref{poch2}, \eqref{poch3}, \eqref{masm} and \eqref{strivert}, we see that $|I_2(x,w)|$ is also finite, hence $|I(x,w)|<\infty$. Therefore, Theorem \ref{ldcttemme} allows us to interchange the order of integration and summation in \eqref{bassettype2}. Thus from \eqref{bassettype2},
\begin{align}\label{bassettype5}
I(x,w)
=\sum_{n=0}^\infty\frac{\left(-\frac{w^4}{16}\right)^n}{n! \left(\frac{3}{2}\right)_n \left(\frac{3}{2}\right)_{2n}} A_n(x,w),
\end{align}
where
\begin{align*}
A_n(x,w)&:=\frac{1}{2\pi i}\int_{(c)}\frac{\Gamma\left(\frac{1+s}{2}\right) \Gamma\left(\frac{1+s}{2}+n\right)\Gamma\left(\frac{2-s}{2}+n\right)}{\Gamma\left(\frac{2-s}{2}\right)} {}_1F_1\left(\frac{1+s}{2}+n;\frac{3}{2}+2n;-\frac{w^2}{2}\right)x^{-s}\ ds \nonumber\\
&=\frac{1}{2\pi i}\int_{(c)}\frac{\Gamma\left(\frac{1+s}{2}\right) \Gamma\left(\frac{1+s}{2}+n\right)\Gamma\left(\frac{2-s}{2}+n\right)}{\Gamma\left(\frac{2-s}{2}\right)}\sum_{m=0}^\infty \frac{\left(\frac{1+s}{2}+n\right)_m}{m!\left(\frac{3}{2}+2n\right)_{m}}\left(-\frac{w^2}{2}\right)^m x^{-s}\ ds.
\end{align*}
Proceeding along the similar lines as in \eqref{bassettype2} to interchange the order of summation and integration, we see that
\begin{align}\label{bassettype6}
A_n(x,w)
=\Gamma\left(\frac{3}{2}+2n\right)\sum_{m=0}^\infty\frac{(-w^2/2)^m}{m!}B_{n,m}(x),
\end{align}
where
\begin{align*}
B_{n,m}(x):&=\frac{1}{2\pi i}\int_{(c)}\frac{\Gamma\left(\frac{1+s}{2}\right)\Gamma\left(\frac{2-s}{2}+n\right)\Gamma\left(\frac{1+s}{2}+n+m\right)}{\Gamma\left(\frac{2-s}{2}\right)\Gamma\left(\frac{3}{2}+2n+m\right)}x^{-s}\ ds.
\end{align*}
Next, from \cite[p. 42, Formula 5.1]{ober} and the fact $\Gamma\left(\frac{1+s}{2}\right)/\Gamma\left(\frac{2-s}{2}\right)=\pi^{-1/2}2^{1-s}\Gamma(s)\sin\left(\frac{\pi s}{2}\right)$, we see that for $-1<d_1=\textup{Re}(s)<1$,
\begin{align}\label{bassettype8}
\frac{1}{2\pi i}\int_{(d_1)}\frac{\Gamma\left(\frac{1+s}{2}\right)}{\Gamma\left(\frac{2-s}{2}\right)}x^{-s}\ ds &=\frac{2}{\sqrt{\pi}}\sin(2x).
\end{align}
Also, Euler's beta integral for $0<d=\mathrm{Re}(s)<\mathrm{Re}(z)$ is given by
\begin{align*}
\frac{1}{2\pi i}\int_{(d)}\frac{\Gamma(s_1)\Gamma(z-s_1)}{\Gamma(z)}x^{-s_1}\ ds_1=\frac{1}{(1+x)^z}.
\end{align*}
Substitute $s_1=\frac{2-s}{2}+n,\ z=\frac{3}{2}+2n+m$ and replace $x$ by $x^{-2}$ to have for $-1-2n-2m<d_2=\mathrm{Re}(s)<2+2n$,
\begin{align}\label{bassettype9}
\frac{1}{2\pi i}\int_{(d_2)}\frac{\Gamma\left(\frac{2-s}{2}+n\right)\Gamma\left(\frac{1+s}{2}+n+m\right)}{\Gamma\left(\frac{3}{2}+2n+m\right)}x^{-s}\ dt=\frac{2x^{s+2n+2m}}{(1+x^2)^{\frac{3}{2}+2n+m}}.
\end{align}
Hence from \eqref{bassettype8}, \eqref{bassettype9} and \eqref{Persval},
\begin{align}\label{bassettype10}
B_{n,m}(x)=\frac{4}{\sqrt{\pi}}x^{1+2n+2m}\int_0^\infty\frac{t^{2n+1}\sin(2t)}{(x^2+t^2)^{\frac{3}{2}+2n+m}}\ dt.
\end{align}
Thus from \eqref{bassettype6} and \eqref{bassettype10},
\begin{align*}
A_n(x,w)
=\frac{4x^{2n+1}}{\sqrt{\pi}}\Gamma\left(\frac{3}{2}+2n\right)\sum_{m=0}^\infty \frac{(-w^2x^2/2)^m}{m!}\int_0^\infty \frac{t^{2n+1}\sin(2t)}{(x^2+t^2)^{\frac{3}{2}+2n+m}}\ dt.
\end{align*}
Since the series $\displaystyle\sum_{m=0}^\infty \frac{(-w^2x^2/2)^m}{m!(x^2+t^2)^m}$ converges uniformly to $\displaystyle\exp\left(-\frac{w^2x^2}{2(x^2+t^2)}\right)$ on any compact interval of $(0,\infty)$ and $\displaystyle\left|\int_0^\infty \frac{t^{2n+1}\sin(2t)}{(x^2+t^2)^{\frac{3}{2}+2n}}\sum_{m=0}^\infty \frac{(-w^2x^2/2)^m}{m!(x^2+t^2)^m}\ dt\right|$ is finite, invoking Theorem \ref{ldcttemme}, we have
\begin{align}\label{bassettype11}
A_n(x,w)
=\frac{4x^{2n+1}}{\sqrt{\pi}}\Gamma\left(\frac{3}{2}+2n\right)\int_0^\infty \frac{t^{2n+1}\sin(2t)}{(x^2+t^2)^{\frac{3}{2}+2n}}\exp\left(-\frac{w^2x^2}{2(x^2+t^2)}\right)\ dt.
\end{align}
Substituting \eqref{bassettype11} in \eqref{bassettype5}, we deduce that
\begin{align}\label{bassettype12}
I(x,w)
&=2\sum_{n=0}^\infty \frac{x^{2n+1}}{n!\left(\frac{3}{2}\right)_n}\left(-\frac{w^4}{16}\right)^n\int_0^\infty \frac{t^{2n+1}\sin(2t)}{(x^2+t^2)^{\frac{3}{2}+2n}}\exp\left(-\frac{w^2x^2}{2(x^2+t^2)}\right)\ dt \nonumber\\
&=2x\int_0^\infty \exp\left(-\frac{w^2x^2}{2(x^2+t^2)}\right) \frac{t\sin(2t)}{(x^2+t^2)^{3/2}}\sum_{n=0}^\infty \frac{t^{2n}(-w^4x^2/16)^n}{n!\left(\frac{3}{2}\right)_n(x^2+t^2)^{2n}}\ dt\nonumber\\
&=\frac{4}{w^2}\int_0^\infty \exp\left(-\frac{w^2x^2}{2(x^2+t^2)}\right) \sin\left(\frac{w^2xt}{2(x^2+t^2)}\right) \frac{\sin(2t)}{(x^2+t^2)^{1/2}} \ dt,
\end{align}
where, in the second step above, we interchanged the order of summation and integration using Theorem \ref{ldcttemme} since  $\displaystyle\sum_{n=0}^\infty \frac{t^{2n}(-w^4x^2/16)^n}{n!\left(\frac{3}{2}\right)_n(x^2+t^2)^{2n}}$ converges uniformly on any compact interval of $(0,\infty)$ to $\displaystyle\frac{2}{w^2xt}(x^2+t^2)\sin\left(\frac{w^2xt}{2(x^2+t^2)}\right)$ and since $$\left|\int_0^\infty \exp\left(-\frac{w^2x^2}{2(x^2+t^2)}\right) \frac{\sin(2t)}{(x^2+t^2)^{3/2}} \left(2(x^2+t^2)\sin\left(\frac{w^2xt}{2(x^2+t^2)}\right)\right)\ dt\right|<\infty.$$
Finally let $t=u/2$ in the integral on the extreme right of \eqref{bassettype12} and simplify to obtain \eqref{bassset} for $x>0$. Now observe that both sides of \eqref{bassset} are analytic in $|\arg(x)|<\frac{\pi}{4}$, so by analytic continuation, it holds for $|\arg(x)|<\frac{\pi}{4}$.

\end{proof}
We now derive a representation for ${}_1K_{z,w}(x)$ as an infinite series of Laplace transforms of function involving ${}_0F_{2}$ hypergeometric functions.
\begin{theorem}\label{LT}
Let $w\in\mathbb{C},\mathrm{Re}(z)>-\frac{1}{2}$ and $|\arg(x)|<\frac{\pi}{4}$. Then
\begin{align}\label{LTeqn}
{}_1K_{z,w}(x)&=\frac{2^{z+\frac{1}{2}}x^{z+\frac{3}{2}}}{\Gamma(z+\frac{1}{2})}\sum_{n=0}^\infty\frac{\left(-\frac{w^2x}{2}\right)^n}{(2n+1)!}\int_0^\infty(t^2+t)^{z-\frac{1}{2}}(2t+1)^{-n+\frac{1}{2}}K_{n+\frac{1}{2}}(x(2t+1)) \nonumber\\
& \qquad\qquad \times {}_0F_2\left(-;\frac{1}{2}+z,\frac{3}{2};-\frac{w^2x^2t(t+1)}{4}\right)\, dt.
\end{align}
\end{theorem}
\begin{proof}
Replace $x$ by $\frac{x}{2}$ in Theorem \ref{integralRepr}, let $u=\left(\frac{xt}{2}\right)^{1/2}$ in the resulting equation, then replace $z$ by $-z$ and use the fact ${}_1K_{-z,w}(x)={}_1K_{z,w}(x)$ to obtain 
\begin{align}\label{lt1}
{}_1K_{z,w}(x)&=\frac{1}{w^2}\int_0^\infty \exp\left(-\frac{x}{2}\left(u+\frac{1}{u}\right)\right)\sin\left(\frac{w\sqrt{xu}}{\sqrt{2}}\right)\sin\left(\frac{w\sqrt{x}}{\sqrt{2u}}\right)u^{-z-1}\, du.
\end{align}
From \cite[p. 186, Formula 4.26]{ober}\footnote{There is typo in the argument of ${}_0F_2$. It should be $-\frac{a^2y}{4}$ instead of $-\frac{a^2y}{2}$. Also, $y^{z-1/2}$ is missing from the integrand.}, for $\mathrm{Re}(u)>0,\ \mathrm{Re}(z)>-\frac{1}{2}$,
\begin{align*}
u^{-z}\sin(au^{-1/2})=\frac{a}{\Gamma\left(\frac{1}{2}+z\right)}\int_0^\infty e^{-yu}y^{z-\frac{1}{2}}{}_0F_2\left(-;\frac{1}{2}+z,\frac{3}{2};-\frac{a^2y}{4}\right)\, dy.
\end{align*}
Replacing $a$ by $w\sqrt{x/2}$ in the above equation, we have
\begin{align}\label{lt2}
u^{-z}\sin\left(w\sqrt{\frac{x}{2u}}\right)=\frac{w\sqrt{x}}{\sqrt{2}\Gamma\left(\frac{1}{2}+z\right)}\int_0^\infty e^{-yu} y^{z-\frac{1}{2}}{}_0F_2\left(-;\frac{1}{2}+z,\frac{3}{2};-\frac{w^2xy}{8}\right)\, dy.
\end{align}
Substituting \eqref{lt2} in \eqref{lt1} and then interchanging the order of integration using Fubini's theorem, we deduce that
\begin{align}\label{lt4}
{}_1K_{z,w}(x)&=\frac{\sqrt{x}}{\sqrt{2}w\Gamma\left(z+\frac{1}{2}\right)}\int_0^\infty y^{z-\frac{1}{2}} {}_0F_2\left(-;\frac{1}{2}+z,\frac{3}{2};-\frac{w^2xy}{8}\right)\, dy\nonumber\\
&\quad\times \int_0^\infty  \exp\left(-u\left(y+\frac{x}{2}\right)-\frac{x}{2u}\right)\sin\left(\frac{w\sqrt{xu}}{\sqrt{2}}\right) \frac{\, du}{u}.
\end{align}
Now using the same logic as in the proof of Lemma \ref{bessetTypeLemma}, one can derive
\begin{align}\label{lt6}
&\int_0^\infty   \exp\left(-u\left(y+\frac{x}{2}\right)-\frac{x}{2u}\right)\sin\left(\frac{w\sqrt{xu}}{\sqrt{2}}\right) \frac{\, du}{u}\nonumber \\
&=\frac{2w(x/2)^{3/4}}{\left(y+\frac{x}{2}\right)^{1/4}}\sum_{n=0}^{\infty}\frac{1}{(2n+1)!}\left(-\frac{w^2x^{3/2}}{2^{3/2}\left(y+\frac{x}{2}\right)^{1/2}}\right)^n K_{n+\frac{1}{2}}\left(\sqrt{2x\left(y+\frac{x}{2}\right)}\right).
\end{align}
Substituting \eqref{lt6} in \eqref{lt4} and then interchanging the order of summation and integration, we have
\begin{align*}
{}_1K_{z,w}(x)&=\frac{2(x/2)^{5/4}}{\Gamma\left(z+\frac{1}{2}\right)}\int_0^\infty y^{z-\frac{1}{2}} {}_0F_2\left(-;\frac{1}{2}+z,\frac{3}{2};-\frac{w^2xy}{8}\right)  \sum_{n=0}^{\infty}\frac{\left(-\frac{w^2x^{3/2}}{2^{3/2}}\right)^n}{(2n+1)!}\nonumber\\ &\qquad\qquad\qquad \times \left(y+\frac{x}{2}\right)^{-\frac{n}{2}-\frac{1}{4}} K_{n+\frac{1}{2}}\left(\sqrt{2x\left(y+\frac{x}{2}\right)}\right)\, dy,
\end{align*}
which upon employing change of variable $2t+1=\frac{\sqrt{y+\frac{x}{2}}}{\sqrt{x/2}}$ and simplifying gives \eqref{LTeqn}.
\end{proof}
\begin{remark}
The integral in Theorem \textup{\ref{LT}} is indeed a Laplace transform since \cite[p.~934, Formula \textbf{8.468}]{grn}
\begin{equation*}
K_{n+\frac{1}{2}}(y)=\sqrt{\frac{\pi}{2y}}e^{-y}\sum_{k=0}^{n}\frac{(n+k)!}{k!(n-k)!(2y)^k}.
\end{equation*}
Since only the $n=0$ term survives when we let $w\to0$ in Theorem \textup{\ref{LT}}, we recover the well-known result \cite[Equation (1.29)]{dkmt}
\begin{equation*}
K_{z}(x)=\frac{\sqrt{\pi}(2x)^{z}e^{-x}}{\G\left(z+\frac{1}{2}\right)}\int_{0}^{\infty}e^{-2xt}t^{z-\frac{1}{2}}(t+1)^{z-\frac{1}{2}}\, dt
\end{equation*}
for $\textup{Re}(z)>-\frac{1}{2}$ and $|\arg(x)|<\frac{\pi}{4}$.
\end{remark}
A double integral representation for ${}_1K_{z,w}(x)$ is now given.
\begin{theorem}\label{doubleIntRepr}
Let $w\in\mathbb{C}$. For \textup{Re}$(z)>-1$ and $|\arg(x)|<\frac{\pi}{2}$,
\begin{align}\label{doubleIntRepreqn}
{}_1K_{z,w}(x)&=\frac{x}{2\Gamma(z+1)}\int_0^\infty \int_0^\infty \frac{y^z t^{-\frac{1}{2}}}{\sqrt{y+\frac{x}{2}}}\exp\left(-2\sqrt{\left(y+\frac{x}{2}\right)\left(t+\frac{x}{2}\right)}\right) \nonumber \\
&\hspace{3cm} \times {}_0F_2\left(-;1+z,\frac{3}{2};-\frac{w^2xy}{8}\right) {}_0F_2\left(-;\frac{1}{2},\frac{3}{2};-\frac{w^2xt}{8}\right)\, dt\, dy.
\end{align}
\end{theorem}
\begin{proof}
Replace $z$ by $z+1/2$ in \eqref{lt2}, substitute the resulting equation in \eqref{lt1} and then interchange the order of integration so as to get for Re$(z)>-1$,
\begin{align}\label{dir2}
{}_1K_{z,w}(x)&=\frac{\sqrt{x}}{w\sqrt{2}\Gamma(z+1)}\int_0^\infty y^z {}_0F_2\left(-;1+z,\frac{3}{2};-\frac{w^2xy}{8}\right) \int_0^\infty \exp\left(u\left(y+\frac{x}{2}\right)-\frac{x}{2u}\right) \nonumber\\
&\hspace{3cm}\times \sin\left(\frac{w\sqrt{xu}}{\sqrt{2}}\right) \frac{\, du}{\sqrt{u}}\, dy.
\end{align}
Next, let $z=0$ in \eqref{lt2} and replace $u$ by $1/u$ in \eqref{lt2} so that
\begin{align}\label{dir3}
\sin\left(\frac{w\sqrt{xu}}{\sqrt{2}}\right)&=\int_0^\infty \frac{w\sqrt{x}t^{-\frac{1}{2}}}{\sqrt{2}\Gamma\left(\frac{1}{2}\right)}{}_0F_2\left(-;\frac{1}{2},\frac{3}{2};-\frac{w^2xt}{8}\right)e^{-t/u}\, dt.
\end{align}
Substitute \eqref{dir3} in \eqref{dir2} and then interchange the order of integration to arrive at
\begin{align}\label{dir4}
{}_1K_{z,w}(x)&=\frac{x}{2\sqrt{\pi}\Gamma(z+1)}\int_0^\infty \int_0^\infty y^z t^{-\frac{1}{2}} {}_0F_2\left(-;1+z,\frac{3}{2};-\frac{w^2xy}{8}\right) {}_0F_2\left(-;\frac{1}{2},\frac{3}{2};-\frac{w^2xt}{8}\right) \nonumber \\
&\qquad\qquad\qquad\times \int_0^\infty \exp\left(-u\left(y+\frac{x}{2}\right)-\left(t+\frac{x}{2}\right)\frac{1}{u}\right)\frac{\, du}{\sqrt{u}}\, dt \, dy.
\end{align}
Now let $s=\frac{1}{2},p=y+\frac{x}{2}$ and $q=t+\frac{x}{2}$ in \eqref{PrudnikovFormula} and then use $K_{\frac{1}{2}}(x)=\sqrt{\frac{\pi}{2x}}e^{-x}$ so that for Re$(x)>0$,
\begin{align}\label{dir5}
 \int_0^\infty \exp\left(-u\left(y+\frac{x}{2}\right)-\left(t+\frac{x}{2}\right)\frac{1}{u}\right)\frac{\, du}{\sqrt{u}}=\frac{\sqrt{\pi}}{\left(y+\frac{x}{2}\right)^{1/2}}e^{-2\sqrt{\left(y+\frac{x}{2}\right)\left(t+\frac{x}{2}\right)}}.
\end{align}
Finally, \eqref{doubleIntRepreqn} follows by substituting \eqref{dir5} in \eqref{dir4}.
\end{proof}
\begin{remark}
Letting $w=0$ in Theorem \textup{\ref{doubleIntRepr}} gives us an integral representation for $K_z(x)$ that was recently obtained in \cite[p.~393]{dkmt}.
\end{remark}

\subsection{An explicit pair of functions reciprocal in the second Koshliakov kernel}\label{expair}
\hfill\\

Here we obtain a generalization of \eqref{koshlyakov-2} which is the crux of the machinery needed to prove Theorem \ref{xiintgenramhurthm}.
\begin{theorem}\label{reciFn}
Let $-\frac{3}{4}<\textup{Re}(z)<\frac{3}{4}$ and let $w\in\mathbb{C}$. Let $L_{\nu}(x)$ be defined in \eqref{kernels}. Let $\alpha$ and $\beta$ be two positive numbers such that $\alpha\beta=1$. The functions $e^{-w^2/2}{}_1K_{z,iw}(2\alpha x)$ and $\beta\ {}_1K_{z,w}(2\beta x)$ form a pair of reciprocal functions with respect to the second Koshliakov kernel, that is,
\begin{align}
e^{-w^2/2}{}_1K_{z,iw}(2\alpha x)=2 \int_0^\infty \beta\  {}_1K_{z,w}(2\beta t)\left(\sin(\pi z)J_{2z}(4\sqrt{xt})-\cos(\pi z)L_{2z}(4\sqrt{xt})\right)\, dt,\label{1st}\\
\beta\ {}_1K_{z,w}(2\beta x)=2 \int_0^\infty e^{-w^2/2} {}_1K_{z,iw}(2\alpha t)\left(\sin(\pi z)J_{2z}(4\sqrt{xt})-\cos(\pi z)L_{2z}(4\sqrt{xt})\right)\, dt.\label{2nd}
\end{align}
\end{theorem}
\begin{proof}
Using Theorems \ref{asymptotcs large x}, \ref{asymSmall} and the bound
\begin{equation*}
\left\lvert{\sin(\pi z)J_{2z}(4\sqrt{xt})-\cos(\pi z)L_{2z}(4\sqrt{xt})}\right\rvert \ll_z 
			\begin{cases}
				1+|\log (tx)|, &\mbox{ if } \quad z=0, 0 \leq tx \leq 1, \\
				(tx)^{-|\textup{Re}(z)|}, &\mbox{ if } \quad z\neq 0,  0 \leq tx \leq 1, \\
				(tx)^{-1/4}, &\mbox{ if } \quad tx \geq 1,
				\end{cases}	
\end{equation*}
we see that the integrals in Theorem \ref{reciFn} converge for $-\frac{3}{4}<\textup{Re}(z)<\frac{3}{4}$. Now replacing $x, z$ and $y$ in \cite[Lemma 5.2]{dixitmoll} by $t/\pi^2, 2z$ and $x$ respectively, we see that for $\pm\textup{Re}(z)<\textup{Re}(s)<\frac{3}{4}$ and $x>0$,
\begin{align}\label{DixitEqn}
\int_0^\infty t^{s-1}\left(\sin(\pi z)J_{2z}(4\sqrt{xt})-\cos(\pi z)L_{2z}(4\sqrt{xt})\right)dt
=\frac{\Gamma\left(s-z\right)\Gamma\left(s+z\right)\left(\cos(\pi z)-\cos(\pi s)\right)}{2^{2s}\pi x^s}.
\end{align}
From \eqref{def}, for $\textup{Re}(s)>-1\pm$Re$(z)$,
 \begin{align}\label{MellTransGenBesFn}
 \int_0^\infty t^{s-1}\beta\ {}_1K_{z,w}(2\beta t)dt&=\frac{\beta^{1-s}}{2}\Gamma\left(\frac{1+s-z}{2}\right)\Gamma\left(\frac{1+s+z}{2}\right){}_1F_1\left(\frac{1+s-z}{2};\frac{3}{2};-\frac{w^2}{4}\right)\nonumber\\
&\quad\times{}_1F_1\left(\frac{1+s+z}{2};\frac{3}{2};-\frac{w^2}{4}\right).
 \end{align}
Using \eqref{DixitEqn}, \eqref{MellTransGenBesFn} and Parseval's formula \eqref{par}, and then employing \eqref{kft}, we see that for $\pm\textup{Re}(z)<c=\textup{Re}(s)<\min\left(\frac{3}{4},2\pm\textup{Re}(z)\right)$,
\begin{align*}
& \int_0^\infty\beta\ {}_1K_{z,w}(2\beta t)\left(\sin(\pi z)J_{2z}(4\sqrt{xt})-\cos(\pi z)L_{2z}(4\sqrt{xt})\right)dt\nonumber \\
& =\frac{e^{-w^2/2}}{2\pi i}\int_{(c)}\frac{\beta^s}{2}\Gamma\left(\frac{2-s-z}{2}\right)\Gamma\left(\frac{2-s+z}{2}\right)\frac{\Gamma\left(s-z\right)\Gamma\left(s+z\right)}{2^{2s}\pi x^s}\left(\cos(\pi z)-\cos(\pi s)\right)\nonumber\\ 
 &\qquad\qquad\quad\times{}_1F_1\left(\frac{1+s+z}{2};\frac{3}{2};\frac{w^2}{4}\right){}_1F_1\left(\frac{1+s-z}{2};\frac{3}{2};\frac{w^2}{4}\right)\, ds
\end{align*}
Note that we need $-3/4<\textup{Re}(z)<3/4$ since both $\pm\textup{Re}(z)<3/4$ and $\pm\textup{Re}(z)<2\pm\textup{Re}(z)$ have to be satisfied. Now simplify the integrand on the above right-hand side using the duplication and reflection formulas for the gamma function and then use the fact $\a\b=1$ to arrive at
\begin{align*}
& \int_0^\infty\beta\ {}_1K_{z,w}(2\beta t)\left(\sin(\pi z)J_{2z}(4\sqrt{xt})-\cos(\pi z)L_{2z}(4\sqrt{xt})\right)dt\nonumber \\
&=\frac{e^{-w^2/2}}{4\pi i}\int_{(c)}\Gamma\left(\frac{1+s-z}{2}\right)\Gamma\left(\frac{1+s+z}{2}\right){}_1F_1\left(\frac{1+s+z}{2};\frac{3}{2};\frac{w^2}{4}\right){}_1F_1\left(\frac{1+s-z}{2};\frac{3}{2};\frac{w^2}{4}\right)  \nonumber \\ 
 &\hspace{2cm}\times (2\alpha x)^{-s}2^{s-1}ds\nonumber\\
&=\frac{1}{2}e^{-w^2/2}{}_1K_{z,iw}(2\alpha x),
\end{align*}
as can be seen from \eqref{def}. This proves \eqref{1st}. Similarly one can prove \eqref{2nd}.
\end{proof}
\begin{remark}
Letting $\a=\b=1$ and then $w=0$ in either \eqref{1st} or \eqref{2nd} leads to Koshliakov's result \eqref{koshlyakov-2}.
\end{remark}

\subsection{An integral evaluation involving ${}_1K_{z,w}(x)$ and a hypergeometric function}\label{1kzwhyper}
\hfill\\

We now transform an integral involving ${}_1K_{z,w}(x)$ into a double integral. This is one of the main ingredients in the proof of Theorem \ref{xiintgenramhurthm}.
\begin{theorem}\label{integr}
Let $0<\textup{Re}(z)<1$. Let $w\in\mathbb{C}, n\in\mathbb{N}$ and $\a>0$. Then
\begin{align}\label{integreqn}
&\int_0^\infty {}_1K_{\frac{z}{2},iw}(2\alpha x)\left({}_2F_1\left(1,\frac{z}{2};\frac{1}{2};-\frac{x^2}{\pi^2n^2}\right)-1\right)x^{\frac{z-2}{2}}\, dx\nonumber\\
&=\frac{-4n^\frac{z}{2}\pi^{\frac{z+1}{2}}e^{\frac{w^2}{4}}}{w^2\Gamma\left(\frac{z}{2}\right)}\int_0^\infty\int_0^\infty u^{-1+\frac{z}{2}}v^ze^{-v^2(u^2+1)}K_{\frac{z}{2}}\left(\frac{2n\pi \alpha}{u}\right)\sin(wv)\sinh(wuv)\, dudv.
\end{align}
\end{theorem}
\begin{proof}
From \eqref{def}, for Re$(s)>-1\pm$Re$\left(\frac{z}{2}\right)$,
\begin{align*}
&\int_0^\infty x^{s-1}{}_1K_{\frac{z}{2},iw}(2\alpha x)dx \nonumber\\
&=\frac{1}{2\alpha^{s}}\Gamma\left(\frac{1+s}{2}-\frac{z}{4}\right)\Gamma\left(\frac{1+s}{2}+\frac{z}{4}\right){}_1F_1\left(\frac{1+s}{2}-\frac{z}{4};\frac{3}{2};\frac{w^2}{4}\right){}_1F_1\left(\frac{1+s}{2}+\frac{z}{4};\frac{3}{2};\frac{w^2}{4}\right).
\end{align*}
Replace $s$ by $s+\frac{z}{2}-1$ in the above equation to get for Re$(s)>0$,
\begin{align}\label{mellin transform of 1Kzw}
&\int_0^\infty x^{s+\frac{z}{2}-2}{}_1K_{\frac{z}{2},iw}(2\alpha x)dx \nonumber\\
&=\frac{1}{2\alpha^{s+z/2-1}}\Gamma\left(\frac{s}{2}\right)\Gamma\left(\frac{s+z}{2}\right){}_1F_1\left(\frac{s}{2};\frac{3}{2};\frac{w^2}{4}\right){}_1F_1\left(\frac{s+z}{2};\frac{3}{2};\frac{w^2}{4}\right).
\end{align}
From Lemma \ref{IMT of 2F1}, \eqref{mellin transform of 1Kzw} and Parseval's formula in the form \eqref{par}, we have for $0<c=$ Re$(s)<$ Re$(z)$,
\begin{align}\label{Int1K2F1}
&\int_0^\infty {}_1K_{\frac{z}{2},iw}(2\alpha x){}_2F_1\left(1,\frac{z}{2};\frac{1}{2};-\frac{x^2}{\pi^2n^2}\right)x^{\frac{z-2}{2}}dx\nonumber\\
&=\frac{1}{2\pi i}\int_{(c)}\frac{1}{2\alpha^{z/2-s}}\Gamma\left(\frac{1-s}{2}\right)\Gamma\left(\frac{1-s+z}{2}\right){}_1F_1\left(\frac{1-s}{2};\frac{3}{2};\frac{w^2}{4}\right){}_1F_1\left(\frac{1-s+z}{2};\frac{3}{2};\frac{w^2}{4}\right)\nonumber\\
&\qquad\qquad\times \frac{n^s\pi^{\frac{3}{2}+s}}{2\sin\left(\frac{\pi s}{2}\right)}\frac{\Gamma\left(\frac{z-s}{2}\right)}{\Gamma\left(\frac{1-s}{2}\right)\Gamma\left(\frac{z}{2}\right)}\, ds \nonumber\\
&=\frac{\sqrt{\pi}}{8i\alpha^{z/2}\Gamma\left(\frac{z}{2}\right)}\int_{(c)}\frac{\Gamma\left(\frac{1-s+z}{2}\right)\Gamma\left(\frac{z-s}{2}\right)}{\sin\left(\frac{\pi s}{2}\right)}{}_1F_1\left(\frac{1-s}{2};\frac{3}{2};\frac{w^2}{4}\right){}_1F_1\left(\frac{1-s+z}{2};\frac{3}{2};\frac{w^2}{4}\right)(\pi n \alpha)^sds\nonumber\\
\end{align}
Replace $x$ by $2n\pi\alpha/x$, $s$ by $-s+\frac{z}{2}$ and $z$ by $z/2$ in \eqref{ober1} to obtain for $c_1=\mathrm{Re}(s)<0$,
\begin{align}\label{mellin transform of gamma functions}
\frac{1}{2\pi i}\int_{(c_1)}\Gamma\left(\frac{z-s}{2}\right)\Gamma\left(\frac{-s}{2}\right)(n\pi\alpha)^sx^{-s}ds=4(n\pi \alpha)^{\frac{z}{2}}x^{-\frac{z}{2}}K_{\frac{z}{2}}\left(\frac{2n\pi\alpha}{x}\right).
\end{align}
Replace $s$ by $s+1$ in \eqref{mellin transform of gamma 1F1 1} and let $a=1,\ b=w$ to see that, for $c_2=\textup{Re}(s)>-2$,
{\allowdisplaybreaks\begin{align}\label{MT1F1 2}
\frac{1}{2\pi i}\int_{(c_2)}\Gamma\left(1+\frac{s}{2}\right){}_1F_1\left(\frac{1-s}{2};\frac{3}{2};\frac{w^2}{4}\right)x^{-s}ds=\frac{2x}{w}e^{\frac{w^2}{4}}e^{-x^2}\sin(wx).
\end{align}}
Apply \eqref{kft} in \eqref{mellin transform of gamma 1F1 1}, then replace $s$ by $z-s$, let $a=1, b=iw$ and then replace $x$ by $1/x$ so that for $c_3=\textup{Re}(s)<1+\textup{Re}(z)$,
\begin{align}\label{MT1F1 3}
\frac{1}{2\pi i}\int_{(c_3)}\Gamma\left(\frac{1-s+z}{2}\right){}_1F_1\left(\frac{1-s+z}{2};\frac{3}{2};\frac{w^2}{4}\right)x^{-s}ds=\frac{2x^{-z}}{w}e^{-1/x^{2}}\sinh\left(\frac{w}{x}\right).
\end{align}
From \cite[p.~91]{hardymellin} (see also \cite[p. 121, Exercise (36)]{aar}), we have
\begin{align}\label{MT3functions}
\frac{1}{2\pi i}\int_{(\l)}F_1(s)F_2(s)F_3(s)ds=\int_0^\infty\int_0^\infty f_1(u)f_2(v)f_3\left(\frac{1}{uv}\right)\, \frac{dudv}{uv},
\end{align}
where the conditions for its validity are given in \cite{hardymellin}. One can check that the conditions for \eqref{MT3functions} to hold are satisfied by the functions on the right-hand sides of \eqref{mellin transform of gamma functions}, \eqref{MT1F1 2} and \eqref{MT1F1 3} as long as $-2<\l=\textup{Re}(s)<0$. Hence upon using \eqref{refl} in the first step below, we see that for $-2<\l=\textup{Re}(s)<0$,
\begin{align}\label{MT}
&\frac{1}{2\pi i}\int_{(\l)}\frac{\Gamma\left(\frac{1-s+z}{2}\right)\Gamma\left(\frac{z-s}{2}\right)}{\sin\left(\frac{\pi s}{2}\right)}{}_1F_1\left(\frac{1-s}{2};\frac{3}{2};\frac{w^2}{4}\right){}_1F_1\left(\frac{1-s+z}{2};\frac{3}{2};\frac{w^2}{4}\right)(n\pi\alpha)^sds \nonumber\\
&=-\frac{1}{2\pi^2 i}\int_{(\l)}\Gamma\left(\frac{z-s}{2}\right)\Gamma\left(-\frac{s}{2}\right)\Gamma\left(1+\frac{s}{2}\right){}_1F_1\left(\frac{1-s}{2};\frac{3}{2};\frac{w^2}{4}\right)\nonumber\\
&\qquad\qquad\qquad\times\Gamma\left(\frac{1-s+z}{2}\right){}_1F_1\left(\frac{1-s+z}{2};\frac{3}{2};\frac{w^2}{4}\right)(n\pi\alpha)^sds \nonumber\\
&=-\frac{1}{\pi}\int_0^\infty\int_0^\infty 4(n\pi\alpha)^{z/2}u^{-z/2}K_{\frac{z}{2}}\left(\frac{2n\pi \alpha}{u}\right)\frac{2v}{w}e^{w^2/4}e^{-v^2}\sin(wv)\frac{2(uv)^z}{w}e^{-u^2v^2}\sinh(wuv)\frac{dudv}{uv}\nonumber\\
&=\frac{-16(n\pi\alpha)^{\frac{z}{2}}}{\pi w^2}e^{w^2/4}\int_0^\infty\int_0^\infty u^{-1+\frac{z}{2}}v^ze^{-v^2(u^2+1)}K_{\frac{z}{2}}\left(\frac{2n\pi \alpha}{u}\right)\sin(wv)\sinh(wuv)\, dudv.
\end{align}
Now consider the contour formed by the line segments $[\l-iT, c-iT],\ [c-iT, c+iT],\ [c+iT,\l+iT]$, and $[\l+iT,\l-iT]$, where $0<c<\textup{Re}(z)$. Then the only pole of the integrand of the left-hand side of \eqref{MT} inside the contour is the simple pole at $s=0$ (due to $\sin\left(\frac{\pi s}{2}\right)$). Since the integrals along the horizontal segments tend to zero as $T\to\infty$ because of \eqref{strivert}, we find from \eqref{MT} that for $0<c=\textup{Re}(s)<\textup{Re}(z)$,
\begin{align}\label{cauchy35}
&\frac{1}{2\pi i}\int_{(c)}\frac{\Gamma\left(\frac{1-s+z}{2}\right)\Gamma\left(\frac{z-s}{2}\right)}{\sin\left(\frac{\pi s}{2}\right)}{}_1F_1\left(\frac{1-s}{2};\frac{3}{2};\frac{w^2}{4}\right){}_1F_1\left(\frac{1-s+z}{2};\frac{3}{2};\frac{w^2}{4}\right)(n\pi\alpha)^sds \nonumber\\
&=\frac{-16(n\pi\alpha)^{\frac{z}{2}}}{\pi w^2}e^{w^2/4}\int_0^\infty\int_0^\infty u^{-1+\frac{z}{2}}v^ze^{-v^2(u^2+1)}K_{\frac{z}{2}}\left(\frac{2n\pi \alpha}{u}\right)\sin(wv)\sinh(wuv)dudv+R_0,
\end{align}
where $R_0=\displaystyle\frac{2^{2-z}}{w}\Gamma(z)\mathrm{erfi}\left(\frac{w}{2}\right){}_1F_1\left(\frac{1+z}{2};\frac{3}{2};\frac{w^2}{4}\right)$.
From \eqref{Int1K2F1} and \eqref{cauchy35}, we find that
\begin{align}\label{1Kzw2F1}
&\int_0^\infty {}_1K_{\frac{z}{2},iw}(2\alpha x){}_2F_1\left(1,\frac{z}{2};\frac{1}{2};-\frac{x^2}{\pi^2n^2}\right)x^{\frac{z-2}{2}}dx\nonumber\\
&=\frac{-4n^\frac{z}{2}\pi^{\frac{z+1}{2}}e^{\frac{w^2}{4}}}{w^2\Gamma\left(\frac{z}{2}\right)}\int_0^\infty\int_0^\infty u^{-1+\frac{z}{2}}v^ze^{-v^2(u^2+1)}K_{\frac{z}{2}}\left(\frac{2n\pi \alpha}{u}\right)\sin(wv)\sinh(wuv)dudv \nonumber\\
&\quad+\frac{\pi\alpha^{-\frac{z}{2}}}{2w}\Gamma\left(\frac{1+z}{2}\right)\mathrm{erfi}\left(\frac{w}{2}\right){}_1F_1\left(\frac{1+z}{2};\frac{3}{2};\frac{w^2}{4}\right).
\end{align}
Now from \eqref{def},
\begin{align}\label{1Kzw}
\int_0^\infty {}_1K_{\frac{z}{2},iw}(2\alpha x)x^{\frac{z-2}{2}}\, dx
=\frac{\pi\alpha^{-\frac{z}{2}}}{2w}\Gamma\left(\frac{1+z}{2}\right)\mathrm{erfi}\left(\frac{w}{2}\right){}_1F_1\left(\frac{1+z}{2};\frac{3}{2};\frac{w^2}{4}\right).
\end{align}
From \eqref{1Kzw2F1} and \eqref{1Kzw}, we are led to \eqref{integreqn}.
\end{proof}
The Lommel functions $s_{\mu, \nu}(\xi)$ and $S_{\mu, \nu}(\xi)$ are defined by \cite[p.~346, Equation (10)]{watson-1966a}
\begin{align*}
s_{\mu,\nu}(\xi)=\frac{\xi^{\mu+1}}{(\mu-\nu+1)(\mu+\nu+1)}{}_1F_2\left(1;\frac{1}{2}\mu-\frac{1}{2}\nu+\frac{3}{2},\frac{1}{2}\mu+\frac{1}{2}\nu+\frac{3}{2};-\frac{1}{4}\xi^2\right).
\end{align*}
and \cite[p.~347, Equation (2)]{watson-1966a}
\begin{align}\label{lommeldef1cap}
S_{\mu,\nu}(\xi)&=s_{\mu,\nu}(\xi)+\frac{2^{\mu-1}\G\left(\frac{\mu-\nu+1}{2}\right)
\G\left(\frac{\mu+\nu+1}{2}\right)}{\sin(\nu\pi)}\nonumber\\
&\quad\quad\quad\quad\quad\times\left\{\cos\left(\frac{1}{2}(\mu-\nu)\pi\right)
J_{-\nu}(\xi)-\cos\left(\frac{1}{2}(\mu+\nu)\pi\right)J_{\nu}(\xi)\right\}
\end{align}
for $\nu\notin\mathbb{Z}$, and
\begin{align*}
S_{\mu,\nu}(\xi)&=s_{\mu,\nu}(\xi)+2^{\mu-1}\G\left(\frac{\mu-\nu+1}{2}\right)\G\left(\frac{\mu+\nu+1}{2}\right)\nonumber\\
&\quad\quad\quad\quad\quad\times\left\{\sin\left(\frac{1}{2}(\mu-\nu)\pi\right)
J_{\nu}(\xi)-\cos\left(\frac{1}{2}(\mu-\nu)\pi\right)Y_{\nu}(\xi)\right\}
\end{align*}
for $\nu\in\mathbb{Z}$. In his PhD thesis, Yu \cite[Appendix C]{yu} compiled a list of $75$ integral representations for $s_{\mu, \nu}(\xi)$ and $S_{\mu, \nu}(\xi)$. However, to the best of our knowledge, the representation for $S_{-z-\frac{3}{2}, \frac{1}{2}}(\xi)$ given below seems to be new. Note that this special case of the first Lommel function arises in the seminal work of Lewis \cite{lewisinvent}, and of Lewis and Zagier \cite{lewzag}.
\begin{corollary}\label{integral21cor}
For $\textup{Re}(z)>-3$, $n\in\mathbb{N}$ and $\a>0$,
\begin{align}\label{integral21}
&\int_0^\infty 2\alpha x^{\frac{z}{2}}K_{\frac{z}{2}}(2\alpha x)\left({}_2F_1\left(1,\frac{z}{2};\frac{1}{2};-\frac{x^2}{\pi^2n^2}\right)-1\right)\, dx\nonumber\\
&=-\sqrt{2}\pi^{z+\frac{3}{2}}n^{z+\frac{1}{2}}\a^{\frac{z+1}{2}}\frac{\G(z+2)}{\G\left(\frac{z}{2}\right)}S_{-z-\frac{3}{2}, \frac{1}{2}}(2\pi n\a).
\end{align}
\end{corollary}
\begin{proof}
We first obtain the result for $0<\mathrm{Re}(z)<1$ as a special case of Theorem \ref{integr} and later extend it to Re$(z)>-3$ by analytic continuation.

Let $w=0$ in \eqref{integreqn} and use \eqref{Basicform} so that
\begin{align}\label{integral21w0}
&\int_0^\infty 2\alpha x^{\frac{z}{2}}K_{\frac{z}{2}}(2\alpha x)\left({}_2F_1\left(1,\frac{z}{2};\frac{1}{2};-\frac{x^2}{\pi^2n^2}\right)-1\right)\, dx\nonumber\\
&=\frac{-4n^\frac{z}{2}\pi^{\frac{z+1}{2}}}{\Gamma\left(\frac{z}{2}\right)}\int_0^\infty\int_0^\infty u^{\frac{z}{2}}v^{z+2}e^{-v^2(u^2+1)}K_{\frac{z}{2}}\left(\frac{2n\pi \alpha}{u}\right)\, dudv.
\end{align}
From \cite[p. 353, Equation 2.16.8.13]{prudII}, for $\mathrm{Re}(a)>0$ and $\mathrm{Re}(b)>0$,
{\allowdisplaybreaks\begin{align*}
&\int_0^\infty x^{s-1}e^{-\frac{a}{x^2}}K_\nu(bx)\ dx\nonumber\\
&=\frac{2^{s-2}}{b^s}\Gamma\left(\frac{s-\nu}{2}\right)\Gamma\left(\frac{s+\nu}{2}\right){}_0F_2\left(-;\frac{2-s-\nu}{2},\frac{2-s+\nu}{2};-\frac{ab^2}{4}\right)\nonumber\\
&\quad+\frac{a^{\frac{s+\nu}{2}}b^\nu}{2^{\nu+2}}\Gamma(-\nu)\Gamma\left(-\frac{s+\nu}{2}\right){}_0F_2\left(-;1+\nu,\frac{s+\nu+2}{2};-\frac{ab^2}{4}\right)\nonumber\\
&\quad+\frac{a^{\frac{s-\nu}{2}}}{2^{2-\nu}b^\nu}\Gamma(\nu)\Gamma\left(\frac{\nu-s}{2}\right){}_0F_2\left(-;1-\nu,\frac{s-\nu+2}{2};-\frac{ab^2}{4}\right).
\end{align*}}
Let $s=-1-\frac{z}{2},\ \nu=\frac{z}{2}$ and $a=v^2, v>0$, and $b=2n\pi\alpha, n>0,$ in the above equation so that for $\mathrm{Re}(\alpha)>0$,
\begin{align}\label{mt of kz exp}
&\int_0^\infty x^{-\frac{z}{2}-2}e^{-\frac{v^2}{x^2}}K_{\frac{z}{2}}(2n\pi\alpha x)\ dx\nonumber\\
&=\frac{(n\pi\alpha)^{1+\frac{z}{2}}}{4}\Gamma\left(-\frac{z}{2}-\frac{1}{2}\right)\Gamma\left(-\frac{1}{2}\right){}_0F_2\left(-;\frac{3}{2},\frac{3+z}{2};-(n\pi\alpha v)^2\right)\nonumber\\
&\quad+\frac{(n\pi\alpha)^{\frac{z}{2}}}{4v}\Gamma\left(-\frac{z}{2}\right)\Gamma\left(\frac{1}{2}\right){}_0F_2\left(-;1+\frac{z}{2},\frac{1}{2};-(n\pi\alpha v)^2\right)\nonumber\\
&\quad+\frac{v^{-z-1}}{4(n\pi\alpha)^{\frac{z}{2}}}\Gamma\left(\frac{z}{2}\right)\Gamma\left(\frac{z+1}{2}\right){}_0F_2\left(-;1-\frac{z}{2},\frac{1-z}{2};-(n\pi\alpha v)^2\right).
\end{align}
Make the change of variable $x=\frac{1}{u}$ in the integral on the left-hand side of \eqref{mt of kz exp} then in the resultant multiply both sides by $v^{z+2}e^{-v^2}$ and integrate with respect to $v$ from $0$ to $\infty$ to obtain
\begin{align*}
&\int_0^\infty\int_0^\infty u^{\frac{z}{2}}v^{z+2}e^{-v^2(u^2+1)}K_{\frac{z}{2}}\left(\frac{2n\pi\alpha}{u}\right)\ dudv\nonumber\\
&=-\frac{\sqrt{\pi}}{2}(n\pi\alpha)^{1+\frac{z}{2}}\Gamma\left(-\frac{z}{2}-\frac{1}{2}\right)\int_0^\infty v^{z+2}e^{-v^2}{}_0F_2\left(-;\frac{3}{2},\frac{3+z}{2};-(n\pi\alpha v)^2\right)\, dv\nonumber\\
&\quad+\frac{\sqrt{\pi}}{4}(n\pi\alpha)^{\frac{z}{2}}\Gamma\left(-\frac{z}{2}\right)\int_0^\infty v^{z+1}e^{-v^2}{}_0F_2\left(-;1+\frac{z}{2},\frac{1}{2};-(n\pi\alpha v)^2\right)\, dv\nonumber\\
&\quad+\frac{\sqrt{\pi}\Gamma(z)}{2^{z+1}(n\pi\alpha)^{\frac{z}{2}}}\int_0^\infty ve^{-v^2}{}_0F_2\left(-;1-\frac{z}{2},\frac{1-z}{2};-(n\pi\alpha v)^2\right)\, dv.
\end{align*}
Now employ change of variable $v=\sqrt{x}$ in all three integrals on the right-hand side and use \eqref{mt of pfq} repeatedly to deduce that for $n>0,\ \mathrm{Re}(z)>-1$ and $\mathrm{Re}(\alpha)>0$,
\begin{align}\label{df1}
&\int_0^\infty\int_0^\infty u^{\frac{z}{2}}v^{z+2}e^{-v^2(u^2+1)}K_{\frac{z}{2}}\left(\frac{2n\pi\alpha}{u}\right)\ dudv\nonumber\\
&=\frac{\pi^{\frac{3+z}{2}}}{8}(n\alpha)^{\frac{z}{2}}\sec\left(\frac{\pi z}{2}\right)\sin(2n\pi\alpha)-\frac{\pi^{\frac{3+z}{2}}}{8}(n\alpha)^{\frac{z}{2}}\mathrm{cosec}\left(\frac{\pi z}{2}\right)\cos(2n\pi\alpha)\nonumber\\
&\qquad\qquad+\frac{\sqrt{\pi}\Gamma(z)}{2^{z+2}(n\pi\alpha)^{\frac{z}{2}}}{}_1F_2\left(-;1-\frac{z}{2},\frac{1-z}{2};-(n\pi\alpha)^2\right)\nonumber\\
&=-\frac{\pi^{\frac{3+z}{2}}}{4}\frac{(n\alpha)^{\frac{z}{2}}}{\sin(\pi z)}\cos\left(\frac{\pi}{2}(z+4n\alpha)\right)+\frac{\sqrt{\pi}\Gamma(z)}{2^{z+2}(n\pi\alpha)^{\frac{z}{2}}}{}_1F_2\left(1;1-\frac{z}{2},\frac{1-z}{2};-(n\pi\alpha)^2\right)\nonumber\\
&=\frac{\sqrt{\pi}}{2\sqrt{2}}(\pi n\a)^{\frac{z+1}{2}}\G(z+2)S_{-z-\frac{3}{2}, \frac{1}{2}}(2\pi n\a),
\end{align}
where in the last step, we used \eqref{lommeldef1cap}.
Thus from \eqref{integral21w0} and \eqref{df1}, we obtain \eqref{integral21} for $0<\textup{Re}(z)<1$. To see that the result is actually valid for Re$(z)>-3$, it suffices to show that both sides of \eqref{integral21} are analytic in Re$(z)>-3$. This is shown using Theorem \ref{ldcttemme1}. Let
\begin{align*}
g(x,z):=x^{\frac{z}{2}}K_{\frac{z}{2}}(2\alpha x)\left({}_2F_1\left(1,\frac{z}{2};\frac{1}{2};-\frac{x^2}{\pi^2n^2}\right)-1\right).
\end{align*}
Clearly, for $\mathrm{Re}(z)>-3$, $g(x, z)$ is continuous in $x$ and $z$ and is also analytic in $z$ in this region for a fixed $x$. It only remains to show that $\int_{0}^{\infty}g(x, z)\, dx$ converges uniformly at both the limits in any compact subset of Re$(z)>-3$.\\

\textbf{Case I:} $x$ is small and positive.

For $\mathrm{Re}(z)>0$, \eqref{kzxasy0} and ${}_2F_1\left(1,\frac{z}{2};\frac{1}{2};-\frac{x^2}{\pi^2n^2}\right)-1=O_{n, z}(x^2)$ implies that
$g(x,z)=O(x^2)$ as $x\rightarrow0$. Similarly, for $\mathrm{Re}(z)<0$, from \eqref{kzxasy0},
 \begin{align*}
K_{\frac{z}{2}}(2\alpha x)=K_{-\frac{z}{2}}(2\alpha x)\sim\frac{1}{2}\Gamma\left(-\frac{z}{2}\right)(\alpha x)^{\frac{z}{2}},
\end{align*}
and hence $g(x,z)=O\left(x^{\mathrm{Re}(z)+2}\right)$. It is thus easy to see that the integral converges uniformly in any compact subset of $\mathrm{Re}(z)>-3, \textup{Re}(z)\neq 0$. Next, for $z=0$, \eqref{kzxasy0} implies $K_{0}(2\alpha x)\sim-\log(2\alpha x)$ as $x\rightarrow0$, so in this case too the integral is uniformly convergent for $\mathrm{Re}(z)>-3$. The remaining case for $\mathrm{Re}(z)=0, z\neq 0$, follows from a result of Dunster \cite[Equation (2.14)]{dunster}.\\

\textbf{Case II:} $x$ is large and positive.

As $x\rightarrow\infty$,
\begin{align}\label{kex}
K_{\frac{z}{2}}(2\alpha x)\sim\sqrt{\frac{\pi}{4\alpha x}}e^{-4\alpha x}.
\end{align}
The exponential decay ensures that the integral is uniformly convergent in any compact subset of $\mathrm{Re}(z)>-3$. 

Consequently, we see that the left-hand side of \eqref{integral21} is analytic for $\mathrm{Re}(z)>-3$. 
Now consider the right-hand side. We know that $S_{-z-\frac{3}{2}, \frac{1}{2}}(\xi)$ is analytic in Re$(z)>-3$ except for possible poles at $z=2m$ or $2m-1$, where $m\in\mathbb{N}\cup\{0\}$. However, these are removable singularities (see \cite[p.~347-349]{watson-1966a}). Also, the possible pole of the right-hand side at $z=-2$ is easily seen to be a removable singularity since both $\G(z+2)$ and $\G(z/2)$ have simple poles there. Hence the right-hand side of \eqref{integral21} is also analytic in Re$(z)>-3$.

By analytic continuation, we see that \eqref{integral21} holds for Re$(z)>-3$.
\end{proof}
\begin{remark}
One could possibly extend the results in Theorem \textup{\ref{integr}} to \textup{Re}$(z)>-3$. However, among other things, this would require studying the asymptotics of ${}_1K_{z,w}(x)$ for \textup{Re}$(z)=0$, which would generalize a result of Dunster \cite[Equation (2.14)]{dunster}. However, we do not venture into this since it is not needed for our immediate objectives.
\end{remark}
\section{Generalized modular relation with theta structure and the Riemann $\Xi$-function}
We first obtain a generalized modular-type transformation of the form $F(z, w, \a)=F(z, iw, \b)$ between two triple integrals and then derive Theorem \ref{genramhureq} from it. To the best of our knowledge, this modular-type transformation given in Theorem \ref{modular trans in intergal form} below is first-of-its-kind in the literature to involve triple integrals. It involves the function $\Omega(x, z)$ which was first studied in \cite[Section 6]{dixitmoll} and is defined by
\begin{equation}\label{omdef}
\Omega(x,z) := 2 \sum_{n=1}^{\infty} \sigma_{-z}(n) n^{z/2} 
\left( e^{\pi i z/4} K_{z}( 4 \pi e^{\pi i/4} \sqrt{nx} ) +
 e^{-\pi i z/4} K_{z}( 4 \pi e^{-\pi i/4} \sqrt{nx} ) \right).
\end{equation}
It plays a crucial role in deriving a short proof of the generalized Vorono\"{\dotlessi} summation formula for $\sigma_{-z}(n)$ \cite[Theorem 6.1]{bdrz}.

It satisfies two important properties, the first of which is, for $-1<\textup{Re}(z)<1$ and Re$(x)>0$ \cite[Proposition 6.1]{dixitmoll},
\begin{equation}\label{omegarep}
\Omega(x,z) = - \frac{\Gamma(z) \zeta(z)}{(2 \pi \sqrt{x})^{z}}  +
\frac{x^{z/2-1}}{2 \pi} \zeta(z) - 
\frac{x^{z/2}}{2} \zeta(z+1)  +
\frac{x^{z/2+1}}{\pi} \sum_{n=1}^{\infty} \frac{\sigma_{-z}(n)}{n^{2}+x^{2}}.
\end{equation}
The second property implies that the function $\Omega(x,z)-\frac{1}{2\pi}\zeta(z)x^{\frac{z}{2}-1}$ is self-reciprocal (up to a constant) in the Hankel kernel, namely \cite[Lemma 4.2]{koshkernel}, for $-1<\mathrm{Re}(z)<1,\ \mathrm{Re}(x)>0$, we have
\begin{align}\label{dix kosh}
\int_0^\infty J_z(4\pi\sqrt{xy})\left(\Omega(y,z)-\frac{1}{2\pi}\zeta(z)y^{\frac{z}{2}-1}\right)\ dy=\frac{1}{2\pi}\left(\Omega(x,z)-\frac{1}{2\pi}\zeta(z)x^{\frac{z}{2}-1}\right).
\end{align}

\subsection{A generalized modular relation involving $\zeta_w(s, a)$: Proof of Theorem \ref{genramhureq}}\label{gmrzwsa}
\hfill\\

We commence this section by first obtaining a generalized modular relation between two triple integrals.

\begin{theorem}\label{modular trans in intergal form}
Let $w\in\mathbb{C},\ -1<\mathrm{Re}(z)<1$. For $\a, \b>0$ such that $\alpha\beta=1$,
\begin{align}\label{treeple}
&\frac{\sqrt{\alpha}}{w^2}\int_0^\infty\int_0^\infty\int_0^\infty \frac{(uv)^{\frac{z}{2}}}{u}e^{-(u^2+v^2)}\sin(wv)\sinh(wu)\nonumber\\
&\qquad\qquad\qquad\times J_{\frac{z}{2}}\left(\frac{2\pi\alpha x v}{u}\right)\left(\Omega(x,z)-\frac{1}{2\pi}\zeta(z)x^{\frac{z}{2}-1}\right)\, dxdudv \nonumber\\
&=\frac{\sqrt{\beta}}{w^2}\int_0^\infty\int_0^\infty\int_0^\infty\frac{(uv)^{\frac{z}{2}}}{u}e^{-(u^2+v^2)}\sin(wu)\sinh(wv)\nonumber\\
&\qquad\qquad\qquad\times J_{\frac{z}{2}}\left(\frac{2\pi\beta x v}{u}\right)\left(\Omega(x,z)-\frac{1}{2\pi}\zeta(z)x^{\frac{z}{2}-1}\right)\, dxdudv.
\end{align}
\end{theorem}
\begin{proof}
We first show that the triple integrals in \eqref{treeple} converge. This is shown only for the one on the left-hand side; the convergence of the integral on the right-hand side can be similarly shown. To do this, we first analyze the behavior of the innermost integral. From \eqref{omegarep}, it is clear that as $x\to 0$,
\begin{align}\label{bound on omege for small x}
\Omega(x,z)-\frac{1}{2\pi}\zeta(z)x^{\frac{z}{2}-1}=O\left(x^{\frac{|\mathrm{Re}(z)|}{2}}\right).
\end{align}
Using \eqref{kex} in \eqref{omdef}, we see that as $x\to\infty$,
\begin{align}\label{bound on omege for large x}
\Omega(x,z)-\frac{1}{2\pi}\zeta(z)x^{\frac{z}{2}-1}&=O\left(e^{-2\pi\sqrt{2x}+\frac{\pi}{4}|\mathrm{Im}(z)|}\right)+O\left(x^{\frac{\mathrm{Re}(z)}{2}-1}\right)\nonumber\\
&=O\left(x^{\frac{\mathrm{Re}(z)}{2}-1}\right).
\end{align}
Also, from \cite[p. 360]{as}, for $\epsilon_1$ small enough and $M_1$ large enough,
\begin{align}\label{bound on J}
 J_\nu(x) \ll
  \begin{cases}
   x^{\mathrm{Re}(\nu)}, & \text{if\ $0\leq x< \epsilon_1$}, \\
   A_{\epsilon_1, M_1}, & \text{if}\ \epsilon_1\leq x\leq M_1,\\
   x^{-1/2}, & \text{if\ $x>M_1$},\\
  \end{cases}
\end{align}
where $A_{\epsilon_1, M_1}$ depends only on $\epsilon_1$ and $M_1$. Then from \eqref{bound on omege for small x}, \eqref{bound on omege for large x} and \eqref{bound on J}, for $\epsilon_2$ small enough and $M_2$ large enough,
\begin{align}\label{bound on integrand}
J_{\frac{z}{2}}\left(\frac{2\pi\alpha xv}{u}\right)\left(\Omega(x,z)-\frac{1}{2\pi}\zeta(z)x^{\frac{z}{2}-1}\right)&\ll
\begin{cases}
\left(\frac{v}{u}\right)^{\frac{\mathrm{Re}(z)}{2}} x^{\frac{\mathrm{Re}(z)}{2}+\frac{|\mathrm{Re}(z)|}{2}}, & \text{if}\ 0<x<\epsilon_2,\\
B_{\epsilon_2, M_2}, &\text{if}\ \epsilon_2\leq x\leq M_2,\\
\left(\frac{v}{u}\right)^{-\frac{1}{2}}x^{\frac{\mathrm{Re}(z)}{2}-\frac{3}{2}}, &\text{if}\ x>M_2.
\end{cases}
\end{align}
Since $-1<\textup{Re}(z)<1$, \eqref{bound on integrand} implies
\begin{align*}
\int_0^\infty J_{\frac{z}{2}}\left(\frac{2\pi\alpha xv}{u}\right)\left(\Omega(x,z)-\frac{1}{2\pi}\zeta(z)x^{\frac{z}{2}-1}\right)\ dx\ll_{\a}\left(v/u\right)^{\mathrm{Re}(z)/2}+C_{\epsilon_2, M_2}+(v/u)^{-1/2}.
\end{align*} 
Substituting the above bound for the innermost integral on the left-hand side of \eqref{treeple}, it is seen that the resulting double integral over $u$ and $v$ converges. This proves the convergence of the triple integral for $-1<\textup{Re}(z)<1$.

Now let $\mathfrak{I}(z,w, \a)$ denote the expression on the left-hand side of \eqref{treeple}. Employing \eqref{dix kosh} to represent $\Omega(x,z)-\frac{1}{2\pi}\zeta(z)x^{\frac{z}{2}-1}$ in $\mathfrak{I}(z,w, \a)$ as an integral, we observe that for $-1<\mathrm{Re}(z)<1$,
\begin{align}\label{mod10}
\mathfrak{I}(z,w, \a)&=\frac{2\pi\sqrt{\alpha}}{w^2}\int_0^\infty\int_0^\infty\int_0^\infty \frac{(uv)^{\frac{z}{2}}}{u}e^{-(u^2+v^2)}\sin(wv)\sinh(wu)\nonumber\\
&\quad\times J_{\frac{z}{2}}\left(\frac{2\pi\alpha x v}{u}\right)\int_0^\infty J_z(4\pi\sqrt{xy})\left(\Omega(y,z)-\frac{1}{2\pi}\zeta(z)y^{\frac{z}{2}-1}\right)\ dydxdudv \nonumber\\
&=\frac{2\pi\sqrt{\alpha}}{w^2}\int_0^\infty\int_0^\infty \frac{(uv)^{\frac{z}{2}}}{u}e^{-(u^2+v^2)}\sin(wv)\sinh(wu)\int_0^\infty \left(\Omega(y,z)-\frac{1}{2\pi}\zeta(z)y^{\frac{z}{2}-1}\right) \nonumber\\
&\quad\times \int_0^\infty J_{\frac{z}{2}}\left(\frac{2\pi x\alpha v}{u}\right)J_z(4\pi\sqrt{xy})\ dxdydudv,
\end{align}
where we interchanged the order of integration of the inner two integrals, permissible because of absolute convergence. From \cite[p.~215, Formula \textbf{2.12.34.1}]{prudII}, for $b,\ c>0,\ \mathrm{Re}(\nu)>-\frac{1}{2}$,
\begin{align*}
\int_0^\infty J_{2\nu}(b\sqrt{x})J_\nu(cx)\ dx=\frac{1}{c}J_\nu\left(\frac{b^2}{4c}\right).
\end{align*}
Letting $b=4\pi\sqrt{y},\ c=\frac{2\pi\alpha v}{u}$ and $\nu=\frac{z}{2}$ in the above integral evaluation, for $\mathrm{Re}(z)>-1, y>0, \a>0, v>0$ and $u>0$, we get
\begin{align}\label{bessel j product besse j}
\int_0^\infty J_{\frac{z}{2}}\left(\frac{2\pi\alpha vx}{u}\right)J_z(4\pi\sqrt{xy})\ dx=\frac{u}{2\pi\alpha v}J_{\frac{z}{2}}\left(\frac{2\pi uy}{v\alpha}\right).
\end{align}
From \eqref{mod10} and \eqref{bessel j product besse j} and the fact $\a\b=1$, we deduce, for $-1<\mathrm{Re}(z)<1$,
{\allowdisplaybreaks\begin{align*}
\mathfrak{I}(z,w, \a)&=\frac{\sqrt{\beta}}{w^2}\int_0^\infty\int_0^\infty\int_0^\infty \frac{(uv)^{\frac{z}{2}}}{v}e^{-(u^2+v^2)}\sin(wv)\sinh(wu)\nonumber\\
&\quad\times J_{\frac{z}{2}}\left(\frac{2\pi y\beta u}{ v}\right)\left(\Omega(y,z)-\frac{1}{2\pi}\zeta(z)y^{\frac{z}{2}-1}\right)\, dydudv \nonumber\\
&=\frac{\sqrt{\beta}}{w^2}\int_0^\infty\int_0^\infty\int_0^\infty \frac{(uv)^{\frac{z}{2}}}{u}e^{-(u^2+v^2)}\sin(wu)\sinh(wv)\nonumber\\
&\quad\times J_{\frac{z}{2}}\left(\frac{2\pi y\beta v}{u}\right)\left(\Omega(y,z)-\frac{1}{2\pi}\zeta(z)y^{\frac{z}{2}-1}\right)\, dydudv,
\end{align*}}
where in the last step, we swapped the variables $u$ and $v$ and then interchanged the order of integration, permissible again because of absolute convergence. 
This completes the proof.
\end{proof}
\begin{remark}
If we let $w\to 0$ in Theorem \textup{\ref{modular trans in intergal form}}, we obtain \cite[Equation (4.10)]{koshkernel} as a special case, namely, for $-1<\textup{Re}(z)<1$ and $\a\b=1$,
\begin{align*}
&\alpha^{(z+1)/2} 
\int_{0}^{\infty} e^{-2 \pi \alpha x} x^{z/2} 
\left( \Omega(x,z) - \frac{1}{2 \pi} \zeta(z) x^{z/2-1} \right)\, dx\nonumber \\
&=\beta^{(z+1)/2} 
\int_{0}^{\infty} e^{-2 \pi \beta x} x^{z/2} 
\left( \Omega(x,z) - \frac{1}{2 \pi} \zeta(z) x^{z/2-1} \right)\, dx,
\end{align*}
since the double integrals over $u$ and $v$ can be explicitly evaluated in closed-form using Lemma \textup{\ref{useful}}.
\end{remark}

We prove Theorem \ref{genramhureq} here.
\begin{proof}[Theorem \textup{\ref{genramhureq}}][]
It suffices to show that for any $\a>0$ and $-1<\textup{Re}(z)<1$,
\begin{align}\label{lk12}
&\frac{\sqrt{\alpha}}{w^2}\int_0^\infty\int_0^\infty\int_0^\infty \frac{(uv)^{\frac{z}{2}}}{u}e^{-(u^2+v^2)}\sin(wv)\sinh(wu)\nonumber\\
&\qquad\qquad\qquad\times J_{\frac{z}{2}}\left(\frac{2\pi\alpha x v}{u}\right)\left(\Omega(x,z)-\frac{1}{2\pi}\zeta(z)x^{\frac{z}{2}-1}\right)\, dxdudv \nonumber\\
&=\frac{\Gamma(z+1)}{2^{z+3}\pi^{\frac{z+1}{2}}}\alpha^{\frac{z+1}{2}}\left(\sum_{m=1}^\infty \varphi_w(z,m\alpha)-\frac{\zeta(z+1)A_w(z)}{2\alpha^{z+1}}-\frac{\zeta(z)A_{w}(-z)}{\alpha z}\right),
\end{align}
where $\varphi_{w}(z, m\a)$ and $A_w(z)$ are defined in \eqref{genramhureqeq} and \eqref{ab} respectively, for then, one can replace $\a$ by $\b$, and simultaneously $w$ by $iw$, in \eqref{lk12}, use the fact $\a\b=1$, and then appeal to Theorem \ref{modular trans in intergal form} to arrive at Theorem \ref{genramhureq}.

Observe from Theorem \ref{lse} that $\zeta(z+1, m\a+1)$ is well-defined and analytic in $-1<\textup{Re}(z)<1$ except for a simple pole at $z=0$. However, at $z=0$, the expression $-A_{iw}(-z)x^{-z}/z$ in \eqref{genramhureqeq} also has a simple pole. Hence we see that $\varphi_w(z,m\alpha)$ is well-defined for $-1<\textup{Re}(z)<1$. The convergence of the series $\sum_{m=1}^\infty \varphi_w(z,m\alpha)$ in $-1<\textup{Re}(z)<1$ (in fact, absolute and uniform convergence) follows from \eqref{hermiteeqn1}. 
We now prove \eqref{lk12}. Invoke \eqref{omegarep} to rephrase the integral on the left-hand side of \eqref{lk12} as
\begin{align}\label{lk12rep}
&\frac{\sqrt{\alpha}}{w^2}\int_0^\infty\int_0^\infty\int_0^\infty \frac{(uv)^{\frac{z}{2}}}{u}e^{-(u^2+v^2)}\sin(wv)\sinh(wu)\nonumber\\
&\qquad\qquad\qquad\times J_{\frac{z}{2}}\left(\frac{2\pi\alpha x v}{u}\right)\left(\Omega(x,z)-\frac{1}{2\pi}\zeta(z)x^{\frac{z}{2}-1}\right)\, dxdudv \nonumber\\
&=\frac{\sqrt{\alpha}}{w^2}\left(T_1(z, w, \a)+T_2(z, w, \a)\right),
\end{align}
where
{\allowdisplaybreaks\begin{align*}
T_1(z, w, \a)&:=\int_0^\infty\int_0^\infty\int_0^\infty \frac{(uv)^{\frac{z}{2}}}{u}e^{-(u^2+v^2)}\sin(wv)\sinh(wu)\nonumber\\
&\qquad\qquad\qquad\quad\times J_{\frac{z}{2}}\left(\frac{2\pi\alpha x v}{u}\right)\frac{x^{\frac{z}{2}+1}}{\pi}\sum_{n=1}^\infty\frac{\sigma_{-z}(n)}{x^2+n^2}\, dxdudv\nonumber\\
T_2(z, w, \a)&:=\int_0^\infty\int_0^\infty\int_0^\infty \frac{(uv)^{\frac{z}{2}}}{u}e^{-(u^2+v^2)}\sin(wv)\sinh(wu)\nonumber\\
&\qquad\qquad\qquad\quad\times J_{\frac{z}{2}}\left(\frac{2\pi\alpha x v}{u}\right)\left(-\frac{x^{z/2}}{2} \zeta(z+1)-\frac{\Gamma(z) \zeta(z)}{(2 \pi \sqrt{x})^{z}}\right)\, dxdudv
\end{align*}}
Note that
\begin{align*}
T_1(z, w, \a)&=\sum_{m=1}^\infty m^{-z}\int_0^\infty\int_0^\infty\int_0^\infty \frac{(uv)^{\frac{z}{2}}}{u}e^{-(u^2+v^2)}\sin(wv)\sinh(wu)J_{\frac{z}{2}}\left(\frac{2\pi\alpha vx}{u}\right)\nonumber\\
&\qquad\qquad\qquad\qquad\qquad\times \frac{x^{\frac{z}{2}+1}}{\pi}\sum_{k=1}^\infty\frac{1}{x^2+m^2k^2}\ dxdudv\nonumber\\
&=(2\pi)^{-1-\frac{z}{2}}\sum_{m=1}^\infty m^{-\frac{z}{2}}\int_0^\infty\int_0^\infty\int_0^\infty \frac{(uv)^{\frac{z}{2}}}{u}e^{-(u^2+v^2)}\sin(wv)\sinh(wu)\nonumber\\
&\qquad\qquad\qquad\qquad\times J_{\frac{z}{2}}\left(\frac{m\a vt}{u}\right)t^{z/2}\left(2t\sum_{k=1}^\infty \frac{1}{t^2+4k^2\pi^2}\right)\, dtdudv,
\end{align*}
where we interchanged the order of integration and summation in the first step and employed change of variable $x=mt/(2\pi)$ in the second. However, from \cite[p. 191]{con}, for $t\neq0$,
\begin{align*}
2t\sum_{k=1}^\infty \frac{1}{t^2+4k^2\pi^2}=\frac{1}{e^t-1}-\frac{1}{t}+\frac{1}{2}.
\end{align*}
Hence
\begin{align}\label{mod11}
T_1(z, w, \a)&=(2\pi)^{-1-\frac{z}{2}}\sum_{m=1}^\infty m^{-\frac{z}{2}}\int_0^\infty\int_0^\infty\int_0^\infty \frac{(uv)^{\frac{z}{2}}}{u}e^{-(u^2+v^2)}\sin(wv)\sinh(wu)\nonumber\\
&\quad\times t^{\frac{z}{2}}J_{\frac{z}{2}}\left(\frac{\alpha mvt}{u}\right)\left(\frac{1}{e^{t}-1}-\frac{1}{t}+\frac{1}{2}\right)\ dtdudv\nonumber\\
&=\frac{w^2\G(z+1)\a^{z/2}}{2^{z+3}\pi^{\frac{z+1}{2}}}\sum_{m=1}^{\infty}\varphi_w(z,m\a),
\end{align}
where in the last step we invoked Theorem \ref{befhermite} with $s=z+1$ and $a=m\a$, and then used the definition in \eqref{genramhureqeq}.

To simplify $T_2(z, w, \a)$, we first evaluate the inner integral over $x$. To do this, we first let $\mu=\frac{z}{2}=\nu$ and $b=\frac{2\pi\alpha v}{u}$ in \eqref{mt of bessel J} to observe for $-1<\textup{Re}(z)<1$, $\a>0$, $v>0$ and $u>0$,
\begin{align}\label{eno}
\int_0^\infty x^{\frac{z}{2}}J_{\frac{z}{2}}\left(\frac{2\pi\alpha vx}{u}\right)\ dx=\frac{2^{\frac{z}{2}}}{\sqrt{\pi}}\Gamma\left(\frac{1+z}{2}\right)\left(\frac{2\pi\alpha v}{u}\right)^{-1-\frac{z}{2}}.
\end{align}
Similarly,
\begin{align}\label{owt}
\int_0^\infty x^{-\frac{z}{2}}J_{\frac{z}{2}}\left(\frac{2\pi\alpha vx}{u}\right)\ dx=\frac{2^{-\frac{z}{2}}\sqrt{\pi}}{\Gamma\left(\frac{1+z}{2}\right)}\left(\frac{2\pi\alpha v}{u}\right)^{-1+\frac{z}{2}}.
\end{align}
Thus from \eqref{eno} and \eqref{owt},
\begin{align*}
&\int_{0}^{\infty}J_{\frac{z}{2}}\left(\frac{2\pi\alpha x v}{u}\right)\left(-\frac{x^{z/2}}{2} \zeta(z+1)-\frac{\Gamma(z) \zeta(z)}{(2 \pi \sqrt{x})^{z}}\right)\, dx\nonumber\\
&=-\frac{1}{4\pi^{(z+3)/2}}\G\left(\frac{z+1}{2}\right)\zeta(z+1)\left(\frac{\a v}{u}\right)^{-1-\frac{z}{2}}-\frac{1}{2^{z+1}\pi^{(z+1)/2}}\frac{\G(z)}{\G\left(\frac{z+1}{2}\right)}\zeta(z)\left(\frac{\a v}{u}\right)^{-1+\frac{z}{2}},
\end{align*}
so that
\begin{align}\label{mod999}
T_2(z, w, \a)&=-\frac{\a^{-1-\frac{z}{2}}}{4\pi^{(z+3)/2}}\G\left(\frac{z+1}{2}\right)\zeta(z+1)\int_0^\infty\int_0^\infty \frac{u^{z}}{v}e^{-(u^2+v^2)}\sin(wv)\sinh(wu)\, dudv \nonumber\\
&\quad-\frac{\a^{-1+\frac{z}{2}}}{2^{z+1}\pi^{(z+1)/2}}\frac{\G(z)}{\G\left(\frac{z+1}{2}\right)}\zeta(z)\int_0^\infty\int_0^\infty v^{z-1}e^{-(u^2+v^2)}\sin(wv)\sinh(wu)\ dudv\nonumber\\
&=-\frac{\a^{-1-\frac{z}{2}}}{2^{z+4}\pi^{\frac{z+1}{2}}}w^2\G(z+1)\zeta(z+1)A_w(z)-\frac{\a^{-1+\frac{z}{2}}}{2^{z+3}\pi^{\frac{z+1}{2}}}w^2\G(z)\zeta(z)A_w(-z),
\end{align}
where in the second step we invoked Lemma \ref{int}. From \eqref{lk12rep}, \eqref{mod11} and \eqref{mod999}, we obtain \eqref{lk12} upon simplification. This completes the proof.
\end{proof}

\subsection{Evaluation of an integral of the Riemann $\Xi$-function: Proof of Theorem \ref{xiintgenramhurthm}}\label{ee}

\begin{proof}[Theorem \textup{\ref{xiintgenramhurthm}}][]
We first prove the result for $0<\textup{Re}(z)<1$ and later extend it to $-1<\textup{Re}(z)<1$ by analytic continuation.

We know from Theorem \ref{reciFn}\footnote{Even though this result is true only for $-\frac{3}{4}<\textup{Re}(z)<\frac{3}{4}$, later while actually using the reciprocal functions in this result, we will replace $z$ by $z/2$, in which case the result then actually holds for $-\frac{3}{2}<\textup{Re}(z)<\frac{3}{2}$.} that $e^{-w^2/2}{}_1K_{z,iw}(2\alpha x)$ and $\beta{}_1K_{z,w}(2\beta x)$, where $\a\b=1$, are second Koshliakov transforms of each other. So let us choose $\phi(x, z, w)$ and $\psi(x, z, w)$ from \eqref{phii} and \eqref{psii} to be 
\begin{equation*}
\left(\phi(x, z, w), \psi(x, z, w)\right)=\left(e^{-w^2/2}{}_1K_{z,iw}(2\alpha x), \beta{}_1K_{z,w}(2\beta x)\right)
\end{equation*}
so that from \eqref{def}, \eqref{defZ1}, \eqref{defZ2} and \eqref{defTheta},
\begin{align}
\Theta(x,z,w)&=e^{-\frac{w^2}{2}}{}_1K_{z,iw}(2\alpha x)+\beta{}_1K_{z,w}(2\beta x),\label{Thetasp}\\
Z(s,z,w)&=Z_1(s,z,w)+Z_2(s,z,w),\label{zszwsp}
\end{align}
where, using \eqref{def},
\begin{align}
Z_1(s,z,w)&=\frac{e^{-\frac{w^2}{2}}}{2\alpha^s}{}_1F_1\left(\frac{1+s-z}{2},\frac{3}{2},\frac{w^2}{4}\right){}_1F_1\left(\frac{1+s+z}{2},\frac{3}{2},\frac{w^2}{4}\right)\nonumber\\
&=\frac{1}{2\alpha^s}{}_1F_1\left(1-\frac{s-z}{2},\frac{3}{2},-\frac{w^2}{4}\right){}_1F_1\left(1-\frac{s+z}{2},\frac{3}{2},-\frac{w^2}{4}\right),\label{z1szwsp}\\
Z_2(s,z,w)&=\frac{1}{2\beta^{s-1}}{}_1F_1\left(\frac{1+s-z}{2},\frac{3}{2},-\frac{w^2}{4}\right){}_1F_1\left(\frac{1+s+z}{2},\frac{3}{2},-\frac{w^2}{4}\right).\label{z2szwsp}
\end{align}
Note that in the second equality in \eqref{z1szwsp}, we used \eqref{kft} twice.

In order to use the above choices of $\phi, \psi, \Theta$ and $Z(s, z, w)$ in Theorem \ref{analogueReslutthm}, we need to first verify that $\phi$ and $\psi$ indeed belong to $\Diamond_{\eta, \omega}$. To that end, it suffices to show that ${}_1K_{z,w}(x)\in\Diamond_{\eta, \omega}$, which is done next. 

From Theorem \ref{integralRepr}, it is clear that ${}_1K_{z,w}(x)$ is an analytic function of $x$ for $|\arg(x)|<\frac{\pi}{4}$. Therefore part (i) of Definition \eqref{diamondc} is satisfied with $\omega=\frac{\pi}{4}$. Next, from Theorem \ref{asymSmall}, for $z\neq 0$,
\begin{align*}
{}_1K_{z,w}(x)=O\left(x^{1-\mathrm{Re}(z)}\right)
\end{align*} 
as $x\rightarrow 0$. Since Re$(z)<1$ in the hypotheses of the theorem, we have
\begin{align}\label{small x}
{}_1K_{z,w}(x)=O(x^{-\delta}),
\end{align}
for small values of $x$ and for every $\d>0$. Note from \eqref{kzwsmallii} that \eqref{small x} is valid for $z=0$ too. Similarly, the near-exponential decay of ${}_1K_{z, w}(x)$ for large values of $x$, as can be seen from Theorem \ref{asymptotcs large x}, ensures that
\begin{align}\label{large x}
{}_1K_{z,w}(x)=O\left(x^{-\eta-1-\left|\mathrm{Re}(z)\right|}\right)
\end{align} 
for $x$ large and $\eta>0$. Hence \eqref{small x} and \eqref{large x} imply that part (ii) of Definition \eqref{diamondc} is also satisfied whence ${}_1K_{z,w}(x)\in\Diamond_{\eta,\omega}$.

Next, from \eqref{Del} and \eqref{zszwsp} and the hypothesis $\a\b=1$, we have
\begin{equation}\label{zhalf}
Z\left(\frac{1+it}{2},\frac{z}{2}, w\right)=\frac{1}{2\sqrt{\a}}e^{-\frac{w^2}{4}}\Delta_2\left(\alpha,\frac{z}{2},w,\frac{1+it}{2}\right).
\end{equation}
Also, from \eqref{z1szwsp} and \eqref{z2szwsp},
\begin{align*}
Z_1\left(1\pm\frac{z}{2},\frac{z}{2},w\right)&=\frac{1}{2}\alpha^{-1\mp\frac{z}{2}}\frac{\sqrt{\pi}\mathrm{erf}\left(\frac{w}{2}\right)}{w}{}_1F_1\left(\frac{1\mp z}{2};\frac{3}{2};-\frac{w^2}{4}\right)\nonumber\\
Z_2\left(1\pm\frac{z}{2},\frac{z}{2},w\right)&=\frac{1}{2}\beta^{\mp\frac{z}{2}}\frac{\sqrt{\pi} e^{-\frac{w^2}{4}}\mathrm{erfi}\left(\frac{w}{2}\right)}{w}{}_1F_1\left(1\pm\frac{z}{2};\frac{3}{2};-\frac{w^2}{4}\right)
\end{align*}
so that from \eqref{defS} and \eqref{zszwsp} and upon using \eqref{kft} for each of ${}_1F_{1}$'s in the expressions involving $\a$, we find that
\begin{align}\label{szwexp}
S(z, w)&=\frac{\G(z+1)e^{\frac{-w^2}{4}}}{2^{z+1}\sqrt{\a}}\bigg\{\frac{1}{2}\a^{\frac{-(z+1)}{2}}\zeta(z+1)A_w(z)+\frac{1}{2}\b^{\frac{-(z+1)}{2}}\zeta(z+1)A_{iw}(z)\nonumber\\
&\qquad\qquad\qquad\qquad+\a^{\frac{z-1}{2}}\frac{\zeta(z)}{z}A_w(-z)+\b^{\frac{z-1}{2}}\frac{\zeta(z)}{z}A_{iw}(-z)\bigg\}\nonumber\\
&=\frac{\G(z+1)e^{\frac{-w^2}{4}}}{2^{z+1}\sqrt{\a}}\bigg\{\a^{\frac{(z+1)}{2}}\left(\frac{\zeta(z+1)A_w(z)}{2\alpha^{z+1}}+\frac{\zeta(z)A_w(-z)}{\alpha z}\right)\nonumber\\
&\qquad\qquad\qquad\qquad+\b^{\frac{(z+1)}{2}}\left(\frac{\zeta(z+1)A_{iw}(z)}{2\b^{z+1}}+\frac{\zeta(z)A_{iw}(-z)}{\b z}\right)\bigg\},
\end{align}
where $A_w(z)$ is defined in \eqref{ab}.

Thus using the aforementioned choices of $\phi, \psi$ as well as \eqref{Thetasp}, \eqref{zszwsp}, \eqref{zhalf} and \eqref{szwexp} in Theorem \ref{analogueReslutthm}, we arrive at
{\allowdisplaybreaks\begin{align}\label{inbet}
&\frac{e^{-\frac{w^2}{4}}\pi^{\frac{z-3}{2}}}{2\sqrt{\a}}\int_0^\infty \Gamma\left(\frac{z-1+it}{4}\right)\Gamma\left(\frac{z-1-it}{4}\right)\Xi\left(\frac{t+iz}{2}\right)\Xi\left(\frac{t-iz}{2}\right)\frac{\Delta_2\left(\alpha,\frac{z}{2},w,\frac{1+it}{2}\right)}{(z+1)^2+t^2}dt\nonumber\\
&=-\frac{1}{\pi}\Gamma\left(\tfrac{z}{2}\right)\sum_{n=1}^\infty\sigma_{-z}(n) \int_0^\infty \big\{e^{\frac{-w^2}{2}}{}_1K_{\frac{z}{2},iw}(2\alpha x)+\beta{}_1K_{\frac{z}{2},w}(2\beta x)\big\}\left({}_2F_1\left(1,\tfrac{z}{2};\tfrac{1}{2};-\tfrac{x^2}{\pi^2n^2}\right)-1\right)x^{\frac{z-2}{2}}dx\nonumber\\
&-\frac{\G(z+1)e^{\frac{-w^2}{4}}}{2^{z+1}\sqrt{\a}}\bigg\{\a^{\frac{(z+1)}{2}}\left(\frac{\zeta(z+1)A_w(z)}{2\alpha^{z+1}}+\frac{\zeta(z)A_w(-z)}{\alpha z}\right)+\b^{\frac{(z+1)}{2}}\left(\frac{\zeta(z+1)A_{iw}(z)}{2\b^{z+1}}+\frac{\zeta(z)A_{iw}(-z)}{\b z}\right)\bigg\}
\end{align}}
We will now be done if we can show that
\begin{align}\label{alphawala}
&-\frac{1}{\pi}\Gamma\left(\frac{z}{2}\right)\sum_{n=1}^\infty\sigma_{-z}(n) \int_0^\infty e^{\frac{-w^2}{2}}{}_1K_{\frac{z}{2},iw}(2\alpha x)\left({}_2F_1\left(1,\frac{z}{2};\frac{1}{2};-\frac{x^2}{\pi^2n^2}\right)-1\right)x^{\frac{z-2}{2}}\, dx\nonumber\\
&=\frac{\a^{\frac{z}{2}}\G(z+1)e^{-\frac{w^2}{4}}}{2^{z+1}}\sum_{m=1}^{\infty}\varphi_{w}(z,m\a),
\end{align}
where $\phi_{w}(z,m\a)$ is defined in \eqref{genramhureqeq}, for then, replacing $\a$ by $\b$ and simultaneously $w\to iw$ in \eqref{alphawala}, and then multiplying both sides of the resulting equation by $\b e^{-\frac{w^2}{2}}$ will lead, upon noticing ${}_1K_{\frac{z}{2},-w}(x)={}_1K_{\frac{z}{2},w}(x)$, to
\begin{align}\label{betawala}
&-\frac{1}{\pi}\Gamma\left(\frac{z}{2}\right)\sum_{n=1}^\infty\sigma_{-z}(n) \int_0^\infty \b{}_1K_{\frac{z}{2},w}(2\b x)\left({}_2F_1\left(1,\frac{z}{2};\frac{1}{2};-\frac{x^2}{\pi^2n^2}\right)-1\right)x^{\frac{z-2}{2}}\, dx\nonumber\\
&=\frac{\b^{\frac{z}{2}+1}\G(z+1)e^{-\frac{w^2}{4}}}{2^{z+1}}\sum_{m=1}^{\infty}\varphi_{iw}(z,m\b),
\end{align}
and then substituting \eqref{alphawala} and \eqref{betawala} in \eqref{inbet}, multiplying both sides of the resulting equation by $\frac{2^{z+1}e^{\frac{w^2}{4}}\sqrt{\a}}{\G(z+1)}$ and simplifying would imply
\begin{align*}
&\frac{2^{z}\pi^{\frac{z-3}{2}}}{\Gamma(z+1)}\int_0^\infty \Gamma\left(\frac{z-1+it}{4}\right)\Gamma\left(\frac{z-1-it}{4}\right)\Xi\left(\frac{t+iz}{2}\right)\Xi\left(\frac{t-iz}{2}\right)\frac{\Delta_2\left(\alpha,\frac{z}{2},w,\frac{1+it}{2}\right)}{(z+1)^2+t^2}dt\nonumber\\
&=\a^{\frac{z+1}{2}}\left(\sum_{m=1}^{\infty}\varphi_w(z,m\a)-\frac{\zeta(z+1)A_w(z)}{2\alpha^{z+1}}-\frac{\zeta(z)A_w(-z)}{\alpha z}\right)\nonumber\\
&\quad+\b^{\frac{z+1}{2}}\left(\sum_{m=1}^{\infty}\varphi_{iw}(z,m\b)-\frac{\zeta(z+1)A_{iw}(z)}{2\beta^{z+1}}-\frac{\zeta(z)A_{iw}(-z)}{\b z}\right).
\end{align*}
Then invoking Theorem \ref{genramhureq}, we would finally arrive at \eqref{xiintgenramhurthmeqn}.

Thus our aim now is to establish \eqref{alphawala}. To that end, note that employing Theorem \ref{integr} in the first step below, we have
\begin{align}\label{beforeWatson}
&-\frac{1}{\pi}\Gamma\left(\frac{z}{2}\right)\sum_{n=1}^\infty\sigma_{-z}(n) \int_0^\infty e^{\frac{-w^2}{2}}{}_1K_{\frac{z}{2},iw}(2\alpha x)\left({}_2F_1\left(1,\frac{z}{2};\frac{1}{2};-\frac{x^2}{\pi^2n^2}\right)-1\right)x^{\frac{z-2}{2}}\, dx\nonumber\\
&=\frac{4\pi^{\frac{z-1}{2}}e^{-\frac{w^2}{4}}}{w^2}\sum_{n=1}^\infty\sigma_{-z}(n)n^\frac{z}{2}\int_0^\infty\int_0^\infty u^{-1+\frac{z}{2}}v^ze^{-v^2(u^2+1)}K_{\frac{z}{2}}\left(\frac{2n\pi \alpha}{u}\right)\sin(wv)\sinh(wuv)\, dudv\nonumber\\
&=\frac{4\pi^{\frac{z-1}{2}}e^{-\frac{w^2}{4}}}{w^2}\sum_{m=1}^\infty m^{-\frac{z}{2}}\int_0^\infty\int_0^\infty u^{-1+\frac{z}{2}}v^ze^{-v^2(u^2+1)}\sin(wv)\sinh(wuv)\nonumber\\
&\qquad\qquad\qquad\qquad\times \sum_{n=1}^{\infty}n^\frac{z}{2}K_{\frac{z}{2}}\left(\frac{2mn\pi \alpha}{u}\right)\, dudv,
\end{align}
where the interchange of the order of summation and integration follows from absolute convergence. Now from \cite{watsonself}, we have for Re$(\xi)>0$ and Re$(\nu)>0$,
\begin{align}\label{watson lemma}
\sum_{n=1}^\infty n^\nu K_\nu(2\pi n\xi)&=\frac{1}{4}\sqrt{\pi}(\pi\xi)^{-\nu-1}\Gamma\left(\nu+\frac{1}{2}\right)-\frac{1}{4}(\pi\xi)^{-\nu}\Gamma(\nu)\nonumber\\
&\quad+\frac{\sqrt{\pi}}{2\xi}\left(\frac{\xi}{\pi}\right)^{\nu+1}\Gamma\left(\nu+\frac{1}{2}\right)\sum_{n=1}^\infty (n^2+\xi^2)^{-\nu-\frac{1}{2}}.
\end{align}
Let $\nu=z/2$ and $\xi=m\a/u$ in \eqref{watson lemma} and substitute the resulting right-hand side to simplify the inner sum on the extreme right-hand side of \eqref{beforeWatson} thereby obtaining
\begin{align}\label{afterwatson1}
&-\frac{1}{\pi}\Gamma\left(\frac{z}{2}\right)\sum_{n=1}^\infty \sigma_{-z}(n)\int_0^\infty e^{-\frac{w^2}{2}}{}_1K_{\frac{z}{2},iw}(2\alpha x)\left({}_2F_1\left(1,\frac{z}{2};\frac{1}{2};-\frac{x^2}{\pi^2n^2}\right)-1\right)x^{\frac{z-2}{2}}\, dx\nonumber\\
&=\sum_{m=1}^\infty m^{-\frac{z}{2}}\frac{4\pi^{\frac{z-1}{2}}e^{-\frac{w^2}{4}}}{w^2}\int_0^\infty\int_0^\infty u^{-1+\frac{z}{2}}v^ze^{-v^2(u^2+1)}\sin(wv)\sinh(wuv)\nonumber\\
&\quad\qquad\times\bigg\{\frac{\sqrt{\pi}}{4}\left(\frac{\pi m \alpha}{u}\right)^{-\frac{z}{2}-1}\Gamma\left(\frac{z+1}{2}\right)-\frac{1}{4}\left(\frac{\pi m\alpha}{u}\right)^{-\frac{z}{2}}\Gamma\left(\frac{z}{2}\right)\nonumber\\
&\qquad\qquad\quad+\frac{\sqrt{\pi}u}{2m\alpha}\left(\frac{m\alpha}{\pi u}\right)^{1+\frac{z}{2}}\Gamma\left(\frac{z+1}{2}\right)\sum_{n=1}^\infty\left(n^2+\frac{m^2\alpha^2}{u^2}\right)^{\frac{-1-z}{2}}\bigg\}\, dudv 
\end{align}
Our next job is to simplify the three double integrals occurring inside the sum on $m$. Observe that using Lemma \ref{int} and \eqref{dup} and recalling the definition of $A_w(z)$ given in \eqref{ab}, we have
\begin{align}\label{doub1}
&\frac{4\pi^{\frac{z-1}{2}}e^{-\frac{w^2}{4}}}{w^2}\int_0^\infty\int_0^\infty u^{-1+\frac{z}{2}}v^ze^{-v^2(u^2+1)}\sin(wv)\sinh(wuv)\frac{\sqrt{\pi}}{4}\left(\frac{\pi m \alpha}{u}\right)^{-\frac{z}{2}-1}\Gamma\left(\frac{z+1}{2}\right)\, dudv\nonumber\\
&=\frac{(m\a)^{\frac{z}{2}}\G(z+1)e^{-\frac{w^2}{4}}}{2^{z+1}}\left(\frac{1}{2}A_w(z)(m\a)^{-z-1}\right),
\end{align}
and
\begin{align}\label{doub2}
&\frac{4\pi^{\frac{z-1}{2}}e^{-\frac{w^2}{4}}}{w^2}\int_0^\infty\int_0^\infty u^{-1+\frac{z}{2}}v^ze^{-v^2(u^2+1)}\sin(wv)\sinh(wuv)\frac{1}{4}\left(\frac{\pi m\alpha}{u}\right)^{-\frac{z}{2}}\Gamma\left(\frac{z}{2}\right)\, dudv\nonumber\\
&=\frac{(m\a)^{\frac{z}{2}}\G(z+1)e^{-\frac{w^2}{4}}}{2^{z+1}}\left(\frac{A_{iw}(-z)}{z}(m\a)^{-z}\right).
\end{align}
Moreover, employing change of variable $u=t/v$ in the last double integral occurring in \eqref{afterwatson1}, we see that
{\allowdisplaybreaks\begin{align}\label{doub3}
&\frac{4\pi^{\frac{z-1}{2}}e^{-\frac{w^2}{4}}}{w^2}\int_0^\infty\int_0^\infty u^{-1+\frac{z}{2}}v^ze^{-v^2(u^2+1)}\sin(wv)\sinh(wuv)\frac{\sqrt{\pi}u}{2m\alpha}\left(\frac{m\alpha}{\pi u}\right)^{1+\frac{z}{2}}\Gamma\left(\frac{z+1}{2}\right)\nonumber\\
&\qquad\qquad\qquad\qquad\times\sum_{n=1}^\infty\left(n^2+\frac{m^2\alpha^2}{u^2}\right)^{\frac{-z-1}{2}}dudv\nonumber\\
&=\frac{2(m\a)^{\frac{z}{2}}e^{-\frac{w^2}{4}}}{w^2\pi}\G\left(\frac{z+1}{2}\right)\int_0^\infty\int_0^\infty t^{z}v^ze^{-(t^2+v^2)}\sin(wv)\sinh(wt)\sum_{n=1}^{\infty}(n^2t^2+m^2\a^2v^2)^{\frac{-z-1}{2}}dtdv\nonumber\\
&=\frac{2(m\a)^{\frac{z}{2}}e^{-\frac{w^2}{4}}}{w^2\pi}\G\left(\frac{z+1}{2}\right)\sum_{n=1}^{\infty}\int_0^\infty\int_0^\infty \frac{t^{z}v^ze^{-(t^2+v^2)}\sin(wv)\sinh(wt)\, dtdv}{(n^2t^2+(m\a)^2v^2)^{\frac{z+1}{2}}}\nonumber\\
&=\frac{(m\a)^{\frac{z}{2}}\G(z+1)e^{-\frac{w^2}{4}}}{2^{z+1}}\zeta_{w}(z+1, m\a+1),
\end{align}}
where $\zeta_{w}(s, a)$ is defined in \eqref{new zeta}. Note that the interchange of the order of summation and integration in the second step of the above calculation is valid because of absolute convergence. 

Thus from \eqref{afterwatson1}, \eqref{doub1}, \eqref{doub2} and \eqref{doub3}, we are led to \eqref{alphawala}. This completes the proof of Theorem \ref{xiintgenramhurthm} for $0<\textup{Re}(z)<1$. However, both sides of \eqref{xiintgenramhurthmeqn} are analytic in $-1<\textup{Re}(z)<1$. (Note that $z=0$ is the removable singularity of the right-hand side as can be seen from Theorem \ref{lse}). Hence we conclude that the result holds for $-1<\textup{Re}(z)<1$.
\end{proof}

\begin{corollary}
Equation \eqref{hurwitzeqn} holds for $-1<\textup{Re}(z)<1$.
\end{corollary}
\begin{proof}
Let $w\to 0$ in \eqref{genramhureqeqn} and \eqref{xiintgenramhurthmeqn}. To show that passing the limit through the summation is justified, it suffices to show that $\sum_{n=1}^{\infty}\varphi_{w}(z, m\a)$ is uniformly convergent on any compact subset of the $w$-complex plane containing the origin $w=0$, say, $|w|\leq M$ for $M>0$. To that end, first note that from \eqref{hermiteeqn},
\begin{equation*}
\varphi_w(z, m\a)=\frac{2^{z+3}(m\a)^{-z}}{w^2\Gamma\left(\frac{1-z}{2}\right)\Gamma(z+1)}\int_0^\infty\int_0^\infty\int_{1}^\infty \frac{u^{z-1}e^{-(u^2+v^2)}\sin(wv)\sinh(wu)}{(e^{\frac{2\pi m\a vt}{u}}-1)(t^2-1)^{\frac{z+1}{2}}}\ dtdudv.
\end{equation*}
Using \eqref{bound for a}, we see that
\begin{equation*}
\left|\int_0^\infty \frac{e^{-v^2}\sin(wv)}{e^{\frac{2\pi m\a vt}{u}}-1}\ dv\right|\ll\frac{|w|u^2}{m^2\a^2t^2}
\end{equation*}
as $m\to\infty$. Thus,
\begin{align*}
\left|\varphi_w(z, m\a)\right|&\ll_{z, \a}\frac{1}{m^{\textup{Re}(z)+2}}\int_0^\infty\int_{1}^\infty \frac{u^{\textup{Re}(z)+1}e^{-u^2}}{t^2(t^2-1)^{\frac{\textup{Re}(z)+1}{2}}}\left|\frac{\sinh(wu)}{w}\right|\ dtdu\nonumber\\
&\ll_{z, \a}\frac{1}{m^{\textup{Re}(z)+2}}\int_0^\infty\int_{1}^\infty \frac{u^{\textup{Re}(z)+1}e^{-u^2}}{t^2(t^2-1)^{\frac{\textup{Re}(z)+1}{2}}}\frac{\sinh(Mu)}{M}\ dtdu\nonumber\\
&\ll_{z, \a, M}\frac{1}{m^{\textup{Re}(z)+2}},
\end{align*}
since the double integral over $t$ and $u$ converges for $-1<\textup{Re}(z)<1$. This implies the uniform convergence of the series $\sum_{n=1}^{\infty}\varphi_{w}(z, m\a)$ on $|w|\leq M$ whence we can pass the limit through the summation. 

Also, from \cite[p.~318]{mos}, for $|\kappa|$ large,
\begin{equation*}
M_{\kappa, \mu}(z)=\Gamma(1+2\mu)\pi^{-1/2}z^{1/4}\kappa^{-\mu-1/4}\cos\left(2\sqrt{x}\sqrt{z}-\frac{\pi}{4}(1+4\mu)\right)+O\left(|\kappa|^{-\mu-3/4}\right).
\end{equation*}
Hence from the above equation and \eqref{m1f1}, we see that the confluent hypergeometric functions occurring in $\Delta_2\left(\alpha,\frac{z}{2},w,\frac{1+it}{2}\right)$ in the integral in \eqref{xiintgenramhurthm} satisfy, as $|t|\to\infty$ and $-1<\textup{Re}(z)<1$,
\begin{equation}\label{bd1f1}
{}_1F_{1}\left(\frac{3\mp it\pm z}{4};\frac{3}{2};-\frac{w^2}{4}\right)\sim\frac{1}{2}\frac{e^{-w^2/8}}{\sqrt{-\frac{w^2}{4}\left(\frac{\pm it\pm z}{4}\right)}}\sin\left(\sqrt{-w^2\left(\frac{\pm it\pm z}{4}\right)}\right).
\end{equation}
Thus, from Stirling's formula \eqref{strivert}, the fact that $\Xi(t)=O\left(t^{A}e^{-\frac{\pi}{4}t}\right)$ for $A>0$, and the asymptotic in \eqref{bd1f1},
we see that the limit $w\to 0$ can be passed through the integral in \eqref{xiintgenramhurthmeqn}. The result now follows easily from \eqref{genramhureqeqn}, \eqref{xiintgenramhurthmeqn}, Theorem \ref{hurwitzsc} and \eqref{hurrel}.
\end{proof}
\section{Concluding Remarks}\label{cr}
This project commenced with our desire to obtain a generalization of \eqref{hurwitzeqn} in which the modular relation will be of the form $F(z, w, \a)=F(z, iw, \b)$ with $\a\b=1$, thus incorporating the theta structure. In the course of evaluating the integral in \eqref{xiintgenramhurthmeqn} occurring in the generalization of the integral in \eqref{hurwitzeqn}, we stumbled upon a surprising new generalization of the Hurwitz zeta function $\zeta(s, a)$, that is, \eqref{new zeta}. It was then shown that this generalized Hurwitz zeta function $\zeta_w(s, a)$ satisfies a beautiful theory analogous to that of $\zeta(s, a)$ and $\zeta(s)$, and has a nice additional property not present in $\zeta(s, a)$ when Re$(s)>1$, namely, the symmetry along the line Re$(a)=1$ as can be seen from \eqref{zwfea}. This justifies the usefulness of working with the integrals involving the Riemann $\Xi$-function of the type we have studied here. They have also led us to several new and interesting special functions $K_{z, w}(x), {}_1K_{z,w}(x)$ and $\zeta_{w}(s, a)$ defined in \eqref{kzw}, \eqref{def} and \eqref{new zeta} respectively. However, it would be interesting and worthwhile seeing what the deformation parameter $w$ in $\zeta_w(s, a)$ corresponds to, in the theory of Jacobi forms or related branches of number theory.

We would also like to emphasize that Theorem \ref{xiintgenramhurthm} is but one special case of Theorem \ref{analogueReslutthm}, where we choose the pair of functions reciprocal in the kernel $\sin(\pi z)J_{2z}(4\sqrt{xt})-\cos(\pi z)L_{2z}(4\sqrt{xt})$ to be $\left(e^{-w^2/2}{}_1K_{z,iw}(2\alpha x), \beta {}_1K_{z,w}(2\beta x)\right)$ with $\a\b=1$, as done in Theorem \ref{reciFn}. As such, one can start with other pairs of functions reciprocal in this kernel and obtain further new identities.

In his performance analysis of power line communication systems, Abou-Rjeily \cite[p.~407]{abou} writes,

 \textit{``\dots the CF [Characteristic Function] of $V_{EGC}$ [the noise term] is proportional to the square of a confluent hypergeometric function which, evidently, renders all subsequent manipulations very hard to accomplish. Consequently, additional approximations are mandatory".}

Even though the confluent hypergeometric function in our setting is different from Abou-Rjeily's, the fact that from Theorem \ref{lse}, we have
\begin{equation*}
\lim_{s\to 1}(s-1)\zeta_w(s, a)=e^{\frac{w^2}{4}}{}_1F_{1}^{2}\left(1;\frac{3}{2};-\frac{w^2}{4}\right),
\end{equation*}
hints at possible applications of our new zeta function $\zeta_w(s, a)$ in Physics and Engineering where one might not need to resort to approximations but rather use exact representations for $\zeta_w(s, a)$ that are established in this paper for further analysis.\\

\noindent
\textbf{Problems for further study:}\\

We list below potential problems that are worthy of investigation.\\

1. Can $\zeta_w(s, a)$ be analytically continued (in the variable $s$) to $\mathbb{C}\backslash\{1\}$? \\

2. One of the fundamental formulas in the theory of Hurwitz zeta function $\zeta(s, a)$ is Hurwitz's formula \cite[p.~37, Equation (2.17.3)]{titch} given for $0<a\leq 1$ and $\textup{Re}(s)<0$ by\footnote{This result also holds for \textup{Re}$(s)<1$ if $0<a<1$. See \cite[p.~257, Theorem 12.6]{apostol}.}
\begin{equation}\label{fehur}
\zeta(s, a)=\frac{2\Gamma(1-s)}{(2\pi)^{1-s}}\left\{\sin\left(\frac{1}{2}\pi s\right)\sum_{n=1}^\infty \frac{\cos(2\pi na)}{n^{1-s}}+\cos\left(\frac{1}{2}\pi s\right)\sum_{n=1}^\infty \frac{\sin(2\pi na)}{n^{1-s}}\right\}.
\end{equation}
Several proofs of this result exist in the literature. We refer the interested reader to \cite{ktty} for a survey of most of the existing proofs (see also \cite{kanetsuk}). 

Find the generalization of Hurwitz's formula in the setting of $\zeta_w(s, a)$.\\ 

3. Find the role of the deformation parameter $w$ occurring in $\zeta_w(s, a)$, ${}_1K_{z,w}(x)$ and $K_{z, w}(x)$ in Mathematics and Physics.\\

4. Corollary \ref{integral21cor} contains a new integral representation for $n^{z+1/2}S_{-z-\frac{3}{2},\frac{1}{2}}(2\pi na)$ (modulo a factor involving $z$ and $\a$), where $S_{-z-\frac{3}{2},\frac{1}{2}}(2\pi na)$ is a special case of the first Lommel function. Theorem \ref{integr}, which is a generalization of Corollary \ref{integral21cor}, therefore contains a generalized Lommel function in the form of an integral (on the left-hand side of \eqref{integreqn}). This latter integral is the summand of the infinite series in \eqref{alphawala}. The special case $w=0$ of such infinite series are similar to those considered by Lewis and Zagier in \cite[Equation (2.11)]{lewzag}, and which represent the period functions of Maass wave forms with spectral parameter $(z+1)/2$. Hence it would be of interest to see what do the generalized series in \eqref{alphawala} represent in the context of Maass wave forms. Also, in view of \eqref{alphawala}, it would be worthwhile seeing if the generalized modular relation in Theorem \ref{genramhureq} has an application in Maass wave forms.\\

5. We know of no other way to prove the generalized modular-type transformation in Theorem \ref{genramhureq} and the integral evaluation in Theorem \ref{xiintgenramhurthm}. It would be interesting to see if other proofs of these results could be constructed, especially since the special case \eqref{hurwitzeqn} has several proofs \cite{dixit}, \cite{transf}.\\

6. Find a relation between $\zeta_w(s, a+1)$ and $\zeta_w(s,a)$ which would then generalize the well-known result \eqref{hurrel}.\\

7. Since Hermite's formula \eqref{hermiteor} for $\zeta(s, a)$ is valid for Re$(a)>0$, it would nice to analytically continue \eqref{hermiteeqn} \emph{as a function of $a$} in Re$(a)>0$.\\

8. It would be nice to study special values of $\zeta_w(s, a)$. For example, note that from \eqref{hermiteeqn}, we have
\begin{equation*}
\zeta_w(0,a+1)=\frac{\sqrt{\pi}}{w}\textup{erfi}\left(\frac{w}{2}\right)\left\{-\frac{1}{2}\frac{\sqrt{\pi}}{w}\textup{erf}\left(\frac{w}{2}\right)-ae^{-\frac{w^2}{4}}\right\},
\end{equation*}
which, for $w=0$, gives the well-known special value of the Hurwitz zeta function, namely, \cite[p.~264, Theorem 12.13]{apostol}
\begin{equation*}
\zeta(0,a+1)=-\frac{1}{2}-a.
\end{equation*}

9. In Theorem \ref{lse}, we introduced a new generalization of the digamma function $\psi(a)$, namely, $\psi_w(a)$ defined in \eqref{new psi function}. Traditionally, however, $\psi(a)$ is defined as the logarithmic derivative of the gamma function $\G(a)$. In view of this, it would be interesting to find the generalized gamma function $\G_w(a)$ related to $\psi_w(a)$ in a similar way as $\G(a)$ is to $\psi(a)$. Perhaps, a generalization of Lerch's formula \cite[p.~271]{ww}, namely
\begin{equation*}
\left.\frac{d}{ds}\zeta(s, a)\right|_{s=0}=\log\G(a)-\frac{1}{2}\log(2\pi),
\end{equation*}
in the context of $\zeta_w(s, a)$ would be a good place to start.\\

10. While developing the Theory of ${}_1K_{z,w}(x)$ in Section \ref{1kzw}, we found a Basset-type integral representation for ${}_1K_{0,w}(x)$ (see Theorem \ref{bessetTypeRepr}). It would be interesting to find a Basset-type integral representation for ${}_1K_{z,w}(x)$ in general. This would then generalize \eqref{Besset2}.\\

11. From \cite[Equation (1.19)]{kumar}, we know that the hypergeometric function ${}_2F_{2}\left(1,1;2,\frac{3}{2};-\frac{w^2}{4}\right)$ arises in a representation for the cumulative distribution function of the Voigt line profile which is an important model in molecular spectroscopy and radiative transfer, namely,
\begin{equation*}
F(x;\sigma,\beta)=\textup{Re}\left[\frac{1}{2}+\frac{\textup{erf}(w)}{2}+\frac{iw^2}{\pi}{}_2F_{2}\left(1,1;2,\frac{3}{2};-w^2\right)\right].
\end{equation*}
Also, the generalized modified Bessel function $K_{z,w}(x)$ defined in \eqref{kzw} has a nice connection with the Voigt line profile as shown in \cite[Theorem 3]{kumar}. It would be interesting to see if the functions ${}_2F_{2}\left(1,1;2,\frac{5}{2};-w^2\right)$ arising in \eqref{kzwsmallii} also has a connection with a some model in physics.\\

We conclude this paper with a quote by Andrews and Berndt \cite[p.~4]{geabcblost4} since it harmonizes with the content of this paper. They say, \emph{`One of the hallmarks of Ramanujan's mathematics is that it frequently generates further interesting mathematics\dots'}. Also, as remarked by Dyson \cite[p.~263]{dyson}, Ramanujan indeed left so much in his garden for other people to discover! We hope that the results obtained here stimulate further studies on $\zeta_w(s, a)$.

\begin{center}
\textbf{Acknowledgements}
\end{center}
The authors sincerely thank Nico M. Temme for his help. They also thank Richard B. Paris, Don Zagier, M. Lawrence Glasser, Robert Maier, Alexandru Zaharescu and Alexandre Kisselev for interesting discussions. The first author's research is partially supported by the the SERB-DST grant ECR/2015/000070 and partially by the SERB MATRICS grant MTR/2018/000251. He sincerely thanks SERB for the support.

%
%
%

\end{document}